\documentclass[a4paper,twoside,11pt]{article}
\usepackage{amssymb,amsfonts,amsmath,amsbsy,theorem,amscd}
\usepackage[
left = 2.2cm,
right = 2.2cm,
top = 2.8 cm,
bottom = 3 cm
]{geometry}
\usepackage{epsfig,psfrag,graphicx,xcolor}
\usepackage{stmaryrd}
\usepackage{extarrows}
\usepackage{hyperref}
\hypersetup{colorlinks=true,citecolor=blue,linkcolor=blue,urlcolor=black}
\usepackage{csquotes}
\usepackage{mathtools} 
\usepackage{mathrsfs} % for mathscr typeface
\usepackage{cite} % for citations
\usepackage{color}

\definecolor{CMUrot}{RGB}{128,18,18}
\definecolor{Gold}{RGB}{238,180,34}
\allowdisplaybreaks[4]

\newcommand{\ol}[1]{\overline{#1}}

\numberwithin{equation}{section}

\newcommand{\R}{\ensuremath{\mathbb{R}}}

\newcommand{\Om}{\ensuremath{\Omega}}

\newcommand{\N}{\ensuremath{\mathbb{N}}}

\newcommand{\dist}{\operatorname{dist}}

\newcommand{\sd}{{\,\rm d}}
\newcommand{\supp}{\operatorname{supp}}
\newcommand{\eps}{\ensuremath{\varepsilon}}
\newcommand{\weight}[1]{\langle #1\rangle}

\newcommand{\Div}{\operatorname{div}}

\newcommand{\T}{\ensuremath{\mathbb{T}}}

\newcommand{\no}{\mathbf{n}}

\newcommand{\tn}[1]{\mathbf{#1}}
\def\nn{\mathbf{n}}

\newcommand{\ve}{\mathbf{v}}
\newcommand{\we}{\mathbf{w}}
\newcommand{\ue}{\mathbf{u}}

\newcommand{\btau}{{\boldsymbol{\tau}}}
\newtheorem{thm}{Theorem}[section]
\newtheorem{cor}[thm]{Corollary}
\newtheorem{lem}[thm]{Lemma}
\newtheorem{defn}[thm]{Definition}
\newtheorem{theorem}[thm]{Theorem}
\newtheorem{prop}[thm]{Proposition}

\newtheorem{claim*}{Claim}
\newtheorem{assumption}[thm]{Assumption}
\theorembodyfont{\rmfamily}
\newtheorem{rem}[thm]{Remark}

\newenvironment{proof*}[1]{{\bf Proof
#1:}}{\hspace*{\fill}\rule{1.2ex}{1.2ex}\\ }
\newenvironment{proof}{{\bf
Proof:\,}}{\hspace*{\fill}\rule{1.2ex}{1.2ex}\\ }

\newcommand{\p}{\partial}
\newcommand{\G}{\Gamma}
\def\({\left(}
\def\){\right)}

\newcommand{\tc}{\hat{c}}
\newcommand{\tv}{\hat{\ve}}
\newcommand{\tp}{\hat{p}}
\newcommand{\tr}{\hat{r}}

\newcommand{\order}{N}
\newcommand{\zg}{\zeta_\Gamma}

\newcommand{\jump}[1]{\left\llbracket #1 \right\rrbracket}

\newcommand{\abs}[1]{\left\vert #1 \right \vert}
\newcommand{\absm}[1]{\vert #1 \vert}
\newcommand{\norm}[1]{\left\Vert #1 \right \Vert}
\newcommand{\normm}[1]{\Vert #1 \Vert}

\newcommand{\prho}{\partial_{\rho}}

\renewcommand{\d}{\mathrm{d}}
\newcommand{\dx}{\,\d x}
\newcommand{\dt}{\,\d t}
\newcommand{\dr}{\,\d r}
\newcommand{\ds}{\,\d s}
\newcommand{\dsigma}{\,\d \sigma}

\newcommand{\ddt}{\frac{\d}{\d t}}
\newcommand{\drho}{\,\d \rho}

\newcommand{\ptial}[1]{ \partial_{#1} }
\newcommand{\pt}{\ptial{t}}
\newcommand{\pr}{\partial_r}

\newcommand{\onehalf}{\frac{1}{2}}

\newcommand{\nablax}{\nabla_x}
\newcommand{\Divx}{\Div_x}
\newcommand{\Deltax}{\Delta_x}

\newcommand{\bb}[1]{\mathbb{#1}}
\newcommand{\bbr}{\bb{R}}

\newcommand{\bbn}{\bb{N}}

\newcommand{\bu}{\mathbf{u}}
\newcommand{\bv}{\mathbf{v}}
\newcommand{\bw}{\mathbf{w}}
\newcommand{\bn}{\mathbf{n}}

\newcommand{\br}{\mathbf{r}}

\newcommand{\bR}{\mathbf{R}}
\newcommand{\bphi}{\boldsymbol{\varphi}}
\newcommand{\hc}{\hat{c}}
\newcommand{\hv}{\hat{\bv}}
\newcommand{\hp}{\hat{p}}

\newcommand{\tw}{\widetilde{\bw}}

\newcommand{\cL}{\mathcal{L}}

\newcommand{\cN}{\mathcal{N}}

\newcommand{\oeps}{\omega_\eps}
\newcommand{\toeps}{\tilde{\omega}_\eps}
\newcommand{\wW}{H_\eps} %%% \oeps-weighted L^2-Sobolev space 
\newcommand{\Veps}{V_t^\eps} %%% Scaled H^1-space for \ol u in neighborhood of interface

\begin{document}
\begin{titlepage}
% Old Version  
    \title{\bfseries Higher Order Convergence for the Sharp Interface Limit of 3D Navier--Stokes/Allen--Cahn Systems}
\author{Helmut Abels, Mingwen Fei, Yadong Liu, and Maximilian Moser}
\end{titlepage}

\maketitle
\abstract{We show convergence of solutions to a Navier--Stokes/Allen--Cahn system  as the interfacial thickness $\varepsilon>0$ tends to zero for well-prepared initial data as long as the limit system possesses a sufficiently smooth solution. The limit system consists of a two-phase Navier--Stokes system separated by a sharp interface in the presence of surface tension coupled to a convective mean curvature flow equation. In comparison to previous results we obtain improved convergence estimates for higher-order norms. These enable us to prove convergence in the case of three space dimensions and non-constant viscosity, which was unknown before. The convergence result relies crucially on uniform higher-order estimates for the associated linearized Navier--Stokes/Allen--Cahn system in suitably weighted $L^2$-Sobolev spaces. Here a novel problem-adapted weight proportional to the sum of $\varepsilon$ and the distance to the sharp interface of the limit, which gives improved and sharp estimates, is an important new ingredient. This approach can potentially be adapted to other sharp interface limits as well. }

\medskip

{\small\noindent
%% Capillarity (surface tension) for incompressible viscous fluids 

\section{Introduction and Main Result}\label{sec:Introduction}
The study of two-phase flows is a challenging and important problem in fluid dynamics and partial differential equations with many applications in the sciences and in engineering fields. Diffuse interface models, also called phase field models, are widely adopted mathematical models for two-phase flows in theoretical analysis and numerical computation,  see e.g.~\cite{AMW1998}. The approach recognizes micro-scale mixing of the macroscopically immiscible fluids and hence treats the interface as a transition layer with finite (small) width  $\eps> 0$, and a suitable (the so-called phase field, vector-valued in some cases) $c_\eps\colon \Omega \subseteq \mathbb{R}^d \to \mathbb{R} $ is introduced, which accounts for mixing in the interface region that real fluids always display. In the absence of a fluid flow the corresponding energy is the Cahn--Hilliard energy (of Ginzburg--Landau type)
\begin{align*}
	E(c_\eps) 
	= \int_\Omega \left(\frac{\eps}{2} \abs{\nabla c_\eps}^2 + \frac{1}{\eps} f(c_\eps)\right) \dx,
\end{align*}
where $f$ is typically a double well potential (e.g.\ $f(c_\eps) = \tfrac{1}{4} (c_\eps^2 - 1)^2$). Taking the $L^2$- or $H^{-1}$-gradient flow, one obtains the classical Allen--Cahn equation \cite{AC1979} or the Cahn--Hilliard equation \cite{CahnHilliard}, respectively.
As $\eps \to 0$, the domain $\Omega$ will be separated into two regions $\Omega^\pm(t)$, where $c_\eps \to \pm 1$. Moreover, the interface $\Gamma_t$ between these two regions evolves according to the mean curvature flow or the Mullins--Sekerka flow, respectively. 
In classical sharp interface models the interface is treated as a sufficiently smooth surface and the equations of motion hold in each phase and are supplemented by the boundary conditions at the sharp interface, which leads to a free boundary problem. In comparison the phase field models can describe singularities of interfaces due to topological changes, which are fundamental physical processes, such as pinch-off and reconnection. Hence they have many advantages in numerical simulations of the interfacial motion. A widely concerned important problem is the consistency of sharp and phase field models. Mathematically, it is the investigation of the convergence of solutions of the phase field model to solutions of its corresponding sharp interface model as the thickness  $\eps$ of the transition layer tends to zero. Such a problem is referred to as the sharp interface limit of the phase field model.

The goal of this contribution is two-fold. The first is to study the singular limit $\eps\to 0$ of the following Navier--Stokes/Allen--Cahn system:
\begin{subequations}\label{eq:NSAC}
	\begin{alignat}{2}\label{eq:NSAC1}
		\partial_t \bv_\eps +\bv_\eps\cdot \nabla \bv_\eps-\Div(2\nu(c_\eps)D\bv_\eps)  +\nabla p_\eps & = -\eps \Div (\nabla c_\eps \otimes \nabla c_\eps)&\quad & \text{in}\ \Omega\times(0,T_0),\\\label{eq:NSAC2}
		\Div \bv_\eps& = 0&\quad & \text{in}\ \Omega\times(0,T_0),\\\label{eq:NSAC3}
		\partial_t c_\eps +\bv_\eps\cdot \nabla c_\eps & =m_0\left(\Delta c_\eps - \tfrac1{\eps^2} f'(c_\eps)\right)&\quad & \text{in}\ \Omega\times(0,T_0),\\
		\label{eq:NSAC4}
		(\bv_\eps,c_\eps)|_{\partial\Omega}&= (0,-1)&\quad & \text{on }\partial\Omega\times (0,T_0),\\
		\label{eq:NSAC5}
		(\bv_\eps,c_\eps) |_{t=0}& = (\bv_{0,\eps},c_{0,\eps})&\quad& \text{in }\Omega,
	\end{alignat}
\end{subequations}
where $\Omega\subseteq \R^d$, $d=2,3$, is a bounded smooth domain. Here $\ve_\eps\colon \Omega\times (0,T_0)\to \R^d$, $p_\eps\colon \Omega\times (0,T_0)\to \R$, and $c_\eps\colon \Omega\times (0,T_0)\to \R$ are the (mean) velocity, the pressure, and the concentration difference, respectively, of the fluid mixture. Moreover, $\nu\colon \R\to (0,\infty)$ is the viscosity of the mixture, $f\colon \R\to \R$ is a (homogeneous) free energy density of double well shape, specified below, $\eps>0$ is a parameter proportional to the interfacial thickness of the diffuse interface between the two fluids, and $m_0>0$ is the mobility coefficient, chosen to be independent of $\eps$. In this model the density difference of the fluids is neglected and the densities are assumed to be the same (and put to one for simplicity). The Navier--Stokes/Allen--Cahn model \eqref{eq:NSAC} was proposed by Liu and Shen in \cite{LiuShen} as an alternative approximation of a classical sharp interface model for a two-phase flow of viscous, incompressible, Newtonian fluids. Later a more general model was derived by Jiang, Li, and Liu \cite{TwoPhaseVariableDensityJiangEtAl} for fluids with different densities and phase transitions. Moreover, existence of weak solutions, strong solutions and long-time behavior of the model were studied in \cite{TwoPhaseVariableDensityJiangEtAl} as well. The long-time behavior of solutions to the model \eqref{eq:NSAC} was studied by Gal and Grasselli \cite{GalGrasselliDCDS}. Recent analytic results on a mass-conserving Navier--Stokes/Allen--Cahn system and further references can be found in \cite{GGW2022,GHP2026}.

We will prove convergence of solutions to \eqref{eq:NSAC} to solutions of the following two-phase Navier--Stokes/mean curvature flow system:
\begin{subequations}
	\label{eq:TPNS}
	\begin{alignat}{2}
		\label{eq:TPNS1}
		\p_t \bv^\pm+\bv^\pm \cdot\nabla\bv^\pm-\nu^\pm\Delta \bv^\pm  +\nabla p^\pm &= 0 &\qquad &\text{in }\Omega^\pm (t), t\in (0,T_0),\\\label{eq:TPNS2}
		\Div \bv^\pm &= 0 &\qquad &\text{in }\Omega^\pm (t), t\in (0,T_0),\\\label{eq:TPNS3}
		\llbracket 2\nu^\pm D\bv^\pm -p^\pm \tn{I}\rrbracket\no_{\Gamma_t} &=- \sigma H_{\Gamma_t}\no_{\Gamma_t} && \text{on }\Gamma_t, t\in (0,T_0),\\ \label{eq:TPNS4}
		\llbracket\bv^\pm \rrbracket &=0 && \text{on }\Gamma_t, t\in (0,T_0),\\
		\label{eq:TPNS5}
		V_{\Gamma_t} -\no_{\Gamma_t}\cdot \bv^\pm &= m_0 H_{\Gamma_t} && \text{on }\Gamma_t, t\in (0,T_0),\\
		\bv^-|_{\partial\Omega}&= 0&&\text{on }\partial\Omega\times (0,T_0),\\
		\label{eq:TPNS6}
		(\bv^\pm, \Gamma_t)|_{t=0} &= (\bv_0^\pm, \Gamma_0).
	\end{alignat}
\end{subequations}
Here $\nu^\pm=\nu(\pm1)$, $\no_{\Gamma_t}$ is the interior normal of $\Omega^+(t)$, $ \sigma= \int_{-1}^1 \sqrt{2f(s)}\sd s > 0 $ denotes the surface tension coefficient, $V_{\Gamma_t}$ and $H_{\Gamma_t}$ are the normal velocity and mean curvature of $\Gamma_t$, respectively, and
\begin{equation*}
	\llbracket g\rrbracket:=\lim_{h\to 0+} \left[g(x+h\no_{\Gamma_t})- g(x-h\no_{\Gamma_t})\right], \quad x \in \Gamma_t,
\end{equation*}
denotes the jump of $g$ at $\Gamma_t$ in normal direction $\no_{\Gamma_t}$. It is noted that the existence of strong solutions to \eqref{eq:TPNS} for sufficiently smooth initial data was proved by Moser and the first author in \cite{AM2018}. A similar proof is given in \cite[Appendix A]{HL2023}. By standard parabolic theory one
can show that the solution is indeed smooth for smooth initial data satisfying the
necessary compatibility conditions. Existence of weak solutions for a non-Newtonian
variant of \eqref{eq:TPNS} was proved by Liu, Sato, and Tonegawa \cite{LiuSatoTonegawa2}. We note that if the mobility constant $m_0$ vanishes, then the limit system becomes the classical two-phase flow of incompressible viscous fluids with surface tension. This has been well-studied both for generalized/varifold solutions in \cite{Abels2007}, strong solutions in \cite{DenisovaTwoPhase,KoehnePruessWilkeTwoPhase,WTK2014,PruessSimonettTwoPhaseFlow,PruessSimonettMovingInterfaces}, weak-strong uniqueness in \cite{FischerHensel_NavStk}, and references therein. 

We will show convergence using the following:
\begin{assumption}\label{assump:main}
	Let $\Omega\subseteq \R^d$ be a bounded smooth domain, $d=2,3$, $\nu \colon \R\to (0,\infty)$ be smooth, bounded and with bounded derivatives such that $\inf_{s\in\R}\nu(s)=:\nu_0>0$, $T_0>0$, let  $(\Gamma_t)_{t\in [0,T_0]}\subseteq \Omega$ be smoothly evolving hypersurfaces such that $\Gamma_t= \partial\Omega^+(t)$ for some open set $\Omega^+(t)\subseteq \Omega$ and denote $\Omega^-(t)= \Omega \setminus \ol{\Omega^+(t)}$ for all $t\in [0,T_0]$ and $\Omega^\pm:=\bigcup_{t\in[0,T_0]} \ol{\Omega^\pm(t)} \times \{t\}$.   Moreover, let $\bv^\pm \colon \ol{\Omega^\pm} \to \R^d$, $p^\pm\colon \ol{\Omega^\pm}\to \R$ be smooth, and let  $(\ve^\pm,p^\pm,(\Gamma_t)_{t\in [0,T_0]})$  solve \eqref{eq:TPNS}. Finally, as in \cite{StokesAllenCahn,AbelsFei} we assume that $f\colon \R\to \R$  is smooth and satisfies the assumptions
\begin{equation}\label{eq:Assump-f}
  f(\pm 1)=0, \quad f'(\pm 1)=0, \quad \alpha:=f''(\pm 1)>0,\quad f(s)=f(-s)>0 \quad \text{for all }s\in (-1,1).
\end{equation}
\end{assumption}      
For the convergence, we want to obtain rates in suitable norms and an understanding of the regularities and shape of the diffuse interface as $\eps \to 0$. As mentioned above, so far such a result was only obtained in the case $d=2$ in \cite{AbelsFei} and before in \cite{StokesAllenCahn} for a Stokes/Allen--Cahn system and a lack of suitable higher-order estimates of the occurring errors prevented the extension to three space dimensions. In order to resolve the behavior of higher order derivatives close to the diffuse interface suitably it is essential to use a \emph{weight} $\omega_\eps \colon \ol\Omega\times [0,T_0]\to (0,\infty)$ such that
\begin{equation*}
	\omega_\eps (x,t)= \sqrt{\eps^2+\dist (x,\Gamma_t)^2}\qquad \text{if }\dist (x,\Gamma_t)<2\delta
\end{equation*}
for some suitable $\delta>0$, cf.\ \eqref{eq:weight} below.

Our main result on the sharp interface limit of the Navier--Stokes/Allen--Cahn system for well-prepared initial data is as follows:
\begin{thm}\label{thm:main}
	Let $d=2,3$, $N\in\N$, $N\geq 3$, $\Omega$, $f$, $\Omega^\pm$, $(\bv^\pm,p^\pm,\Gamma)$ be as in Assumption~\ref{assump:main}.
	Then there are smooth $c_{A,0}\colon \Omega\to \R$ and $\bv_{A,0}\colon \Omega\to \R^d$, depending on $\eps\in (0,1)$, such that the following is true: Let $\bv_\eps\colon \ol{\Omega}\times [0,T_\eps)\to \R^d$, $c_\eps\colon \ol{\Omega}\times [0,T_\eps)\to \R$ be smooth solutions of \eqref{eq:NSAC} with initial values $c_{0,\eps}\colon \Omega\to [-1,1]$, $\bv_{0,\eps}\colon \Omega\to \R^d$ and some $T_\eps \in (0,T_0]$, $0<\eps\leq 1$, satisfying
	\begin{align}\nonumber
		&\sum_{j+k\leq 2,j,k\in\N_0}\left\|\tfrac{\oeps^{j+k}}{\eps^j}\nabla^k (c_{0,\eps}-c_{A,0})\right\|_{L^2(\Omega)}+\left\|\nabla_\btau (c_{0,\eps}-c_{A,0})\right\|_{L^2(\Gamma_0(2\delta))}\\\label{initial  assumption}
		&+ \|(\bv_{0,\eps}-\bv_{A,0}, \nabla (\bv_{0,\eps}-\bv_{A,0}))\|_{L^2(\Omega)}
		\leq R_0\eps^{\order+\frac12}
	\end{align}
	for all $\eps\in (0,1]$ and some $R_0>0$.
	Then there are some $\eps_0 \in (0,1]$, $R_1>0$, and $c_A\colon \Om\times [0,T_0]\to \R$, $\bv_A\colon \Om\times [0,T_0]\to \R^d$ (depending on $\eps$) such that, for every $\eps\in (0,\eps_0]$, $\ve_\eps, c_\eps$ can be extended to smooth solutions $\bv_\eps\colon \ol{\Omega}\times [0,T_0]\to \R^d$, $c_\eps\colon \ol{\Omega}\times [0,T_0]\to \R$ of \eqref{eq:NSAC} satisfying  
	\begin{align}
		& \left\|\left(c_\eps-c_A,\omega_\eps \nabla (c_\eps-c_A)\right)\right\|_{L^\infty(0,T_0;L^2(\Omega))} + 
		\|\omega_\eps\partial_t (c_\eps-c_A)\|_{L^2(Q_{T_0})}\nonumber \\
		& \quad +\|\omega_\eps^2 \nabla^2 (c_\eps-c_A)\|_{L^2(Q_{T_0})}+\|(\nabla_\btau (c_\eps-c_A), \omega_\eps \nabla_{\btau}\nabla (c_\eps-c_A))\|_{L^2(\Gamma(2\delta))} \label{eq:convc}
		\leq R_1 \eps^{\order},\\
		&\|(\bv_\eps -\bv_A,\oeps \nabla (\bv_\eps-\bv_A)\|_{L^\infty(0,T_0;L^2(\Om))}\nonumber\\ \label{eq:convVelocityb}
		&\quad + \|(\oeps^{\frac32}\nabla^2 (\bv_\eps -\bv_A),\oeps \partial_t (\bv_\eps-\bv_A))\|_{L^2(Q_{T_0})} \leq R_1\eps^{N+\frac12},
	\end{align}
	and
    \begin{align}\label{eq:convc-Linfty}
		\norm{c_\eps - c_A}_{L^\infty(Q_{T_0})} \leq R_1 \eps^{N-1}
	\end{align}
	for all $\eps \in (0,\eps_0]$.
	Here $\Gamma(2\delta)$ is a tubular neighborhood of $\Gamma$ for some $\delta>0$ as in Section~\ref{subsec:coordinates} below, $\oeps$ is defined in \eqref{eq:weight} below, and $c_A,\bv_A$ are as in Section~\ref{subsec:ApproxSol} below. In particular,
	\begin{equation*}
		\lim_{\eps\to 0} c_A= \pm 1 \quad \text{uniformly on compact subsets of } \Omega^\pm.
	\end{equation*}
	and
	\begin{equation*}
		\bv_A=\bv^+\chi_{\Omega^+}+\bv^-\chi_{\Omega^-} + O(\eps) \qquad \text{in }L^\infty(\Om\times (0,T_0))\text{ as }\eps\to 0.
	\end{equation*}
      \end{thm}
      \begin{rem}
        In principle one can get convergence in even higher-order norms for the velocity as well (hence in $L^\infty (\Omega \times (0,T_0))$ similar to \eqref{eq:convc-Linfty}) if one adds a power of $\oeps$ for each additional spatial derivative and extends the results for the linearized system accordingly.
        
        Moreover, the analogous result holds true for solutions $c_\eps$ of the Allen--Cahn equation with Dirichlet boundary conditions, i.e., \eqref{eq:NSAC3}-\eqref{eq:NSAC4} with $\ve_\eps\equiv 0$, as long as a solution of the mean curvature flow, i.e., \eqref{eq:TPNS5} with $\ve^\pm\equiv 0$ exists on a time interval $[0,T_0]$ analogously to Assumption~\ref{assump:main} and the initial data are well-prepared. One just has to delete all terms with $\ve_\eps, \ve_A, p_A,\ve_{0,\eps}, \ve_{A,0}$. This is a new result for the Allen--Cahn equation itself. The proof can be easily extracted from the proofs in the following, where many of the difficulties due to the coupling with the Navier--Stokes system disappear.
      \end{rem}
\begin{rem}
	Actually, in Section~\ref{sec:Convergence} more precise estimates for the error $u=c_\eps-c_A$ are shown. These are based on a decomposition $u=u_A +\ol u$ into a \emph{leading order} $u_A$ of the error, which is constructed quite explicitly and is order $O(\eps^N)$ in $L^\infty(0,T_0;L^2(\Omega))$, and a \emph{lower order} contribution $\ol u$, which is of order  $O(\eps^{N+\frac12})$ in $L^\infty(0,T_0;L^2(\Omega))$. We have chosen simplified estimates in Theorem~\ref{thm:main} for a less involved presentation.
\end{rem}
\begin{rem}
	We note that Theorem~\ref{thm:main} extends \cite[Theorem 1.1]{AbelsFei}, where only the two-dimensional case without the estimates involving $\oeps$ is shown. At first sight the order of the error estimate for ``$c_\eps-c_A$'' in   \cite[Theorem 1.1]{AbelsFei}, which is $O(\eps^{N+\frac12})$ in $L^\infty(0,T_0;L^2(\Omega))$, seems to be better than the one in \eqref{eq:convc}. But this is misleading since the leading part of the error, which is $u_A$ in our contribution, is part of the definition of the approximate solution ``$c_A$'' in \cite{AbelsFei}, i.e., $c_A+u_A$ in our contribution corresponds to the ``$c_A$'' in the estimate in \cite[Theorem 1.1]{AbelsFei}. Therefore ``$c_\eps-c_A$'' in \cite[Theorem 1.1]{AbelsFei} corresponds to $\ol u$ in our contribution, which is of order $O(\eps^{N+\frac12})$ in $L^\infty(0,T_0;L^2(\Omega))$, cf.\ \eqref{eq:MainUbarEstim} below.
\end{rem}

The proof of Theorem~\ref{thm:main} and the improved error estimates are directly linked to uniform higher-order estimates for the associated linearized Navier--Stokes/Allen--Cahn system
\begin{subequations}
	\label{eq:linNSAC}
	\begin{alignat}{2}\nonumber
		\partial_t \we +\bv_A\cdot \nabla \we+\we\cdot \nabla \bv_A&-\Div(2\nu(c_A)D\we) - \Div(2\nu'(c_A)uD\ve_A) \\\label{eq:linNSAC1}
		+\nabla q &= -\eps \Div (\nabla u \otimes \nabla c_A+\nabla c_A \otimes \nabla u)+\mathbf{r}_1&& \quad \text{in}\ Q_{T_0},\\\label{eq:linNSAC2}
		\Div \we& = r_2&& \quad \text{in}\ Q_{T_0},\\\label{eq:linNSAC3}
		\partial_t u +\bv_A\cdot \nabla u+ \we\cdot \nabla c_A & =m_0\left[\Delta u - \tfrac{1}{\eps^{2}} f''(c_A)u\right] +r_3 && \quad \text{in}\ Q_{T_0},\\
		\label{eq:linNSAC4}
		(\we,u)|_{\partial\Omega}&= (\boldsymbol{0},0)&& \quad \text{on }S_{T_0},\\
		\label{eq:linNSAC5}
		(\we,u) |_{t=0}& = (\we_0,u_0)&& \quad \text{in }\Omega,
	\end{alignat}
\end{subequations}
where $Q_{T_0} \coloneqq \Omega\times (0,T_0)$, $S_{T_0} \coloneqq \partial\Omega\times (0,T_0)$
and $\bv_A, c_A$ are suitable ``approximate solutions'', cf.\ Section~\ref{subsec:ApproxSol} and Remark~\ref{rem:PropertiesAproxSol} below. 
Therefore it is the second goal of this contribution to obtain optimal estimates for this linearized system in weighted $L^2$-Sobolev type spaces, which are second order in space of $\we$ and third order in space for $u$. We believe that this result is of independent interest in order to treat sharp interface limits of similar system. Moreover, it might be the basis for treating other models as e.g.\ Navier--Stokes/Cahn--Hilliard type systems as well. To this end it is essential to introduce a decomposition $u= u_A+\ol u$, where $u_A$ is a leading part of (the error) $u$ in terms of decay as $\eps\to 0$. The details are given in Theorem~\ref{thm:FullLinearizedSystem} below, which is our second main result. Here $u_A$ is constructed explicitly with the aid of a linearization of the limit system \eqref{eq:TPNS}, which is \eqref{eq:Limit} below and can be considered as the sharp interface limit of the linearized system \eqref{eq:linNSAC}.

In the following we will usually set $m_0=1$ for simplicity. But all results hold true for general $m_0>0$.

\paragraph{Literature overview.} Rigorous derivations of sharp interface limits were first developed for scalar phase-field equations. For the Allen--Cahn equation,
de Mottoni and Schatzman \cite{DeMottoniSchatzman} proved convergence to smooth solutions of the mean-curvature flow for well-prepared initial data by constructing high-order approximate solutions using matched asymptotic expansions and controlling the error by a spectral estimate for the linearized Allen--Cahn operator. Chen established a general spectral estimate for the Allen--Cahn, Cahn--Hilliard, and phase-field equations near generic interfaces \cite{ChenSpectrumAC}, which is the basis for the error estimates. 
This method has subsequently been extended to several phase-field models, including mass conserving Allen--Cahn equation \cite{ChenHilhorstLogak}, vector-valued nematic-isotropic phase transition \cite{FWZZ2018}, higher-order approximations of the Willmore flow \cite{FL2021}, and matrix-valued Allen--Cahn equation \cite{FLWZ2023} from the celebrated Keller--Rubinstein--Sternberg problem \cite{KRS1989a,KRS1989b}. Other approaches to the sharp interface limit include the De Giorgi’s $\Gamma$-convergence by Modica and Mortola \cite{Modica1,ModicaMortola1} for the Cahn--Hilliard energy, the comparison principle, as long as the mean curvature flow has a smooth solution, by Chen~\cite{ChenACLimit}, the viscosity-solution framework of Evans, Soner, and Souganidis \cite{EvansSonerSouganidis}, Ilmanen's convergence to Brakke motion via varifold solutions \cite{Ilmanen}, the multiphase convergence of Laux and Simon in a BV-setting \cite{LS2018}, and the convergence of minimizers of vectorial Allen--Cahn energy by Lin, Pan, and Wang \cite{LPW2012}. Interested readers are also referred to \cite{Bethuel2025,Chen2026,FonsecaTartar1989,Liu2024,AlikakosGeng2024,BOS2006,MizunoTonegawa} and references therein for other extensions. More recently, Fischer, Laux, and Simon \cite{FischerLauxSimon_AC_MCF} introduced a relative entropy (also called modulated energy) method of quantitative convergence to smooth mean-curvature flow which combines with interface calibration and avoids the spectral analysis of the linearized Allen--Cahn operator. This has been adapted by Laux and Liu \cite{LL2021} for a class of nematic-isotropic phase transition and by
Liu \cite{Liu2025} to the vectorial Allen--Cahn equation (also called Ginzburg--Landau equation) with potentials of high-dimensional wells.

Once hydrodynamics is involved, rigorous convergence of phase-field models is more delicate because the velocity, pressure, and capillary stress interact with the diffuse transition layer. Strong convergence with rates for fluid-coupled models was first obtained in two space dimensions for a Stokes/Allen--Cahn system by Liu and the first author \cite{StokesAllenCahn}. For the Navier--Stokes/Allen--Cahn system with phase-dependent viscosities, the first two authors proved a sharp interface limit in two dimensions by combining matched asymptotic expansions with refined spectral estimates \cite{AbelsFei}. In the case of vanishing mobility $m_\varepsilon=m_0 \sqrt{\varepsilon}$, the first two and fourth authors \cite{AbelsFeiMoser} proved convergence in two dimensions to the classical two-phase Navier--Stokes system with surface tension, using fractional-order expansions, $\varepsilon$-dependent coordinates,
and new higher-order ansatz terms. Later extensions can be referred to \cite{JSX2023,AbelsMumtaz}. All results mentioned here are restricted in two spatial dimensions and convergence is shown in lower order spaces in comparison to the present contribution, mainly related to the energy estimates for the system. 

A second line of recent work is based on a relative entropy method for the sharp interface limit, which was built upon the pioneering \cite{FischerLauxSimon_AC_MCF} for the sharp interface limit for the Allen--Cahn equation mentioned above, as well as on a weak-strong uniqueness result for a classical sharp interface model for a two-phase flow of viscous incompressible fluids by Fischer and Hensel \cite{FischerHensel_NavStk}. Based on these Hensel and Liu \cite{HL2023} proved convergence rates for solutions of \eqref{eq:NSAC} with constant viscosity $\nu$ and mobility $m_0$ in dimensions $d=2,3$ to \eqref{eq:TPNS}. Fischer, the first and last authors \cite{AFM2024ARMA} later extended the approach to approximate the classical two-phase Navier--Stokes system with surface tension by a Navier--Stokes/Allen--Cahn system in dimensions $d=2,3$ for same viscosities and vanishing mobility $m_\varepsilon=m_0\varepsilon^\beta$ with $\beta\in(0,2)$ in the subcritical range. A counter-example for convergence in the case $\beta>2$ is given in \cite{Abels2022} by the first author. 
These works show the advantages of relative entropy method that the dimension is not restricted, the scaling of the mobility can be more general, no construction of approximate solutions of a sufficiently high order is needed and the proofs are relatively short. However, the convergence is shown in norms related to the energy of the system and it does not seem to be possible to obtain $L^\infty$-convergence of the error. Moreover, this method strongly relies on the modulated surface energy construction $ \int_\Omega (1 - \xi \cdot \bn_{\Gamma_t}) \dx $ with a calibration $\xi$ from the Allen--Cahn equation \cite{FischerLauxSimon_AC_MCF}, which does not seem to apply directly for the Cahn--Hilliard coupled system (not even for the Cahn--Hilliard equation itself) to the best of our knowledge.

Now let us mention results concerning the Cahn--Hilliard equation. Alikakos, Bates, and Chen \cite{AlikakosLimitCH} established the convergence to the Mullins--Sekerka problem (also called the two-phase Hele--Shaw problem) with a modification of the matched asymptotic expansion method used by De Mottoni and
Schatzman \cite{DeMottoniSchatzman}, under the assumption that the limiting sharp interface problem has a smooth solution on a certain time interval. For other methods for convergence we refer to \cite{Chen1996,SandierSerfaty2004,LeCahnHilliardLimit,ApproxCHSolutions}. 
An important and widely used diffuse interface model for two-phase flows of viscous incompressible fluids of same densities is the so-called \emph{model H}, which leads to a Navier--Stokes/Cahn--Hilliard system. For such coupled systems, only few convergence results with rates are known. In \cite{AbelsMarquardt1,AbelsMarquardt2}, Marquardt and the first author considered a coupled Stokes/Cahn--Hilliard system in two dimensions. It is shown that smooth solutions of the diffuse interface systems converge for short times to solutions of the corresponding sharp interface model, the so-called two-phase Stokes/Mullins--Sekerka system, where the evolution of the interface is governed by a Mullins--Sekerka system with an additional convection term coupled to a two–phase stationary Stokes system with the Young--Laplace law for the jump of an extra contribution to the stress tensor, representing capillary stresses. Let us remark that proving a corresponding convergence result for the model H has remained open for more than two decades. For global in time convergence results to suitable varifold solutions, we refer to \cite{AL2014,AR2009,LY2025}.
We note that first three authors are working on incorporating the framework of this contribution to establish the rigorous sharp interface limit of the model H. This will also indicate the robustness of our new framework. 

\paragraph{Strategy and Novelties.}
In this contribution, we develop a perturbative framework for the error $U$ with respect to an approximate solution for the sharp interface limit problem in which the linearized dynamics and the nonlinear remainder are separated in a systematic way as
\begin{align*}
	\mathcal L_\varepsilon U = \mathcal N_\varepsilon(U) + \mathcal R_\varepsilon.
\end{align*}
More precisely, after constructing an approximate solution adapted to the solution of sharp interface limit, the equation for the error is decomposed into a principal linear system $ \mathcal{L}_\eps(U) = \mathcal{F} $, given by \eqref{eq:linNSAC} in the present situation, describing the leading order of the error as the diffuse interfacial thickness tends to zero and several terms become singular at the interface for some given data $ \mathcal{F} $. Moreover, $ \mathcal{N}_\eps(U) $ is a nonlinear remainder, which is quadratic in $U$ and which can be dominated by the linear operator $\mathcal L_\eps$ of the approximation with a large enough order. Here $ \mathcal{R}_\eps $ denotes the residual of the approximate solution. This makes the analysis of the system more efficient and transparent in comparison to the previous work \cite{StokesAllenCahn,AbelsFei,AbelsFeiMoser,AbelsMarquardt1,AbelsMarquardt2}, where certain leading order terms in the error were added ad-hoc to the approximate solution in a nonstandard way. Furthermore, since we obtain invertibility of $\mathcal{L}_\eps$ together with sharp uniform estimates in weighted higher-order $L^2$-Sobolev spaces, we are able to estimate $\mathcal{N}_\eps(U)$ efficiently and treat the case of three space dimensions as well. 

A central novelty of the contribution is the derivation of optimal weighted estimates in Sobolev $L^2$-spaces for the linearized system, uniformly in $\varepsilon$, cf.\ Theorem \ref{thm:FullLinearizedSystem} below, where the weight is as before and defined in \eqref{eq:weight} below. 
This weight captures the precise singular scaling at the diffuse interface of thickness $\eps$, which is close to the interface $\Gamma_t$ in the limit. It behaves like $\varepsilon$ inside the diffuse interface and like the distance to $\Gamma_t$ away from it. 
Therefore it allows one to measure quantities like derivatives in normal direction of the interface,  which may become singular of order $\eps^{-1}$ close to the interface as $\varepsilon\to 0$, but decay uniformly at a positive distance to it. This is combined with a decompostion of the error with respect to  $c_\eps$ in a part in the kernel of the one-dimensional Allen-Cahn operator and orthogonal to it, cf.\ Corollary~\ref{cor:Pu} below.

The main difficulty in the analysis of the linearized system is that it contains several terms which are singular as $\varepsilon\to0$. These terms arise from the rapidly varying profile across the diffuse layer, from normal derivatives in the stretched coordinate, from the curvature of the moving interface, and from the coupling between the phase-field variable and the fluid variables. If one estimates the full error directly, these singular contributions cannot be treated as perturbative remainders. To analyse the full linearized system, we essentially decompose it into a leading error $ (\bw_A,u_A) $, cf.\ \eqref{eq:w_A} and \eqref{eq:u_A} below, and a slightly more regular remainder $ (\ol{\bw},\ol{u}) $. This decomposition is one of the crucial points of the analysis. The component $(\bw_A,u_A)$ contains the most singular part of the error generated by the diffuse interface approximation and can be controlled with the aid of the solution to a linearization of the sharp interface limit \eqref{eq:TPNS}, which gives a kind of convergence result for the linearized systems. After this leading part is extracted, the remaining part of the solution $(\ol{\bw},\ol{u})$ satisfies a better behaved reduced linearized system, for which the uniform weighted estimates can be shown.

Finally, we strongly believe that the developed method can be adapted well to other open problems on sharp interface limits or other problems with singular limits since the approach is rather systematic and the problem is reduced to the analysis of the linearized system in the singular limit. Optimal estimates for the linearized system make the approach much more robust than previous works based on suboptimal estimates and ad-hoc introduction of leading order terms. %The first three authors of this contribution are currently working on an extension to the Navier--Stokes/Cahn--Hilliard system given by the model H.
We expect the approach can be used to treat the sharp interface limit for the Navier--Stokes/Cahn--Hilliard system given by the model H. It might also be possible to treat an important model for vesicle-fluid interaction, which couples the incompressible viscous Navier--Stokes equations and a phase-field approximation of the so-called Canham--Helfrich energy, thus providing a detailed description of the lipid bilayer geometry \cite{DLRW2009}. A simplified approximation of the Willmore flow was justified by the second author and Liu in \cite{FL2021}, while the whole hydrodynamics was left open still. 
Since \cite{KRS1989a,KRS1989b}, Rubinstein, Sternberg, and Keller introduced a vector-valued system for fast reaction and slow diffusion:
\begin{align*}
	\partial_t \bu = \Delta \bu - \tfrac{1}{\eps^2} \partial_\bu F(\bu),
\end{align*}
where $ \bu\colon \Omega \times (0,T) \to \bbr^d $ is a phase-indicator function, and the nonnegative, smooth potential function $ F\colon \bbr^d \to \bbr $ vanishes exactly on two disjoint connected sub-manifolds in $ \bbr^d $. In \cite{FLWZ2023}, the second author and his coauthors provided a solution to the Keller--Rubinstein--Sternberg problem in the $ O(d) $ setting. In more general setting this problem still remains largely open.
In the presence of a fluid flow, only a formal derivation of the Landau--de Gennes $ \mathbf{Q} $-tensor theory was given by the second author and his coauthors \cite{FWZZ2015}, while a rigorous justification has remained open. We believe that our framework will be helpful to solve the whole hydrodynamics of the Keller--Rubinstein--Sternberg problem in the future.

The paper is organized as follows. In Section~\ref{sec:Prelim} we summarize basic notations on function spaces, interpolation inequalities, coordinate transformations, and the weights $\oeps$ used in our analysis. Moreover, we recall and extend needed results on existence of approximate solutions, spectral estimate for the linearized Allen--Cahn operator and estimates of remainder terms. This is complemented by some auxiliary results for some elliptic equations and bilinear forms in $\oeps$-weighted $L^2$-Sobolev spaces. In Section~\ref{linearized system} first order ``energy-type'' estimates for the linearized system, including an estimate for $\|\oeps\nabla u\|_{L^2(Q_T)}$, are shown. This is the basis for the higher-order estimates. For a reduced linearized system this is done in Section~\ref{sec:higher-order}. Then in Section~\ref{sec:FullSystem} the full linearized system \eqref{eq:linNSAC} is reduced to the system studied in the previous section using a carefully constructed leading order part of the error $(\we_A,u_A)$, which involves the sharp interface limit of the linearized system \eqref{eq:linNSAC} and $\rho$-derivatives of the inner expansion of the first two terms of the approximate solution $(\ve_A,c_A)$. This shows the main result on our linearized system. With the aid of this result the main result on the sharp interface limit of the Navier--Stokes/Allen--Cahn system \eqref{eq:NSAC} is shown in Section~\ref{sec:Convergence} with the aid of a continuation argument and careful nonlinear estimates.

\section{Notations and Preliminaries}\label{sec:Prelim}

\subsection{Function Spaces}

Throughout the manuscript $\Omega\subseteq \R^d$, $d=2,3$, is a bounded and smooth domain with exterior normal $\no_{\partial\Omega}$.
The standard $L^q$-Sobolev spaces are denoted by $W^m_q(\Omega)$, where $1\leq q\leq \infty$, $m\in \N_0$, and $L^q(0,T;X)\equiv L^q((0,T);X)$, $W^m_q(0,T;X)\equiv W^m_q((0,T);X)$ denote the $X$-valued variants for strongly measurable $f\colon (0,T)\to X$ and an arbitrary Banach space $X$. If $q=2$, we use the convention $H^m=W^m_2$. For brevity we will often write $\|\cdot\|_{L^p(0,T;X)}$ instead of $\|\cdot\|_{L^p(0,T;X(\Omega))}$ for $X=L^q, H^k,...$. Furthermore we will use the subspaces
\begin{align*}
  L^p_{(0)}(\Omega) &\coloneqq \left\{f\in L^p(\Omega): \int_\Omega f(x)\sd x=0 \right\}, \quad\\
  H^1_{(0)}(\Omega)&\coloneqq  H^1(\Omega)\cap L^2_{(0)}(\Omega),\quad H^{-1}_{(0)}(\Omega)= (H^1_{(0)}(\Omega))'.
\end{align*}
If $U\subseteq \R^d\times (0,T)$ is an open set and $U_t\coloneqq \{x\in \R^d: (x,t)\in U\}$ for all $t\in (0,T)$, then we define for sufficiently smooth $u\colon U\to \R^N$ such that $(0,T)\ni t\mapsto \|u(t)\|_{X(U_t)}$ is measurable the norms
\begin{equation*}
  \|u\|_{L^p(0,T;X(U_t))} \coloneqq 
  \begin{cases}
    \left(\int_0^T \|u(t)\|_{X(U_t)}^p\,\dt\right)^{\frac1p}&\text{if }p<\infty,\\
    \operatorname{ess\,sup}_{t\in (0,T)}\|u(t)\|_{X(U_t)} &\text{if } p=\infty,
  \end{cases}
\end{equation*}
where $X(U_t)$ is Sobolev space with a suitable norm. In the following the measurability of $(0,T)\ni t\mapsto \|u(t)\|_{X(U_t)}$ will be clear from the context, where $u$ will be sufficiently regular and $U_t$ will depend on the smoothly evolving $(\Gamma_t)_{t\in [0,T_0]}$ in a suitable manner. 

Moreover, we denote by $L^2_\sigma(\Omega)$ the closure of divergence free smooth and compactly supported functions in $L^2(\Omega)^d$ and
\begin{equation*}
  V_\sigma \coloneqq  H^1_0(\Omega)^d\cap L^2_\sigma(\Omega).
\end{equation*}

\subsection{Evolving Interfaces and Coordinates}\label{subsec:coordinates}
We consider smooth evolving compact orientable hypersurfaces $\Gamma= \bigcup_{t\in [0,T_0]}\Gamma_t \times \{t\}$ for some  $T_0>0$, which separates $\Omega$ into two disjoint open sets $\Omega_t^\pm$ such that $\partial\Omega^+_t= \Gamma_t$ for all $t\in [0,T_0]$. Moreover, we choose a normal field $\no_{\Gamma_t}\colon \Gamma_t\to \R^d$, $t\in[0,T_0]$, such that $\no_{\Gamma_t}$ is the interior normal with respect to $\Omega^+_t$. $V_{\Gamma_t}$ and $H_{\Gamma_t}$ denote the normal velocity and mean curvature (sum of principal curvatures) of $\Gamma_t$ with respect to the chosen normal $\no_{\Gamma_t}$.
Furthermore, let $\Sigma\subseteq \R^d$ be a smooth and compact $(d-1)$-dimension ``reference'' hypersurface such that there is a sufficiently smooth mapping $X_0\colon \Sigma \times [0,T_0]\to \Gamma $ such that $X_0(\cdot, t)\colon \Sigma \to \Gamma_t$ is a diffeomorphism for every $t\in [0,T_0]$. For $\delta'>0$ we set $\Sigma_{\delta'}\coloneqq  (-\delta',\delta')\times \Sigma$. Furthermore, let $\delta>0$ be such that tubular neighborhood coordinates $(d_{\Gamma_t},P_{\Gamma})\colon \overline{\Gamma(3\delta)}\to [-3\delta,3\delta]\times \Gamma$ exist, where $\Gamma(\delta')\coloneqq  \{(x,t)\in \R^d\times [0,T_0]: |d_\Gamma(x,t)|<\delta'\}$, $d_\Gamma(.,t)=d_{\Gamma_t}$ is the signed distance function to $\Gamma_t$ and $P_{\Gamma}(.,t)\equiv P_{\Gamma_t}$ is orthogonal projection on $\Gamma_t$, i.e., $P_{\Gamma_t}x\in \Gamma_t$ is for each $x\in\Gamma_t(3\delta)$ the unique closest point on $\Gamma_t$ to $x$. Here $d_\Gamma\colon \Gamma(3\delta)\to \R$ is smooth.
We also use the notation
\begin{equation*}
  \no(s,t)= \no_{\Gamma_t}(X_0(s,t))\qquad \text{for all }s=(s_1,\ldots, s_{d})\in \Sigma,t\in [0,T_0].
\end{equation*}
For differentiable $u\colon \Gamma(3\delta) \to \R$ we define the tangential gradient
\begin{equation*}
  \nabla_{\btau} u = (I-\no_{\Gamma_t}\otimes \no_{\Gamma_t}) \nabla u.
\end{equation*}
For vector-valued differentiable functions this is applied component-wise.

In $\Gamma(3\delta)$ we introduce new coordinates with the aid of
\begin{equation*}
	X\colon (-3\delta,3\delta)\times \Sigma\times [0,T_0]\to \Gamma(3\delta)\quad\text{with }X(r,s,t)\coloneqq  X_0(s,t)+ r\no(s,t),
\end{equation*}
Then $X^{-1}(x,t)=(r,s,t)$ with
\begin{equation*}
  r= d_\Gamma(x),\qquad s= X_0^{-1}(P_{\Gamma_t}(x),t)\eqqcolon  S(x,t)=(S_1(x,t),\ldots,S_d(x,t)).
\end{equation*}
Furthermore we have for integrable $f\colon \Gamma_t(\Gamma(\delta'))\to \R$, $\delta \in (0,3\delta)$
\begin{equation}
  \label{eq:IntChangeOfVariables}
  \int_{\Gamma_t(\delta')} f(x) \sd x = \int_{-\delta'}^{\delta'}\int_{\Sigma} f(X(r,s,t))J(r,s,t)\sd s\sd r
\end{equation}
where ``$\sd s$'' denotes integration with respect to the canonical volume element on $\Sigma$ and $J\colon \Gamma(3\delta)\to (0,\infty)$ is uniformly bounded above and below by  positive constants.

Throughout this contribution we denote
\begin{equation}
  \label{eq:weight}
  \oeps= \omega_\eps(x,t)= \sqrt{\eps^2+ \tilde{d}(x,t)^2},
\end{equation}
where $\tilde{d}\colon \ol{\Omega} \times [0,T]$ is a smooth function such that 
  \begin{align*}
    \tilde{d}(x,t) &=
    \begin{cases}
      d_{\Gamma_t}(x) & \text{for all } (x,t)\in \Gamma(2\delta),\\
      \pm 1 & \text{for all } (x,t)\in \Omega^\pm \setminus \Gamma(3\delta),
    \end{cases}\\
  \tilde{d}(x,t)&\neq 0\qquad \text{for all }(x,t)\not \in \Gamma,  
  \end{align*}
  and denote $\wW^m(\Omega')= H^m(\Omega')= W^m_2(\Omega')$ normed by $\|\cdot\|_{\wW^m(\Omega')}$ with
  \begin{align*}
    \|u\|_{\wW^m(\Omega')}^2 = \sum_{k+l\leq m}\left\|\tfrac{\oeps^{k+l}}{\eps^l}\nabla^k u\right\|_{L^2(\Omega')}^2, 
  \end{align*}
  where $k,l\in\N_0$ in the sum, $\Omega'\subseteq \Omega$, $\eps\in (0,1]$, and $m\in\N_0$. If $\Omega'=\Omega$, we simply write $\wW^m$. This $\eps$-dependent norm will play an important role for higher order estimates of the linearized system.
  
  For the following we fix some $\zeta\in C^\infty_0(\R)$ with $\supp \zeta \subset [-2\delta,2\delta]$ and $\zeta(r)=1$ for all $r\in [-\delta,\delta]$. Moreover, we set
  \begin{equation}\label{eq:zg}
    \zg(x,t)\coloneqq 
    \begin{cases}
      \zeta(d_{\Gamma_t}(x)) &\text{if } (x,t)\in \Gamma(2\delta),\\
      0 &\text{else}.
    \end{cases}
  \end{equation}
  In $(r,s)$-coordinates we will use
  \begin{equation*}
    \toeps= \toeps(r)= \sqrt{\eps^2+r^2}\qquad \text{for all }r\in(-2\delta,2\delta), \eps>0.
  \end{equation*}

We associate to a function $\phi\colon \Gamma(3\delta)\to \R$ the function $\tilde{\phi}\colon \Sigma_{3\delta}\times [0,T_0]\to\R$ defined by
 \begin{equation}\label{eq:1.4}
   \phi(x,t)=\tilde{\phi}(d_{\G}(x,t),S(x,t),t) \quad \text{for all }(x,t)\in \Gamma(3\delta). %\quad\text{for}\quad\phi(X_0(s,t)+r\no(s,t),t)=\tilde{\phi}(r,s,t)
 \end{equation}
  Then
\begin{equation}\label{Prelim:1.13}
  \begin{split}
    \partial_t \phi(x,t) &= -V_{\Gamma_t} (P_{\Gamma_t}(x)) \partial_r\tilde{\phi}(r,s,t) + \partial_{t}^\Gamma \tilde{\phi}(r,s,t), \\
  \nabla \phi(x,t) &= \no_{\Gamma_t} (P_{\Gamma_t}(x)) \partial_r\tilde{\phi}(r,s,t) + \nabla^ \Gamma  \tilde{\phi}(r,s,t), \\
 \Delta \phi(x,t) &= \partial_r^2\tilde{\phi}(r,s,t) + \Delta d_{\G_t}(x) \partial_r\tilde{\phi}(r,s,t) +  \Delta^{\Gamma} \tilde{\phi}(r,s,t),
  \end{split}
\end{equation}
where $x=X(r,s,t)$ for all $(r,s,t)\in\Sigma_{3\delta}\times [0,T_0]$  and we denote
\begin{equation}\label{Prelim:1.12}
  \begin{split}
    \partial_{t}^\Gamma \tilde{\phi}(r,s,t) &= \partial_t \tilde{\phi}(r,s,t) + \p_t S(x,t)\cdot\nabla_{\Sigma} \tilde{\phi}(r,s,t) ,\\
\nabla^{\Gamma} \tilde{\phi}(r,s,t) &=   D S(x,t)^T \nabla_\Sigma \tilde{\phi}(r,s,t)=\sum_{i=1}^{d}\nabla S_i\cdot \p_{s_i}  \tilde{\phi}(r,s,t) ,\\
\Delta^{\Gamma} \tilde{\phi}(r,s,t) &=   (\Delta S)(x,t)\cdot\nabla_\Sigma \tilde{\phi}(r,s,t)+\sum_{i,j=1}^{d}\nabla S_i\cdot \nabla S_j \p_{s_is_j}^2  \tilde{\phi}(r,s,t) 
  \end{split}
\end{equation}
for all $(r,s,t)\in\Sigma_{3\delta}\times [0,T_0]$ and suitable $\phi$.  Here and in what follows $ (\nabla_{\Sigma}(\cdot))_i=\partial_{s_i}(\cdot),i=1,\cdots,d.$
 We refer to \cite[Section~4.1]{ChenHilhorstLogak} for more details.
Finally, we  define for every suitable $h\colon \Sigma\times [0,T_0]\to \R$
\begin{equation*}%\label{eq:1.27}
%\begin{split}
  (\nabla_\G h)(s,t)\coloneqq (\nabla^\Gamma h)(0,s,t),\quad
  (\Delta_\G h)(s,t)\coloneqq (\Delta^\G h)(0,s,t),\quad
  (D_t h)(s,t)\coloneqq (\p_t^\G h)(0,s,t),
%\end{split}
\end{equation*}
where $h$ is considered as a function on $\Sigma_{3\delta}\times [0,T_0]$, which is independent of $r\in(-3\delta,3\delta)$.

\subsection{An Anisotropic Interpolation Inequality}
In order to obtain good error estimates in a neighborhood of the interface $\Gamma$, we will use several anisotropic variants of Gagliardo-Nirenberg-type interpolation inequalities.
\begin{lem}[Anisotropic $ L^p $-estimate]
	\label{lem:Lp}
	Let $ 2 < p < 6 $, $ d \leq 3 $, $\delta' \in (0,2\delta]$. There is some $C > 0$ such that for every $u \in H^1(\Omega)$ and $t \in [0,T_0]$ we have
	\begin{align*}
		\norm{u}_{L^p(\Gamma_t(\delta'))}^p
		& \leq C \left(\norm{u}_{L^2(\Gamma_t(\delta'))} + \norm{\pr u}_{L^2(\Gamma_t(\delta'))}\right)^{\frac{p-2}{2}} \\
		& \qquad \times \left(\norm{u}_{L^2(\Gamma_t(\delta'))} + \norm{\nabla_\btau u}_{L^2(\Gamma_t(\delta'))}\right)^{p-2} \norm{u}_{L^2(\Gamma_t(\delta'))}^{\frac{6 - p}{2}}.
	\end{align*}
\end{lem}
\begin{proof}
	By means of the H\"older inequality 
	and the Minkowski inequalities, it holds
	\begin{align*}
		\norm{u}_{L^p(\Gamma_t(\delta'))}^p
		& \leq C \int_{\Gamma_t} \norm{u}_{H^1(-\delta',\delta')}^{\frac{p-2}{2}}
		\normm{u}_{L^2(-\delta',\delta')}^{\frac{p+2}{2}} \dsigma \\
		& \leq C \norm{u}_{H^1(-\delta',\delta';L^2(\Gamma_t))}^{\frac{p-2}{2}}
		\norm{\norm{u}_{L^{\frac{2(p+2)}{6-p}}(\Gamma_t)}}_{L^2(-\delta',\delta')}^{\frac{p+2}{2}} \\
		& \leq C \norm{u}_{H^1(-\delta',\delta';L^2(\Gamma_t))}^{\frac{p-2}{2}}
		\norm{
			\norm{u}_{H^1(\Gamma_t)}^{\frac{2(p-2)}{p+2}}
			\norm{u}_{L^2(\Gamma_t)}^{\frac{6-p}{p+2}}
		}_{L^2(-\delta',\delta')}^{\frac{p+2}{2}} \\
		& \leq C\norm{u}_{H^1(-\delta',\delta';L^2(\Gamma_t))}^{\frac{p-2}{2}}
		\norm{
			\norm{u}_{H^1(\Gamma_t)}^{\frac{2(p-2)}{p+2}}
		}_{L^{\frac{p+2}{p-2}}(-\delta',\delta')}^{\frac{p+2}{2}}
		\norm{
			\norm{u}_{L^2(\Gamma_t)}^{\frac{6-p}{p+2}}
		}_{L^{\frac{2(p+2)}{6-p}}(-\delta',\delta')}^{\frac{p+2}{2}} \\
		& \leq C\Big(\norm{u}_{L^2(\Gamma_t(\delta'))}
		+ \norm{\pr u}_{L^2(\Gamma_t(\delta'))}\Big)^{\frac{p-2}{2}} \\
		& \quad \times
		\Big(\norm{u}_{L^2(\Gamma_t(\delta'))}
		+ \norm{ \nabla_\btau u}_{L^2(\Gamma_t(\delta'))}\Big)^{p-2} \norm{u}_{L^2(\Gamma_t(\delta'))}^{\frac{6-p}{2}},
	\end{align*}
	where we used the Gagliardo--Nirenberg inequalities $ \normm{f}_{L^p(-\delta',\delta')} \leq C \normm{f}_{H^1(-\delta',\delta')}^{\frac{p-2}{2p}} \normm{f}_{L^2(-\delta',\delta')}^{\frac{p+2}{2p}} $ for $ f \in H^1(-\delta',\delta') $, and
        \begin{equation*}
        \norm{u}_{L^{\frac{2(p+2)}{6-p}}(\Gamma_t)} \leq C \norm{u}_{H^1(\Gamma_t)}^{\frac{2(p-2)}{p+2}} \norm{u}_{L^2(\Gamma_t)}^{\frac{6-p}{p+2}}\quad \text{for all } u \in H^1(\Gamma_t)  
      \end{equation*}
      since $d\leq 3$.
\end{proof}
      
\subsection{Optimal Profile, Stretched Variable, and Spectral Estimate}

Because of \eqref{eq:Assump-f}, there is a unique  solution $\theta_0 \colon\mathbb{R}\to \R$ of 
\begin{align}\label{eq:OptProfile}
	-\theta_0''+f'(\theta_0)=0\quad\text{ in }\R, \qquad
	\lim_{\rho\to\pm \infty}\theta_0(\rho)=\pm 1, \qquad \theta_0(0)=0,
\end{align}
which is monotone and called optimal profile. Moreover, for every $m\in\N_0$, there is some $C_m>0$ such that
\begin{equation*}
  |\p_\rho^m(\theta_0(\rho)\mp 1)|\leq C_m e^{-\sqrt{\alpha}|\rho|}\quad\text{for all } \rho\in\R  \text{ with }\rho \gtrless 0.
\end{equation*}
Since $f$ is even,  $\theta_0$ is odd and $\theta'_0$ is even.
In what follows, we denote by $\sigma= \int_{\R}\theta_0'(\rho)^2\sd \rho= \int_{-1}^1 \sqrt{2f(s)}\sd s$ the associated surface tension coefficient.

In this subsection we assume, see Section \ref{subsec:ApproxSol} below, that 
\begin{alignat*}{2}
	c_A(x)&= c_{A,0}(x)+ \eps^2c_{A,2+}(x), &\quad& \text{for all } x\in \Om,\\
	c_{A,0}(x)&=\zg \theta_0(\rho)+(1-\zg)\left(
	\chi_{\Omega^+(t)}-\chi_{\Omega^-(t)} \right),&\quad& \text{for all } x\in \Om,
\end{alignat*}
where $\zg$ is as in \eqref{eq:zg}.
Moreover, we assume that $\dist(\Gamma_t,\partial\Omega)> 3\delta$ for all $t\in [0,T_0]$.
For given  continuous functions $(\tilde{h}_\eps)_{0<\eps< 1} \colon \G\to \R$ with $\Gamma\coloneqq \bigcup_{t\in[0,T_0]} \Gamma_t \times {\{t\}}$, we define the \textit{stretched variable} $\rho$ as
\begin{equation}
\label{eq:StretchedVariable}
\rho \coloneqq \frac{d_\Gamma(x,t)}\eps- h_\eps (s,t)\quad \text{for all } (x,t)\in \Gamma(2\delta), \eps \in (0,1], \text{ where } s=S(x,t),
  \end{equation}
with $h_\eps(s,t)= \tilde{h}_\eps (X_0(s,t),t)$.
Furthermore, we assume that
\begin{equation}\label{eq:BoundhcA}
	\sup_{\eps \in (0,1]}\left(\sup_{(p,t)\in \Gamma} |\tilde{h}_\eps (p,t)|+ \sup_{x\in\Om, t\in[0,T_0]} |c_{A,2+} (x,t)|\right)\leq M
\end{equation}
for some $M>0$.  We will apply the results of this subsection to $\tilde{h}_\eps (p,t)= h_\eps(X_0^{-1}(p,t),t)$ for some $h_\eps\colon \Sigma\times [0,T_0]\to \R$.

The following spectral estimate due to  \cite[Theorem~2.13]{StokesAllenCahn} is a key ingredient for the proof of convergence.
\begin{thm}\label{thm:Spectral}
	Let $c_A$ be as above and \eqref{eq:BoundhcA} be satisfied for some $M>0$. Then there are some $C_L,\eps_1>0$, independent of $\tilde{h}_\eps, c_A$, such that for every $\psi\in H^1(\Om)$, $t\in[0,T_0]$, and $\eps\in (0,\eps_1]$ we have
	\begin{align*}
		& \int_{\Om}\left(|\nabla\psi|^2+ \tfrac{f''(c_A(\cdot,t))}{\eps^2}\psi^2\right)\sd x  \geq -C_L\int_\Om \psi^2 \sd x + \int_{\Om\setminus \Gamma_t(\delta)} |\nabla \psi|^2\sd x +  \int_{\Gamma_t({\delta})} |\nabla_\btau \psi|^2\sd x.
	\end{align*}
\end{thm}
For the following we denote $\Veps\coloneqq  H^1(\Omega)$ with $\eps$- and $t$-dependent norm $\|.\|_{\Veps}$ defined by
\begin{equation}\label{eq:VepsNorm}
  \|u\|_{\Veps}^2= \int_{\Om}\left(|\nabla u(x)|^2+ \tfrac{f''(c_A(x,t))}{\eps^2}u^2\right)\sd x +(C_L+1)\int_\Om u(x)^2 \sd x
\end{equation}
for every $u\in \Veps$, $\eps\in (0,\eps_1]$, $t\in [0,T_0]$, where $\eps_1>0$ is as in the previous theorem. Since $f''(c_A(x,t))\geq \alpha/2$ for all $(x,t)\in Q_{T_0}\setminus \Gamma(2\delta)$ and $\eps\in (0,\eps_1]$ (possibly choosing $\eps_1>0$ slightly smaller), we have 
\begin{equation}\label{eq:VepsExterior}
  \|u\|_{\Veps}^2\geq \min (1,\alpha/2) \left\|(\nabla u,\tfrac1\eps u)\right\|_{L^2(\Omega\setminus \Gamma_t(2\delta))}^2
\end{equation}
uniformly in $\eps\in (0,\eps_1]$, $t\in [0,T_0]$, $u\in \Veps$. In $\Gamma_t(2\delta)$ the $V_\eps$-norm is more involved.
As in \cite[Remark~2.10]{AbelsFei} one derives the existence of $C_1,C_2,C,\eps_1\in (0,1]$ such that
  \begin{align}\nonumber
    &\|u\|_{\Veps(2\delta)}^2+\|\nabla_\btau u \|_{L^2(\Gamma_t(2\delta))}^2\\
    &\quad \leq C_1\int_{\Gamma_t(2\delta)} \left(|\nabla u|^2 +\tfrac1{\eps^2} f''(c_A(.,t))u^2\right)\sd x + C_2\|u\|_{L^2(\Gamma_t(2\delta))}^2\leq C\|u\|_{\Veps}^2\label{eq:VepsChar}
  \end{align}
  for all $u\in \Veps(2\delta)$, $\eps\in (0,\eps_1]$, 
where
  \begin{align*}
    \|u\|_{\Veps(2\delta)} \coloneqq  &\inf\left\{ \|Z\|_{H^1(\Sigma)}+ \|v\|_{H^1(\Gamma_t(2\delta))}+\tfrac{1}{\eps}\|v\|_{L^2(\Gamma_t(2\delta))}: Z\in H^1(\Sigma), v\in H^1(\Gamma_t(2\delta)),\right.\\
    &\quad \qquad \left. u(x)=\tfrac{1}{\sqrt{\eps}} Z(s)\theta_0'(\rho) + v(x)\text{ for all }x\in\Gamma_t(2\delta), s=S(x,t),\rho \text{ as in \eqref{eq:StretchedVariable}}\right\}
  \end{align*}
  for $u\in H^1(\Gamma_t(2\delta))$.

        \begin{cor}\label{cor:Pu}
          For $u\in H^1(\Gamma_t(2\delta))$, $t\in [0,T]$, and $\tilde{u}(r,s) = u (X(r,s,t))$ for all $(r,s) \in \Sigma_{2\delta}$  let $P\tilde{u}\colon \Sigma_{2\delta}\to \R$ be defined as
      \begin{align}\label{eq:DefnPu}
        P\tilde{u}(r,s)&\coloneqq  \tilde{u}(r,s) - Q\tilde{u}(r,s),\quad \text{where}\\\nonumber
        Q\tilde{u}(r,s) &\coloneqq  \frac1{\sqrt{\eps\sigma}} \theta_0'(\tfrac{r}\eps -h_\eps(s,t)) 
        \tilde{w}(s), \quad \tilde{w}(s) \coloneqq \frac1{\sqrt{\eps\sigma}}\int_{-2\delta}^{2\delta} \theta_0'(\tfrac{r}\eps -h_\eps(s,t)) \tilde u (r,s)\,\dr
      \end{align}
      for all $r\in (-2\delta,2\delta),s\in \Sigma$. 
      Then there are $C_1,C_2>0,\eps_0\in (0,1)$ (independent of $t\in [0,T_0]$) such that
      \begin{align}
        \label{eq:Pu}
        \frac1{\eps} \|P\tilde{u}\|_{L^2(\Sigma_{2\delta})}+ \|P\tilde{u}\|_{H^1(\Sigma_{2\delta})} 
        +\|\tilde{w}\|_{H^1(\Sigma)}
        \leq C\|u\|_{\Veps(2\delta)}. 
  \end{align}
  for all $u\in H^1(\Gamma_t(2\delta))$ and $\eps\in (0,\eps_0]$.
    \end{cor}
    \begin{proof}
      Let $Z\in H^1(\Gamma_t)$ and $v\in H^1(\Gamma_t(2\delta))$ be such that $u(x)= \tfrac{1}{\sqrt{\eps}} Z(s)\theta_0'(\rho) + v(x)$ and
      \begin{equation*}
        \|Z\|_{H^1(\Gamma_t)}+ \|v\|_{H^1(\Gamma_t(2\delta))}+\frac{1}{\eps}\|v\|_{L^2(\Gamma_t(2\delta))}\leq 2\|u\|_{V_t^\eps}.
      \end{equation*}
      Then
      \begin{equation*}
        P\tilde{u}(r,s) = P\tilde{v}(r,s) + Z(s)\frac1{\sqrt\eps}\theta_0'(\rho)\left(1-  \frac1{\eps\sigma}\int_{-2\delta}^{2\delta} \theta_0'(\tfrac{r}\eps -h_\eps(s,t))^2 \,\dr\right),
      \end{equation*}
      where
      \begin{equation*}
        \left|1- \frac1{\eps\sigma}\int_{-2\delta}^{2\delta} \theta_0'(\tfrac{r}\eps -h_\eps(s,t))^2 \,\dr\right|\leq \int_{\delta}^\infty |\theta_0'(\tfrac{r}\eps)|^2\,\dr \leq Ce^{-\frac{\alpha \delta}\eps}
      \end{equation*}
      for all $\eps>0$ sufficiently small      since $\sigma= \int_{\R}|\theta'_0(z)|^2\, dz$ and
      \begin{align*}
        \|P\tilde{v}\|_{L^2(\Sigma_{2\delta})}\leq C\left(1+\tfrac1{\sigma\eps}\|\theta_0'(\cdot/\eps)\|_{L^2(\R)}^2 \right) \|v\|_{L^2(\Gamma_t(2\delta))}\leq 4C\eps\|u\|_{V_t^\eps}.
      \end{align*}
      This yields the claimed estimate for $\frac1{\eps}\|P\tilde{u}\|_{L^2(\Sigma_{2\delta})}$. Moreover, using
       \begin{align*}
          \partial_r(P\tilde u)(r,s) &= \partial_r(P \tilde v)(r,s) +\partial_r\left( Z(s)\frac1{\sqrt\eps}\theta_0'(\rho)\left(1-  \frac1{\eps\sigma}\int_{-2\delta}^{2\delta} \theta_0'(\tfrac{r}\eps -h_\eps(s,t))^2 \,\dr\right)\right),\\
         \partial_r(P \tilde v)(r,s) &= \partial_r \tilde{v}(r,s) - \frac1{\eps\sigma}\int_{-2\delta}^{2\delta} \theta_0'(\tfrac{r}\eps -h_\eps(s,t)) \tilde v (r,s)\,\dr\, \tfrac1\eps \theta_0''(\tfrac{r}\eps -h_\eps(s,t)), 
       \end{align*}
       and a similar relation for derivatives with respect to $s$ one obtains
      \begin{align*}
        \|\nabla P\tilde u\|_{L^2(\Sigma_{2\delta})}\leq C\left(\|v\|_{H^1(\Gamma_t(2\delta))} +\tfrac1\eps \|v\|_{L^2(\Gamma_t(2\delta))}+\left(1+\frac{e^{-\frac{\alpha\delta}\eps}}\eps\right) \|Z\|_{H^1(\Gamma_t)}\right)
        \leq C\|u\|_{V_t^\eps}.
      \end{align*}
      This shows the estimate for $\|Pu\|_{H^1(\Sigma_{2\delta})}$. Finally, the estimate of $ \|\tilde{w}\|_{H^1(\Sigma)}$ follows in a straightforward manner using the decomposition for $u$.
    \end{proof}

\subsection{Remainder Estimates}
      We use the following classes to systematically handle remainder terms as in \cite[Definition~2.3]{AbelsFei}:
\begin{defn}\label{eq:1.15}
  For any $k\in \R$ and $\tilde{\alpha}>0$,  $\mathcal{R}_{k,\tilde{\alpha}}$ denotes the vector space of families of continuous functions $\tr_\eps\colon \R\times \Gamma(2\delta) \to \R$, indexed by $\eps\in (0,1)$, which are continuously differentiable with respect to $\no_{\Gamma_t}$ for all $t\in [0,T_0]$ such that
  \begin{equation}\label{eq:EstimRkalpha}
    |\partial_{\no_{\Gamma_t}}^j \tr_\eps(\rho,x,t)|\leq Ce^{-\tilde{\alpha} |\rho|}\eps^k\qquad \text{for all }\rho\in \R,(x,t)\in\Gamma(2\delta),\, j\in\{0,1\},\, \eps \in (0,1)
  \end{equation}
  for some $C>0$ independent of $\rho\in \R,(x,t)\in\Gamma(2\delta)$, $\eps\in (0,1)$.  $\mathcal{R}_{k,\tilde{\alpha}}^0$ is the subclass of all $(\tr_\eps)_{\eps\in (0,1)}\in \mathcal{R}_{k,\tilde{\alpha}}$ such that
  $\tr_\eps(\rho,x,t)= 0$ for all  $\rho \in\R, x\in\Gamma_t, t\in [0,T_0]$.
\end{defn}

      Finally, we will frequently use the following result for remainder estimates:
\begin{lem}\label{lem:rescale}
  Let $0<\eps\leq \eps_1$, $h_\eps$ be as in the beginning of this subsection and satisfy
  \begin{equation*}
    M\coloneqq \sup_{0<\eps <\eps_1, (s,t)\in \Sigma\times [0,T_\eps]} |h_\eps (s,t)| <\infty
  \end{equation*}
  for some  $T_\eps \in (0,T_0]$, $\eps_1\in (0,1)$,  and   $(\tr_\eps)_{0<\eps<1}\in \mathcal{R}_{k,\tilde{\alpha}}$ for some $\tilde{\alpha}>0$, $k\in\R$ and let $j=1$ if $(\tr_\eps)_{0<\eps<1}\in \mathcal{R}_{k,\tilde{\alpha}}^0$ and $j=0$ else.
  Then there is some $C>0$, independent of $T_\eps,0<\eps\leq \eps_1$, $\eps_1\in (0,1)$  such that
  \begin{equation*}
    r_\eps (x,t)\coloneqq\tr_\eps\left( \rho, x,t\right)\qquad \text{for all }(x,t)\in\Gamma(2\delta)
  \end{equation*}
  with $\rho$ as in \eqref{eq:StretchedVariable} satisfies
  \begin{align}\label{eq:RemEstim1}
    \left\|{\mathfrak{a}(P_{\Gamma_t}(\cdot))r_\eps \varphi} \right\|_{L^1(\G_t(2\delta))} & \leq C(1+M)^j\eps^{1+k+j}\|\varphi\|_{H^1(\Omega)}\|\mathfrak{a}\|_{L^2(\Gamma_t)}, \\
    \label{eq:RemEstim2}
    \left\| {\mathfrak{a}(P_{\Gamma_t}(\cdot))r_\eps} \right\|_{L^2(\G_t( 2\delta))}       & \leq C (1+M)^j\eps^{\frac 12+k+j} \|\mathfrak{a}\|_{L^2(\Gamma_t)}
  \end{align}
  uniformly for all $\varphi\in H^1(\Omega)$, $\mathfrak{a}\in L^2(\Gamma_t)$, $t\in [0,T_\eps]$, and $\eps\in (0,\eps_1]$.
\end{lem}
We refer to \cite[Corollary 2.7]{StokesAllenCahn} for the proof. We note that in this paper the case $\Sigma=\T^1$ is considered. But the proof carries over directly to general $(d-1)$-dimensional reference manifolds $\Sigma$ and $d$ dimensions.

Moreover, we will use the following result.
\begin{lem}\label{lem:EstimMeanValueFree}
  Let the assumptions above hold true and $f\colon \Gamma_t(2\delta)\to \R$ be such that
    \begin{equation*}
    f(x,t)=
      a(\rho,s,t) w(s) \qquad \text{in } \Gamma_t(3\delta),\quad \text{where } x=X(r,s,t), \rho = \tfrac{r}\eps - h_\eps (s,t),
    \end{equation*}
    with $w\in L^2(\Sigma)$ and $a\in \mathcal{R}_{0,\tilde{\alpha}}$
    \begin{equation}\label{eq:MeanValueFree}
    \int_\R a(\rho,s,t) \theta'_0(\rho) \, \drho=0\qquad \text{for all }s\in \Sigma, t\in [0,T_0].
  \end{equation}
  Then there are constants $C(T_0)$, $\eps_1>0$ independent of $t\in [0,T_0]$ and $w$ such that
  \begin{equation*}
    \|f\|_{(\Veps(2\delta))'} \leq C \eps^{\frac32}\|w\|_{L^2(\Sigma)}
  \end{equation*}
  for every $\eps\in (0,\eps_1]$.
\end{lem}
\begin{proof}
  The result is a generalization of \cite[Lemma~2.11]{AbelsFei} in the case $d=2$ to higher space dimensions. The proof carries over in a straight-forward manner.
\end{proof}

\begin{lem}
	\label{lem:f'-phi-x-dependent}
	Let  $t\in[0,T_0]$ and $\rho=\rho(x,t)$ be as in \eqref{eq:StretchedVariable}. Moreover,  let $f\colon \R\times \Gamma_t(2\delta)\to \R$ be smooth and $g=\partial_\rho f$ and assume that there exist $\tilde\alpha>0$, $M > 0$ and $C>0$
	such that
	$\rho\in\mathbb R$,
	\begin{align}
		\label{ass:f-g-x-decay}
		\sum_{|\beta|\leq 2}|\partial_x^\beta f(\rho,x)|
		+
		\sum_{|\beta|\leq 2}|\partial_x^\beta g(\rho,x)|
		\leq C e^{-\tilde{\alpha} |\rho|} \quad\text{for all }x\in\Gamma_t(2\delta),\rho\in\R. 
	\end{align}
	Then there is some $C>0$ such that
	\begin{align}
		& \left|
                  \int_{\Gamma_t(2\delta)}g(\rho(x,t),x)
		\Psi(x)\sd x
		\right| \leq
		C\varepsilon^{\frac32}
		\|(\partial_{\no_{\Gamma_t}}\Psi,\Psi)\|_{L^2(\Gamma_t(2\delta))}
		\label{est:x-dependent-g-Psi}
		+
		Ce^{-\frac{2\tilde{\alpha}\delta}{\varepsilon}}
		\|\Psi\|_{H^1(\Gamma_t(2\delta))}.
	\end{align}
        for all $0<\varepsilon\leq 1$ and $\Psi\in H^1(\Gamma_t(2\delta))$. 
	Moreover, for every $\varphi\in H^1(\Gamma_t(2\delta))$ and
	$w\in L^2(\Gamma_t)$
	\begin{align}
		\notag
		& \left|
                  \int_{\Gamma_t(2\delta)}g(\rho(x,t),x)
		\varphi(x)w(P_{\Gamma_t}x)\sd x
		\right|\\
		& \quad
		\leq
		C\varepsilon^{\frac32}\|(\partial_{\no_{\Gamma_t}}\varphi,\varphi)\|_{L^2(\Gamma_t(2\delta))}
		\|w\|_{L^2(\Gamma_t)}
		+
		Ce^{-\frac{2\tilde{\alpha}\delta}{\varepsilon}}
		\|\varphi\|_{H^1(\Gamma_t(2\delta))}
		\|w\|_{L^2(\Gamma_t)}.
		\label{est:x-dependent-g-phi-w}
	\end{align}
\end{lem}
\begin{proof}
We prove the first estimate \eqref{est:x-dependent-g-Psi}, while the second estimate \eqref{est:x-dependent-g-phi-w} follows the same way just with an additional factor $w(P_{\Gamma_t}x)$ or $\tilde{w}(s)$, which is independent of $r$.
We first write
\begin{align*}
\int_{\Gamma_t(2\delta)}g(\rho,x)
		\Psi(x)\sd x &=    \int_\Sigma \int_{-2\delta}^{2\delta} (\partial_\rho f)\left(X(r,s),\tfrac{r}{\eps}-h_\eps(s,t)\right)\tilde{\Psi}(r,s)J_t(r,s)\dr\ds.
\end{align*}
Using
\begin{equation*}
  (\partial_\rho f)\left(X(r,s),\tfrac{r}{\eps}-h_\eps(s,t)\right)= \eps \partial_r f\left(X(r,s),\tfrac{r}{\eps}-h_\eps(s,t)\right)- \eps\no_{\Gamma_t}\cdot (\nabla_x f)\left(X(r,s),\tfrac{r}{\eps}-h_\eps(s,t)\right) 
\end{equation*}
and integration by parts yields
\begin{align*}
  &\int_\Sigma \int_{-2\delta}^{2\delta} (\partial_\rho f)\left(X(r,s),\tfrac{r}{\eps}-h_\eps(s,t)\right)\tilde{\Psi}(r,s)J_t(r,s)\dr\ds\\
  &= \left.\eps\left[\int_\Sigma f\left(X(r,s),\tfrac{r}{\eps}-h_\eps(s,t)\right)\tilde{\Psi}(r,s)J_t(r,s)\ds\right]\right|_{r=-2\delta}^{2\delta}\\
  &\quad -\eps \int_\Sigma \int_{-2\delta}^{2\delta} f\left(X(r,s),\tfrac{r}{\eps}-h_\eps(s,t)\right)\partial_r(\tilde{\Psi}(r,s)J_t(r,s))\dr\ds\\
  &\quad +\eps \int_\Sigma \int_{-2\delta}^{2\delta} \no_{\Gamma_t}\cdot (\nabla_x f)\left(X(r,s),\tfrac{r}{\eps}-h_\eps(s,t)\right)\tilde{\Psi}(r,s)J_t(r,s)\dr\ds.
\end{align*}
Now we can use the exponential decay to estimate the first term on the right-hand side and \eqref{eq:RemEstim2} to estimate the last two terms. 
\end{proof}
\begin{lem}
	\label{lem:2delta-Gamma_t}
Let $g\colon \R\to \R$ be continuous such that $|g(\rho)|\leq Ce^{-\tilde\alpha |\rho|}$ for all $\rho\in\R$ and some $C,\tilde\alpha>0$ and let $m:= \int_\R g(\rho)\sd \rho$. There are  $ C > 0$, $\eps_1\in (0,1]$ such that
	\begin{align}	\label{eq:delta-gamma}
		&\left|\int_{\Gamma_t(2\delta)}
		\frac{1}{\varepsilon} g(\rho)a(x,t) \varphi(x,t) \dx - m\int_{\Gamma_t}(a\varphi)|_{\Gamma_t} \dsigma \right|\leq C\sqrt{\eps} \|a \|_{H^1(\Gamma_t(2\delta))}\|\varphi \|_{H^1(\Gamma_t(2\delta))}, %\\
            \end{align}
            for all $\eps\in (0,\eps_1]$, $ t \in [0,T_0] $,  $ a \in H^1(\Gamma_t(2\delta)) $, and $ \varphi \in H^1(\Gamma_t(2\delta)) $.
\end{lem}
\begin{proof} 
  Using $m= \int_\R f(\rho)\sd \rho$, it follows
	\begin{align*}
		&m \int_{\Gamma_t} a \varphi|_{\Gamma_t} \dsigma 
		 = \frac{1}{\varepsilon} \int_\Sigma \int_\bbr g\big(\tfrac{r - \varepsilon h_\varepsilon}{\varepsilon}\big) \,a(0,s,t) \,\varphi(0,s,t) J(0,s,t) \dr \ds \\
		&\quad  = \frac{1}{\varepsilon} \int_\Sigma \int_{-2\delta}^{2\delta} g\big(\tfrac{r - \varepsilon h_\varepsilon}{\varepsilon}\big) \Big((a\varphi J)(0,s,t)
		- (a\varphi J)(r,s,t)\Big) \dr \ds \\
		& \qquad + \frac{1}{\varepsilon} \int_\Sigma \int_{-2\delta}^{2\delta} g\big(\tfrac{r - \varepsilon h_\varepsilon}{\varepsilon}\big) (a\varphi J)(r,s,t) \dr \ds + \frac{1}{\varepsilon} \int_\Sigma \int_{\bbr \setminus (-2\delta,2\delta)} g\big(\tfrac{r - \varepsilon h_\varepsilon}{\varepsilon}\big) \,(a\varphi J)(0,s,t) \dr \ds.
	\end{align*}
	Using $ |\psi(r) - \psi(0)| = \left|\int_0^r \partial_r \psi(\tilde{r}) \d\tilde{r}\right|\leq \sqrt{|r|} \|\partial_r\psi\|_{L^2(-2\delta,2\delta)}  $ we obtain
	\begin{align*}
		& \abs{m \int_{\Gamma_t} (a\varphi)|_{\Gamma_t} \dsigma
			- \frac{1}{\varepsilon} \int_{\Gamma_t(2\delta)} g(\rho) a(x,t)\,\varphi(x,t) \dx} \\
		& \quad \leq
		\abs{\frac{1}{\varepsilon} \int_\Sigma \int_{-2\delta}^{2\delta} g\big(\tfrac{r - \varepsilon h_\varepsilon}{\varepsilon}\big) \sqrt{\absm{r}} \dr \norm{\ptial{r} (a\varphi J)(\cdot,s,t)}_{L^2(-2\delta,2\delta)} \ds} \\
		& \qquad + \abs{\frac{1}{\varepsilon} \int_{\Gamma_t} \int_{\bbr \setminus (-2\delta,2\delta)} g\big(\tfrac{r - \varepsilon h_\varepsilon}{\varepsilon}\big) (a\varphi)|_{\Gamma_t} \dr \dsigma}
		\eqqcolon I_1 + I_2.
	\end{align*}
	By $r = \varepsilon(\rho + h_\varepsilon)$, one has
	\begin{align*}
		I_1 & \leq
		\frac{C}{\varepsilon} \int_\Sigma \int_{-2\delta}^{2\delta} \abs{g\big(\tfrac{r - \varepsilon h_\varepsilon}{\varepsilon}\big)} \sqrt{\absm{r}} \dr \norm{\ptial{r} (a\varphi J)(\cdot,s,t)}_{L^2(-2\delta,2\delta)} \ds \\
		& \leq C
		\sqrt{\varepsilon} \int_\Sigma \int_{- \frac{2\delta}{\varepsilon} - h_\varepsilon}^{\frac{2\delta}{\varepsilon} - h_\varepsilon} \abs{g(\rho)} \sqrt{|\rho + h_\varepsilon|} \drho \norm{\ptial{r} (a\varphi J)(\cdot,s,t)}_{L^2(-2\delta,2\delta)} \ds \\
        & \leq C' \sqrt{\varepsilon} \norm{a}_{H^1(\Gamma_t(2\delta))} \norm{\varphi}_{H^1(\Gamma_t(2\delta))},
	\end{align*}
	where we employed that $ \theta_0'(\rho) $ exhibits exponential decay. Moreover, we find 
	\begin{align}
		I_2 & \leq \frac{1}{\varepsilon} \int_{\Gamma_t} \abs{\int_{\bbr \setminus (-2\delta,2\delta)} g\big(\tfrac{r - \varepsilon h_\varepsilon}{\varepsilon}\big) \dr} \abs{(a\varphi)|_{\Gamma_t}} \dsigma \nonumber\\
		& 
		\leq \frac{C}{\varepsilon} \int_{\Gamma_t} \abs{\int_{\bbr \setminus (-2\delta,2\delta)} e^{- \frac{\tilde{\alpha} \absm{r}}{2 \varepsilon}} \dr} \abs{(a\varphi)|_{\Gamma_t}} \dsigma e^{- \frac{\tilde{\alpha} \delta}{\varepsilon}} 
		\leq C e^{- \frac{\tilde{\alpha} \delta}{\varepsilon}} \norm{a}_{H^1(\Gamma_t(2\delta))} \norm{\varphi}_{H^1(\Gamma_t(2\delta))},
	\nonumber
	\end{align} 
        for sufficiently small $\eps\in (0,1]$, which shows \eqref{eq:delta-gamma}.
      \end{proof}

      \subsection{Approximate Solutions}\label{subsec:ApproxSol}

      For the proof of our main result on the sharp interface limit, Theorem~\ref{thm:main}, we use an approximate solution $(\ve_A, p_A, c_A)$ of the Navier--Stokes/Allen--Cahn system~\eqref{eq:NSAC} and corresponding remainder estimates essentially. These can be constructed and estimated as in \cite[Theorem 3.2]{AbelsFei}. We note that the convergence result in \cite{AbelsFei} is only valid in two space dimensions. But the construction of the approximate solution and the remainder estimates work for three space dimensions as well. To fit our higher-order estimates close to the interface, we include the corresponding weighted higher order remainders estimates.
\begin{thm}\label{thm:remainders}
	Let $N\in \N$. Then there are smooth $(\tilde{\ve}_A^{\mathrm{in}},\tilde{p}_A^{\mathrm{in}}, \tilde{c}_A^{\mathrm{in}})$ defined in $\Gamma(3\delta)$ and $(\ve_A^\pm, p_A^\pm, c_A^\pm)$ defined on $\Omega\times [0,T_0]$, which are smooth,  such that:
	\begin{enumerate}
		\item \emph{Inner expansion:} In $\Gamma(3\delta)$ we have
		\begin{alignat}{2}\nonumber
			\partial_t \tilde{\ve}_A^{\mathrm{in}} + \tilde{\ve}_A^{\mathrm{in}}\cdot \nabla \tilde{\ve}_A^{\mathrm{in}} -\Div (2\nu(\tilde{c}_A^{\mathrm{in}})D \tilde{\ve}_A^{\mathrm{in}}) +\nabla \tilde{p}_A^{\mathrm{in}}&= -\eps \Div (\nabla \tilde{c}_A^{\mathrm{in}}\otimes \nabla \tilde{c}_A^{\mathrm{in}}) + \bR_\eps^{\mathrm{in}},
			\\\nonumber
			\Div \tilde{\ve}_A^{\mathrm{in}} &=\,  
			G_\eps^{\mathrm{in}},\\\nonumber
			\partial_t \tilde{c}_A^{\mathrm{in}} + \tilde{\ve}_A^{\mathrm{in}}\cdot \nabla \tilde{c}_A^{\mathrm{in}} &= \Delta \tilde{c}_A^{\mathrm{in}} -\frac1{\eps^2} f'(\tilde{c}_A^{\mathrm{in}})+ S_\eps^{\mathrm{in}},
		\end{alignat}
		where
		\begin{alignat}{2}\label{eq:RemainderApprox1}
                  \|(\bR_\eps^{\mathrm{in}}, \partial_t G_\eps^{\mathrm{in}} , \nabla G_\eps^{\mathrm{in}},  S_\eps^{\mathrm{in}})\|_{L^\infty(\Gamma(3\delta))}&\leq C\eps^{N+1},\\\label{eq:RemainderApprox2}
                  \|G_\eps^{\mathrm{in}} \|_{L^\infty(\Gamma(3\delta))}&\leq C\eps^{N+2},\\\label{eq:RemainderApprox3}				\|(\oeps^2 \nabla S_\eps^{\mathrm{in}},\oeps \nabla G_\eps^{\mathrm{in}})\|_{L^\infty(\Gamma(3\delta))}
				&\leq C\eps^{N+2}.
		\end{alignat}
		\item \emph{Outer expansion:}  In $\Omega^\pm$ we have $c_A^\pm \equiv \pm 1$ and
		\begin{alignat}{2}
			\partial_t \ve_A^\pm + \ve_A^\pm\cdot \nabla \ve_A^\pm -\nu^\pm \Delta \ve_A^\pm +\nabla p_A^\pm&= \bR^\pm_\eps,
			\\\nonumber
			\Div \ve_A^\pm &=\, 0,\\
			\ve_A^\pm{|_{\partial\Omega}} &=\, \ol{a}_\eps \no_{\partial\Omega},
		\end{alignat}
		where $\ol{a}_\eps\colon (0,T)\to \R$ is continuous and
		\begin{equation*}
			\|\bR_\eps^\pm \|_{L^\infty(\Omega\times [0,T_0])}\leq C\eps^{N+2}\qquad \text{for all }\eps\in (0,1).
		\end{equation*}
		\item \emph{Matching condition:} There is some $\tilde{\alpha}>0$ such that for every $\beta\in\N_0^n$ there is some $C_\beta>0$ such that 
		\begin{equation*}
			\begin{split}
				&\| \p_x^{\beta}( \tilde v_A^{\mathrm{in}}-v_A^{+}
				\chi_+-v_A^{-}\chi_-)\|_{L^\infty(\Gamma(3\delta)\setminus \Gamma(\delta) )}\leq C_\beta e^{-\frac{\tilde{\alpha}\delta}{2\eps}}%,\\
			\end{split}
		\end{equation*}
		for $v=\ve,p,c$ and for all  $\eps \in (0,1]$.
	\end{enumerate}
\end{thm}
\begin{proof}
	The estimates \eqref{eq:RemainderApprox1} and \eqref{eq:RemainderApprox2} were shown in \cite{AbelsFei}. The proof of \eqref{eq:RemainderApprox3} can be justified naturally, since close to the interface, the gradient (more precisely the normal derivative), produces one order loss of $ \eps $, while the factor $ \oeps $ can compensate for it with the help of the (exponential) matching condition. 
\end{proof}
\begin{rem}
  We note that the previous approximate solution was denoted by $(\tilde{\ve}_A, \tilde{p}_A, \tilde{c}_A)$ in \cite{AbelsFei}. There it was essential that some extra (leading order) terms were added to it to obtain the approximate solution (denoted by $({\ve}_A, {p}_A, {c}_A)$ in \cite{AbelsFei}), which then was used to estimate the error $u=c_\eps-c_A$, $\we= \ve_\eps-\ve_A$. (See \cite[Section~3]{AbelsFei} for the details.) Actually, these extra terms depend on the error $u$. In the present contribution the corresponding terms are not treated as part of the approximate solution. Instead they are present as a leading order part $u_A$ of the solution of the linearized system \eqref{eq:linNSAC}, see Theorem~\ref{thm:FullLinearizedSystem} below. 
\end{rem}
\begin{thm}\label{thm:ApproximateSolution}
There are smooth $c_A,p_A\colon \Omega\times [0,T_0]\to \R$, $\ve_A\colon \Omega\times [0,T_0]\to \R^d$, defined in \eqref{eq:DefcA}-\eqref{eq:DefpA} below,  such that
  \begin{alignat*}{2}
    \partial_t \ve_A+ \ve_A\cdot \nabla \ve_A -\Div (2\nu(c_A)D \ve_A) +\nabla p_A&= -\eps \Div (\nabla c_A\otimes \nabla c_A) + \bR_\eps &\quad &\text{in }Q_{T_0},
\\
      \Div \ve_A &=\,  
                        G_\eps&\quad &\text{in }Q_{T_0},\\
      \partial_t c_A + \ve_A\cdot \nabla c_A &= \Delta c_A -\frac1{\eps^2} f'(c_A)
      + S_\eps&\quad &\text{in }Q_{T_0},\\
      (\ve_A,c_A)|_{\partial\Omega}&= (0,-1)&\quad &\text{on }S_{T_0},
\end{alignat*}
where $S_\eps|_{\partial\Omega}=0$ and
\begin{alignat}{2}\label{eq:RemainderApprox1'}
			\|(\bR_\eps, \partial_t G_\eps , \nabla G_\eps,  S_\eps)\|_{L^\infty((0,T_0)\times\Omega)}&\leq C\eps^{N+1},\\\label{eq:RemainderApprox2'}
			\|(G_\eps,\oeps^2 \nabla S_\eps,\oeps \nabla G_\eps) \|_{L^\infty((0,T_0)\times\Omega)}&\leq C\eps^{N+2}.
\end{alignat}
\end{thm}
\begin{proof}
  The proof follows the arguments of the proof of \cite[Theorem~3.1]{AbelsFei} in the simple case that $\ue\equiv 0$.   As in \cite[Section~3]{AbelsFei} we define 
  \begin{alignat}{1}\label{eq:DefcA}
  c_{A}(x,t)&=\zg c_A^{\mathrm{in}}(x,t)+(1-\zg )(c_A^+
  \chi_{\Omega^+}(x,t) +c_A^-\chi_{\Omega^.}(x,t) ),\\\label{eq:DefvA}
     \ve_A(x,t)&=\zg \ve_A^{\mathrm{in}}(x,t)+(1-\zg )\left(\ve_A^{+}(x,t)
  \chi_{\Omega^+}(x,t) +\ve_A^-(x,t)\chi_{\Omega^-}(x,t)\right){- {\mathbf{N}}\bar{a}_\eps(t)},\\\label{eq:DefpA}
  p_A(x,t)&=\zg p_A^{{in}}(x,t)+(1-\zg )
  \left(p_A^{+}(x,t)
  \chi_{\Omega^+}(x,t) +p_A^{-}(x,t)\chi_{\Omega^-}(x,t) \right),
\end{alignat}
where $\rho$ is as in \eqref{eq:StretchedVariable}, and
\begin{alignat*}{1}
  c_A^{\mathrm{in}}(x,t)= \tc_A^{\mathrm{in}}(\rho,x,t)&= \theta_0(\rho) + \sum_{k=2}^{N+2}\eps^k \tc_k(\rho,x,t),\\
 \ve_A^{\mathrm{in}}(x,t)= \tv_A^{\mathrm{in}}(\rho,x,t)&= \sum_{k=0}^{N+2}\eps^k \tv_k(\rho,x,t),\\
  p_A^{\mathrm{in}}(x,t)= \tp_A^{\mathrm{in}}(\rho,x,t)&=\sum_{k=-1}^{N+1}\eps^k \tp_k(\rho,x,t),
\end{alignat*}
and $\mathbf{N}\colon \Omega \to \R^d$ is a smooth vector field such that $\mathbf{N}|_{\partial \Omega}= \no_{\partial\Omega}$. We note that $c_A^\pm \equiv \pm 1$ and $\ve_A^\pm = \sum\limits_{k=0}^{N+2}\eps^k \ve_k^\pm$, $p_A^\pm = \sum\limits_{k=0}^{N+2}\eps^k p_k^\pm$, where $\ve_0^\pm = \ve^\pm$, $p_0^\pm =p^\pm$.
Then we have
    \begin{align}\nonumber
      &\Div (\ve_A(x,t) + \mathbf{N}\bar{a}_\eps(t))\\\nonumber
      &\quad = \zeta(d_\Gamma(x,t)) G_\eps^{\mathrm{in}} + \zeta'(d_\Gamma(x,t))\nabla d_\Gamma(x,t)\cdot \left(\ve_A^{\mathrm{in}} - \ve^+_A(x,t)\chi_{\Omega^+(t)}(x) - \ve^-_A(x,t)\chi_{\Omega^-(t)}(x)\right)\\\nonumber
      &\quad =O(\eps^{N+2})
    \end{align}
    in $L^\infty(Q_{T_0})$. Therefore
    \begin{align*}
      \sup_{0\leq t\leq T_0}|\bar{a}_\eps(t)|  &= \sup_{0\leq t\leq T_0} \frac1{\mathcal{H}^{d-1}(\partial\Omega)} \left|\int_{\partial \Omega}(\ve_A(x,t) + \mathbf{N}\bar{a}_\eps(t))\cdot \no_{\partial\Omega}\, d\sigma(x)\right|\\  
      &=\sup_{0\leq t\leq T_0} \frac1{\mathcal{H}^{d-1}(\partial\Omega)}\left|\int_{\Omega}\Div (\ve_A + \mathbf{N}\bar{a}_\eps) \,\dx\right| \leq C\eps^{N+2}.
    \end{align*}
    and similarly, using $\|\partial_t G_\eps^{\mathrm{in}}\|_{L^\infty(\Gamma(3\delta))}= O(\eps^{N+1})$,
    \begin{align*}
      \sup_{0\leq t\leq T_0}|\partial_t\bar{a}_\eps(t)|  &\leq C\eps^{N+1}.
    \end{align*}
    Hence we have $\ve_A|_{\partial\Omega}=0$ by construction and
    \begin{equation*}
      \Div \ve_A= \Div(\ve_A+ \mathbf{N}\bar{a}_\eps(t))- (\Div \mathbf{N})\bar{a}_\eps(t)= O(\eps^{N+2})\quad \text{in }L^\infty(Q_{T_0}).
    \end{equation*}
The rest of the statements can be verified in a straightforward manner using the estimates from Theorem~\ref{thm:remainders} and the estimate for $\bar{a}_\eps$. We note that $S_\eps|_{\partial\Omega}=0$ since $c_A\equiv -1$ in neighborhood of $\partial\Omega$ and $f'(-1)=0$.
\end{proof}
\begin{rem}\label{rem:PropertiesAproxSol}
	The following properties of $\theta_0$, $d_\Gamma$, $\hc_2$, $\hv_0, \hv_1$, $\hp_0$ will be essential for the results for the full linearized system \eqref{eq:linNSAC} in Section~\ref{sec:FullSystem}. 
	\begin{enumerate}
		\item $\theta_0$ solves \eqref{eq:OptProfile}.
		\item $\partial_t d_\Gamma + \ve^\pm \cdot \nabla d_\Gamma -\Delta d_\Gamma = -V_{\Gamma_t}+\no_{\Gamma_t}\cdot \ve^\pm+H_{\Gamma_t}=0$ on $\Gamma$.
		\item $\hc_2\colon \R\times \Gamma(2\delta)\to  \R$ is bounded and solves
		\begin{align}\nonumber
			&-\partial_\rho^2 \hc_2(\rho,x,t) + f''(\theta_0(\rho))\hc_2(\rho,x,t)\\\label{eq:c2}
			&\qquad =|\nabla^\Gamma h_1(r,s,t)|^2 \theta_0''(\rho)-\theta_0'(\rho) \rho g_0(x,t)=: \hat{a}(\rho,x,t)
		\end{align}
		for all $\rho\in\R, (x,t)\in \Gamma(2\delta)$, where $s=S(x,t)$, $r=d_\Gamma$, and
		\begin{equation*}
			g_0(x,t)=
			\begin{cases}
				-   \frac1{d_\Gamma}\left(\partial_t d_\Gamma + \ve\cdot \nabla d_\Gamma -\Delta d_\Gamma\right)&\text{if }(x,t)\not \in \Gamma,\\
				-\no_{\Gamma_t}\cdot \nabla\left(\partial_t d_\Gamma + \ve\cdot \nabla d_\Gamma -\Delta d_\Gamma\right) &\text{else}.
			\end{cases}
		\end{equation*}
		Moreover, there is some $\tilde\alpha>0$ such that for every $k,m\in\N_0,l\in\N_0^n,$ there is some $C_{k,l,m}>0$ satisfying
		\begin{equation*}
                  |\partial_t^k \partial_x^l \partial_\rho^m \hc_2(\rho,x,t)| \leq C_{k,l,m} e^{-\tilde\alpha|\rho|}\quad \text{for all } \rho \in\R, (x,t)\in \Gamma(2\delta).
                \end{equation*}
              \item $\hp_0\colon \R\times \Gamma(2\delta)\to \R$, $\hv_j\colon \R\times \Gamma(2\delta)\to \R^d$, $j=0,1$ are smooth and $\hv_0$ is given by
                \begin{align*}
                  \hv_0(\rho,x,t)&= \ve^+(x,t)\eta(\rho)+ \ve^-(x,t)(1-\eta(\rho))= \tfrac{\ve^+(x,t)+\ve^-(x,t)}2+d_\Gamma (\eta(\rho)-\tfrac12)\ue_0(x,t)\
                \end{align*}
                for all $\rho\in \R$, $(x,t)\in\Gamma (2\delta)$ with $\eta\colon \R\to[0,1]$ satisfying $\eta=0$ in $(-\infty,-1]$ and $\eta=1$ in $[1,\infty)$, and  
                \begin{align*}   \ue_0(x,t) &=
            	\begin{cases}
            		\frac{\ve^+(x,t)-\ve^-(x,t)}{d_\Gamma(x,t)}&\text{if }(x,t) \notin \Gamma,\\
            		\no_{\Gamma_t} \cdot \nabla (\ve^+(x,t)-\ve^-(x,t)) &\text{if }(x,t)\in \Gamma,
            	\end{cases}
            \end{align*}
                cf.\ \cite[(A.54)]{AbelsFei}.    Here $\ve^\pm$ is extended smoothly to $\Omega$ such that $\Div \ve^\pm =0$ in $\Omega$. In particular, we have
                \begin{align} 
                  \partial_\rho \left(\nu(\theta_0(\rho))\big(\partial_\rho \hv_0(\rho,x,t)-\ue_0 d_\Gamma(x,t)\eta'(\rho)\big)\right)&= 0\label{eq:A.42}\\\label{eq:A.44}
                  \left(\partial_\rho \hv_0(\rho,x,t)-\ue_0(x,t)d_\Gamma(x,t)\eta'(\rho)\right)\cdot\nabla d_\Gamma &=0
                \end{align}
                for all $\rho \in\R$, $(x,t)\in\Gamma(2\delta)$, cf.\ \cite[(A.42), (A.44)]{AbelsFei}, where we note that $\hp_{-1}(\rho,x,t)= -\tfrac12 (\theta_0'(\rho))^2$ and therefore the right-hand side of the first equation vanishes. Equations for $\hv_1$, which corresponds to \eqref{eq:A.42} and \eqref{eq:A.44}, are given by \cite[(A.43) together with (A.62)]{AbelsFei} and \cite[(A.47)]{AbelsFei} with $k=1$. These will only be used once in the following. 
                        \item  For $\hv_A^{\mathrm{in}}(\rho,x,t)= \sum_{k=0}^{N+2}\eps^k\hv_k(\rho,x,t)$ for the definition of $\ve_A^{\mathrm{in}}(x,t)$ we have
                \begin{equation*}
                  \tfrac{1}\eps \partial_\rho \hv_A^{\mathrm{in}}(\rho,x,t)\cdot \nabla (d_\Gamma(x,t)- h_\eps(S(X,t)))+(\Div_x \hv_A^{\mathrm{in}})(\rho,x,t) =\hat{g}_\eps(\rho,x,t)
                \end{equation*}
                for all $\rho\in\R, (x,t)\in\Gamma(2\delta)$, where $\hat{g}_{\eps}=O(\eps^{N+2})$ as $\eps\to$ uniformly in $\rho\in\R,(x,t)\in \Gamma(2\delta)$. This follows from the construction in \cite[Appendix]{AbelsFei} and leads to $G_\eps = O(\eps^{N+2})$ in $L^\infty(\Gamma(2\delta))$. Differentiation with respect to $\rho$ yields
                \begin{equation}\label{eq:InnerDivergence}
                  \tfrac{1}\eps \partial_\rho^2 \hv_A^{\mathrm{in}}(\rho,x,t)\cdot \nabla (d_\Gamma(x,t)- h_\eps(S(x,t)))+(\partial_\rho \Div_x \hv_A^{\mathrm{in}})(\rho,x,t) =\partial_\rho \hat{g}_\eps(\rho,x,t),
                \end{equation}
                where $|\partial_\rho^k \partial_t^k\partial_x^\alpha \hat{g}_\eps(\rho,x,t)|\leq C_{k,l,\alpha}\eps^{N+2}e^{-\frac{\tilde{\alpha}|\rho|}\eps}$ for any $k\in\N$, $l\in\N_0$, $\alpha\in\N_0^d$ and some $\tilde{\alpha}>0$. Actually, $\partial_\rho \hat{g}_\eps(\rho,x,t)$ is a polynomial in $\eps$ with powers larger or equal to $N+2$ and coefficients, depending on $\rho\in\R$ and $(x,t)\in\Gamma(2\delta)$, decaying uniformly exponentially to zero as $|\rho|\to \infty$ since all functions involved have the latter properties. The same is true for all $\rho$ and $(x,t)$ derivatives.
              \end{enumerate}    
            \end{rem}

The following decomposition of $f''(c_A)$ will be essential for $\oeps$-weighted estimates.
\begin{lem}
	\label{lem:f''-decomposition}
	Let $c_A$ be as before. Then
	\begin{align}\label{eq:f''-decomposition}
		f''(c_A)= \alpha + V(\rho) + a_\eps\qquad \text{in }\Gamma(2\delta),
	\end{align}
	where $V(\rho)=f''(\theta_0(\rho)) - \alpha$ for all $\rho\in\R$,
        $\alpha=f''(\pm 1)$, $a_\eps = O(\eps^2)$ in $L^\infty(\Gamma(2\delta))$, and $ \nabla a_\eps = O(\eps^{1+1/p}) $ in $ L^p(\Gamma(2\delta)) $ for $ 1 \leq p \leq \infty $. Moreover,  $|V(\rho)|\leq Ce^{-\tilde{\alpha}|\rho|}$ for all $\rho\in\R$ for some $\tilde\alpha>0$.
\end{lem}
\begin{proof}
  We use Taylor's theorem to expand $ f''(c_A) $ around $\theta_0$  such that
	\begin{align*}
		f''(c_A) = f''(\theta_0) +  b_\eps(x,t)((1-\zg)(\chi_{\Omega^+}-\chi_{\Omega^-}-\theta_0)+ \eps^2\zg\hc_2+O(\eps^3))
        \eqqcolon f''(\theta_0) + a_\eps, 
	\end{align*}
        where
        \begin{equation*}
          b_\eps(x,t)\coloneqq  \int_0^1 f'''(\theta+\tau((1-\zg)(\chi_{\Omega^+}-\chi_{\Omega^-}-\theta_0)+ \eps^2\zg\hc_2+O(\eps^3)) )\sd \tau 
        \end{equation*}
	Now for $ V(\rho) \coloneqq f''(\theta_0(\rho)) - \alpha $ with $\alpha=f''(\pm 1)$, we know
	\begin{align*}
		|V(\rho)|\leq Ce^{-\sqrt{\alpha}|\rho|} \qquad \text{for all }\rho\in\R.
	\end{align*}
        Using $\theta_0(\rho)\to \pm 1$ as $\rho\to \pm \infty$ exponentially and the structure of the terms, one can verify 
	$ a_\eps = O(\eps^2) $ in $ L^\infty(Q_{T_0}) $ and $ \nabla a_\eps = O(\eps^{1+1/p}) $ in $ L^p(Q_{T_0}) $ for $ 1 \leq p \leq \infty $ in a straightforward manner.
      \end{proof}

\subsection{Weighted Elliptic Regularity}

For uniform regularity estimates related to $\oeps$ the following elliptic regularity result will be useful:
\begin{prop}\label{prop:weightedEllipticReg}
  Let $k\in\N$, $f\in L^2(\Omega)$ and $u\in H^1_0(\Omega)$ be a weak solution of
  \begin{alignat*}{2}
    -\Delta u &=f &\qquad &\text{in }\Omega,\\
    u|_{\partial\Omega} &=0  &&\text{on }\partial\Omega.
  \end{alignat*}
  Then $u\in H^2(\Omega)$ and there is a constant $C_k>0$ independent of $\eps\in (0,1]$, such that
  \begin{equation}\label{eq:weightedH2Estim}
    \|\oeps^k\nabla^2 u\|_{L^2(\Omega)}\leq C_k\left(\|u\|_{L^2(\Omega)}+\|\oeps^{k-1}\nabla u\|_{L^2(\Omega)}+ \|\oeps^k f\|_{L^2(\Omega)}\right)
  \end{equation}
\end{prop}
\begin{proof}
  First of all $u\in H^2(\Omega)$ follows for every fixed $\eps\in (0,1]$ from standard elliptic regularity results. Hence it only remains to prove \eqref{eq:weightedH2Estim}. Let $\delta'>0$ be so small that $2\delta'<\dist \left(\partial\Omega, \bigcup_{t\in [0,T_0]}\ol{\Gamma_t(3\delta)}\right)$ and define
  \begin{equation*}
    \Omega_{\kappa} = \{x\in\Omega: \dist (x,\partial\Omega) <\kappa\}\qquad \text{for }\kappa\in (0,2\delta'].
  \end{equation*}
  Then, by standard elliptic regularity results, there is some $C_{\delta'}>0$ such that
  \begin{align*}
    \|\oeps^k \nabla^2 u\|_{L^2(\Omega_{\delta'})}&\leq C\|\nabla^2 u\|_{L^2(\Omega_{\delta'})}\leq C_{\delta'}\left(\|f\|_{L^2(\Omega_{2\delta'})}+\|u\|_{L^2(\Omega)}\right)\\
    &\leq  C_{\delta'}\left(\|\oeps^kf\|_{L^2(\Omega)}+\|u\|_{L^2(\Omega)}\right)
  \end{align*}
  since $\oeps\equiv\sqrt{\eps^2+1}\in [1,2]$ in $\Omega\setminus \ol{\Gamma_t(3\delta)}$. Now we choose some $\eta\in C^\infty_0(\Omega)$ with $\eta \equiv 1$ on $\Omega\setminus \Omega_{\delta'}$. Then testing the equation with $\eta \oeps^{2k} \Delta u$,  and integrating by parts yields
  \begin{align*}
    \int_\Omega \eta \oeps^{2k} |\nabla^2 u|^2\,\dx &= -\int_\Omega\oeps^{k} f\oeps^k \Delta u\,\dx +\int_\Omega \Delta u \nabla (\eta \oeps^{2k})\cdot \nabla u \,\dx - \int_\Omega \nabla^2 u : \nabla (\eta \oeps^{2k})\otimes \nabla u\,\dx   \\
    &\leq C\left(\|\oeps^k f\|_{L^2(\Omega)}+ \|\oeps^{k-1}\nabla u\|_{L^2(\Omega)} \right)\|\oeps^k \nabla^2u\|_{L^2(\Omega)}
  \end{align*}
  due to $|\nabla(\eta \oeps^{2k})|\leq C \oeps^{2k-1}$. Altogether this implies \eqref{eq:weightedH2Estim}.
\end{proof}

The following weighted estimates for a special solution of the divergence equation will be important to deal with an inhomogeneous divergence.
\begin{lem}[Divergence Equation in Weighted Spaces]
	\label{lem:weighted-Dirichlet-Stokes-lifting}
	Let $t\in[0,T_0]$, $\varepsilon\in(0,1]$, and let $g\in H^1(\Omega)\cap L^2_{(0)}(\Omega)$.
	Suppose that $\mathbf z_g\in H^2(\Omega)^d$, $\pi_g\in H^1_{(0)}(\Omega)$ is the solution of
	\begin{subequations}
		\label{eq:Dirichlet-Stokes-lifting}
		\begin{alignat}{3}
			-\Delta\mathbf z_g+\nabla\pi_g&=0
			&&\quad \text{in }\Omega,  \\
			\Div\mathbf z_g&=g
			&&\quad \text{in }\Omega,\\
			\mathbf z_g&=0
			&&\quad \text{on }\partial\Omega.
		\end{alignat}
	\end{subequations}
	Then there is some $C>0$ that is independent of
	$\varepsilon\in(0,1]$ and $t\in[0,T_0]$ such that
	\begin{align}
		\|\mathbf z_g\|_{L^2(\Omega)}
		&\leq
		C\|g\|_{H^{-1}_{(0)}(\Omega)},\label{eq:negative-Stokes-lifting}\\
		\|\nabla\mathbf z_g\|_{L^2(\Omega)}
		+
		\|\pi_g\|_{L^2(\Omega)}
		&\leq
		C\|g\|_{L^2(\Omega)},
		\label{eq:Dirichlet-Stokes-lifting-low-order}
		\\
		\|\omega_\varepsilon\nabla^2\mathbf z_g\|_{L^2(\Omega)}
		+
		\|\omega_\varepsilon\nabla\pi_g\|_{L^2(\Omega)}
		&\leq
		C\left(
		\|g\|_{L^2(\Omega)}
		+
		\|\omega_\varepsilon\nabla g\|_{L^2(\Omega)}
		\right).
		\label{eq:Dirichlet-Stokes-lifting-weighted}
	\end{align}
	If, in addition,
	\begin{align}
		g=\Div \mathbf a
		\qquad\text{for some }\mathbf a\in H^1_0(\Omega)^d,
		\label{eq:divergence-representation}
	\end{align}
	then
	\begin{align}
		\|\omega_\varepsilon\nabla\mathbf z_g\|_{L^2(\Omega)}
		\leq
		C\left(
		\|\mathbf a\|_{L^2(\Omega)}
		+
		\|\omega_\varepsilon
		\nabla\mathbf a\|_{L^2(\Omega)}
		\right).
		\label{eq:weighted-first-order-lifting}
	\end{align}
\end{lem}
\begin{proof}
	To prove \eqref{eq:negative-Stokes-lifting}, let
	$\mathbf F\in L^2(\Omega)^d$ and let
	$(\boldsymbol\psi,\chi)\in
	(H^2(\Omega)\cap H^1_0(\Omega))^d\times H^1_{(0)}(\Omega)$ solve
	\begin{alignat*}{3}
		-\Delta\boldsymbol\psi+\nabla\chi&=\mathbf F, \quad
		\Div\boldsymbol\psi=0
		&&\quad\text{in }\Omega,\\
		\boldsymbol\psi&=0
		&&\quad\text{on }\partial\Omega.
	\end{alignat*}
	The standard estimates of the Stokes equation give
	\begin{align*}
		\|\boldsymbol\psi\|_{H^2(\Omega)}
		+
		\|\chi\|_{H^1(\Omega)}
		\leq
		C\|\mathbf F\|_{L^2(\Omega)}.
	\end{align*}
	Testing the equation for $\boldsymbol\psi$ by $\mathbf z_g$ and the
	equation for $\mathbf z_g$ by $\boldsymbol\psi$ yields
	\begin{align*}
		\int_\Omega \mathbf z_g\cdot\mathbf F\,\dx
		=
		-\langle g,\chi\rangle_
		{H^{-1}_{(0)}(\Omega),H^1_{(0)}(\Omega)},
	\end{align*}
	which proves
	\eqref{eq:negative-Stokes-lifting} by taking the supremum over $\mathbf F$. In particular, the solution map
	$g\mapsto\mathbf z_g$ extends to a bounded linear map
$		H^{-1}_{(0)}(\Omega)\to L^2(\Omega)^d$.
	The unweighted estimate
	\eqref{eq:Dirichlet-Stokes-lifting-low-order} follows from the standard weak theory for the Stokes system with prescribed divergence by using the weak formulation. 
	
	We now prove the weighted estimate.
	Let $\delta'>0$ be so small that $ 2\delta'<\dist 
	\left(\partial\Omega, \bigcup_{t\in[0,T_0]}\Gamma_t(3\delta)
	\right) $,
	and define $ \Omega_\kappa
	\coloneqq
	\left\{
	x\in\Omega:
	\dist (x,\partial\Omega)<\kappa
	\right\} $.
	Moreover, let $\eta\in C^\infty_0(\Omega)$ be such that $ 0\leq\eta\leq1,
	\eta=1
	$ on $\Omega\setminus\Omega_{\delta'} $.
	
	Since $ \nabla\pi_g=\Delta\mathbf z_g $,
	integration by parts twice gives
	\begin{align}
		&\int_\Omega
		\eta^2\omega_\varepsilon^2
		|\nabla\pi_g|^2\,\dx
		=
		\sum_{i,j=1}^d
		\int_\Omega
		\eta^2\omega_\varepsilon^2
		\partial_{jj}(\mathbf z_g)_i
		\partial_i\pi_g\,\dx
		=
		\int_\Omega
		\eta^2\omega_\varepsilon^2
		\nabla g\cdot\nabla\pi_g\,\dx
		\notag\\
		&\quad+
		\sum_{i,j=1}^d
		\int_\Omega
		\partial_i\left(
		\eta^2\omega_\varepsilon^2
		\right)
		\partial_j(\mathbf z_g)_i
		\partial_j\pi_g\,\dx
		-
		\sum_{i,j=1}^d
		\int_\Omega
		\partial_j\left(
		\eta^2\omega_\varepsilon^2
		\right)
		\partial_j(\mathbf z_g)_i
		\partial_i\pi_g\,\dx.
		\label{eq:weighted-pressure-lifting-identity}
	\end{align}
	Here we used $\Div\mathbf z_g=g$ to derive
	\begin{align*}
		\sum_{i,j=1}^d
		\int_\Omega
		\eta^2\omega_\varepsilon^2
		\partial_{ij}(\mathbf z_g)_i
		\partial_j\pi_g\,\dx
		=
		\int_\Omega
		\eta^2\omega_\varepsilon^2
		\nabla g\cdot\nabla\pi_g\,\dx.
	\end{align*}
	By $|\nabla (\eta^2\omega_\varepsilon^2)| \leq C\eta\omega_\varepsilon$, Cauchy's inequality, and
	\eqref{eq:Dirichlet-Stokes-lifting-low-order}, we obtain
	\begin{align*}
		\|\eta\omega_\varepsilon\nabla\pi_g\|_{L^2(\Omega)}^2
		\leq
		\|\omega_\varepsilon\nabla g\|_{L^2(\Omega)}
		\|\eta\omega_\varepsilon\nabla\pi_g\|_{L^2(\Omega)}
		+
		C\|\nabla\mathbf z_g\|_{L^2(\Omega)}
		\|\eta\omega_\varepsilon\nabla\pi_g\|_{L^2(\Omega)}.
	\end{align*}
	Consequently, Young's equality yields
	\begin{align}
		\|\eta\omega_\varepsilon\nabla\pi_g\|_{L^2(\Omega)}
		\leq
		C\left(
		\|g\|_{L^2(\Omega)}
		+
		\|\omega_\varepsilon\nabla g\|_{L^2(\Omega)}
		\right).
		\label{eq:interior-weighted-pressure-lifting}
	\end{align}
	
	It remains to estimate the pressure in the fixed boundary neighborhood. There
	$\omega_\varepsilon\in [1,2]$. The local boundary
	regularity estimate for the Stokes system with prescribed divergence (cf.\ e.g.\ \cite[Theorem~IV.5.2 and Exercise~IV.5.2]{Galdi}) gives 
	\begin{align}
		\|\nabla^2\mathbf z_g\|_{L^2(\Omega_{\delta'})}
		+
		\|\nabla\pi_g\|_{L^2(\Omega_{\delta'})}
		\leq
		C\Bigl(
		\|g\|_{H^1(\Omega_{2\delta'})}
		+
		\|\mathbf z_g\|_{H^1(\Omega)}
		+
		\|\pi_g\|_{L^2(\Omega)}
		\Bigr).
		\label{eq:boundary-Stokes-lifting}
	\end{align}
	Combining \eqref{eq:Dirichlet-Stokes-lifting-low-order} and
	\eqref{eq:boundary-Stokes-lifting} implies
	\begin{align}
		\|\omega_\varepsilon\nabla\pi_g\|_{L^2(\Omega_{\delta'})}
		\leq
		C\left(
		\|g\|_{L^2(\Omega)}
		+
		\|\omega_\varepsilon\nabla g\|_{L^2(\Omega)}
		\right),
		\label{eq:boundary-weighted-pressure-lifting}
	\end{align} 
	which thereafter proves
	\begin{align}
		\|\omega_\varepsilon\nabla\pi_g\|_{L^2(\Omega)}
		\leq
		C\left(
		\|g\|_{L^2(\Omega)}
		+
		\|\omega_\varepsilon\nabla g\|_{L^2(\Omega)}
		\right).
		\label{eq:global-weighted-pressure-lifting}
	\end{align}
	
	Finally, applying Proposition~\ref{prop:weightedEllipticReg} with $k=1$ componentwise for $ \nabla\pi_g=\Delta\mathbf z_g $ gives
	\begin{align*}
		\|\omega_\varepsilon\nabla^2\mathbf z_g\|_{L^2(\Omega)}
		\leq
		C\left(
		\|(\mathbf z_g,\nabla\mathbf z_g,\omega_\varepsilon\nabla\pi_g)\|_{L^2(\Omega)}
		\right)
		\leq
		C\left(
		\|g\|_{L^2(\Omega)}
		+
		\|\omega_\varepsilon\nabla g\|_{L^2(\Omega)}
		\right).
	\end{align*}
	Together with \eqref{eq:global-weighted-pressure-lifting}, this proves
	\eqref{eq:Dirichlet-Stokes-lifting-weighted}.
	
	It remains to prove \eqref{eq:weighted-first-order-lifting}. Assume
	\eqref{eq:divergence-representation} and define $ \mathbf y_g\coloneqq\mathbf z_g-\mathbf a $.
	Then
	\begin{alignat*}{3}
		-\Delta\mathbf y_g
		&=-\nabla\pi_g
		+ \Delta\mathbf a,
		\quad
		\Div \mathbf y_g=0
		&&\quad\text{in }\Omega,\\
		\mathbf y_g&=0
		&&\quad\text{on }\partial\Omega.
	\end{alignat*}
	Testing above equation with $\omega_\varepsilon^2 \mathbf y_g$ yields
	\begin{align*}
		\|\omega_\varepsilon\nabla\mathbf y_g\|_{L^2(\Omega)}^2
		\leq
		C\left(
		\|\mathbf y_g\|_{L^2(\Omega)}^2
		+
		\|\omega_\varepsilon\nabla\mathbf a\|_{L^2(\Omega)}^2
		\right).
	\end{align*}
	Moreover, by \eqref{eq:negative-Stokes-lifting},
	\begin{align*}
		\|\mathbf y_g\|_{L^2(\Omega)}
		\leq
		\|\mathbf z_g\|_{L^2(\Omega)}
		+
		\|\mathbf a\|_{L^2(\Omega)}
		\leq
		C\|\Div \mathbf a\|_{H^{-1}(\Omega)}
		+
		\|\mathbf a\|_{L^2(\Omega)}
		\leq
		C\|\mathbf a\|_{L^2(\Omega)}.
	\end{align*}
	Since $\mathbf z_g=\mathbf y_g+\mathbf a$, this proves
	\eqref{eq:weighted-first-order-lifting}.
\end{proof}

\subsection{Bi-linear Forms from the Allen--Cahn Operator}

For the following we introduce for $k\in\N_0$
\begin{alignat}{2}
	B_\eps^k (\tilde{u}, \tilde{v})&\coloneqq  \int_{\Sigma_{2\delta}}\left(\partial_r \tilde{u}\, \partial_r\tilde{v} + \tfrac{f''(\theta_0(\tfrac{r}\eps-h_\eps (s,t)))}{\eps^2} \tilde{u}\tilde{v}\right) \toeps^{2k}\tilde{J}_t\,\d(r,s), \label{bi-linear-1} \\
  B_\Sigma^k (\tilde{u},\tilde{v})&\coloneqq  \sum_{i,j=1}^d\int_{\Sigma_{2\delta}}a_{i,j}(\nabla_\Sigma \tilde{u})_i(\nabla_\Sigma \tilde{v})_j\toeps^{2k}\tilde{J}_t \,\d(r,s),  \label{bi-linear-2}\\
  B_\eps^k (\tilde u)&\coloneqq  B^k_\eps (\tilde u, \tilde{u}),\quad B_\Sigma^k (\tilde{u})=B_\Sigma^k (\tilde{u},\tilde{u}),\quad B_\eps(\cdot) = B_\eps^0(\cdot), \quad B_\Sigma (\cdot)= B_\Sigma^0(\cdot) \label{bi-linear-3}
 \end{alignat}
for all $\tilde{u},\tilde{v}\in H^1(\Sigma_{2\delta})$, where $\tilde{J}_t(r,s)=J(r,s,t)$ or $\tilde{J}_t(r,s)\equiv 1$, $\toeps= \sqrt{\eps^2+ r^2}$ and
\begin{align*}
	a_{i,j}(r,s,t)&\coloneqq\sum_{k=1}^{d} b_{i,k}(r,s,t) b_{j,k}(r,s,t)=\nabla S_i\cdot \nabla S_j  , \quad
	b_{i,k}(r,s,t)\coloneqq(\partial_{x_k} S_i)(X(r,s,t))            
\end{align*}
for all $i,j,k=1,\ldots, d$.

As in Lemma~\ref{lem:f''-decomposition} we use $f''(\theta_0(\rho))= \alpha +V(\rho)$, where $V(\rho)\to_{|\rho|\to \infty} 0$ exponentially. Then one obtains for every $k\in\N$
\begin{align}\nonumber
  B_\eps^k (\tilde{u}) &\geq c_0\int_{\Sigma_{2\delta}}\left(|\partial_r \tilde{u}|^2 + \tfrac{\alpha}{\eps^2} \tilde{u}^2\right) \toeps^{2k}\,\d(r,s)- C_0\int_{\Sigma_{2\delta}}\tfrac{\toeps^2}{\eps^2} |V(\tfrac{r}\eps-h_\eps)|\tilde u^2 \toeps^{2(k-1)}\,\d(r,s) \\\label{eq:CoercivBepsk}
  &\geq 
	c_0\|\toeps^k \partial_r \tilde{u}\|_{L^2(\Sigma_{2\delta})}^2
	+ c_0\alpha \|\tfrac{\toeps^k}{\eps}\tilde{u}\|_{L^2(\Sigma_{2\delta})}^2
	- C_0\|\toeps^{k-1} \tilde{u}\|_{L^2(\Sigma_{2\delta})}^2\\\nonumber
	B_\eps^k (\tilde{u},\tilde{v})
	&\leq C_0\|\toeps^k \partial_r \tilde{u}\|_{L^2(\Sigma_{2\delta})} \|\toeps^k \partial_r \tilde{v}\|_{L^2(\Sigma_{2\delta})}+ \alpha \|\tfrac{\toeps^k}{\eps}\tilde{u}\|_{L^2(\Sigma_{2\delta})}\|\tfrac{\toeps^k}{\eps}\tilde{v}\|_{L^2(\Sigma_{2\delta})}\\\nonumber
	&\quad + C_0\|\toeps^{k-1} \tilde{u}\|_{L^2(\Sigma_{2\delta})}\|\toeps^{k-1} \tilde{v}\|_{L^2(\Sigma_{2\delta})}\\\label{eq:BddBepsk}
	&\leq C\|(\toeps^k \partial_r \tilde{u},\tfrac{\toeps^k}{\eps}\tilde{u})\|_{L^2(\Sigma_{2\delta})}\|(\toeps^k \partial_r \tilde{v},\tfrac{\toeps^k}{\eps}\tilde{v})\|_{L^2(\Sigma_{2\delta})}
\end{align}
for some $c_0,C_0,C>0$ independent of $\eps$ and $\tilde{u}$, where we have used $1\leq \tfrac{\toeps}\eps$ as well as
\begin{equation}
  \label{eq:VEstim}
  \left|\tfrac{\toeps^2}{\eps^2}V(\tfrac{r}\eps -h_\eps(s,t))) \right|
  \leq (1+ 2(\tfrac{r}{\eps}-h_\eps(s,t))^2 + 2h_\eps(s,t)^2)\left|V(\tfrac{r}\eps -h_\eps(s,t)) \right|\leq C
\end{equation}
for all $r\in \R, s\in\Sigma, t\in [0,T_0],\eps\in (0,1]$.

For $u\in H^2(\Sigma_{2\delta})$ we define
\begin{equation}\label{eq:Leps}
  \widetilde{\mathcal{L}}_\eps\tilde{u}(r,s) = -\partial_r^2 \tilde{u}(r,s) + \tfrac1{\eps^2} f''(\theta_0(\tfrac{r}\eps-h_\eps(s,t)))\tilde{u}(r,s)
\end{equation}
for every $(r,s)\in\Sigma_{2\delta}$, $t\in [0,T_0]$, $\eps\in (0,1]$. Then we have for every $k\geq 2$ 
      \begin{align*}
           \|\toeps^{k}\widetilde{\mathcal{L}}_\eps \tilde u\|_{L^2(\Sigma_{2\delta})}&\geq \|\toeps^{k}(-\partial_r^2+\tfrac{\alpha}{\eps^2}) \tilde u\|_{L^2(\Sigma_{2\delta})}-c_k\|\toeps^{k-2}\tilde u\|_{L^2(\Sigma_{2\delta})}\\
            \|\toeps^{k}\widetilde{\mathcal{L}}_\eps \tilde u\|_{L^2(\Sigma_{2\delta})}&\leq \|\toeps^{k}(-\partial_r^2+\tfrac{\alpha}{\eps^2}) \tilde u\|_{L^2(\Sigma_{2\delta})}+ C_k\|\toeps^{k-2}\tilde u\|_{L^2(\Sigma_{2\delta})}
      \end{align*}
      for some $c_k,C_k>0$ independent of $\tilde u\in H^2(\Sigma_{2\delta})\cap H^1_0(\Sigma_{2\delta})$ and $\eps \in (0,1]$
      because of \eqref{eq:VEstim}. Here
      \begin{align*}
        \|\toeps^{k}(-\partial_r^2+\tfrac{\alpha}{\eps^2}) \tilde u\|_{L^2(\Sigma_{2\delta})}^2 &= \|\toeps^{k}\partial_r^2 \tilde u\|_{L^2(\Sigma_{2\delta})}^2 + \alpha^2\|\tfrac{\toeps^{k}}{\eps^2} \tilde u\|_{L^2(\Sigma_{2\delta})}^2 + 2\alpha\| \tfrac{\toeps^k}\eps \partial_r \tilde u\|_{L^2(\Sigma_{2\delta})}^2\\
        &\quad + 4k\alpha \int_{\Sigma_{2\delta}}\tfrac{\toeps^{2k-1}}{\eps^2}\partial_r \toeps \tilde u \partial_r \tilde u \,\d(r,s)  
      \end{align*}
      since $\tilde u|_{\partial\Sigma_{2\delta}}=0$, where $\partial_r \toeps =\tfrac{r}{\toeps}$, and
      \begin{align*}
        \left|2\int_{\Sigma_{2\delta}}\tfrac{\toeps^{2k-1}}{\eps^2}\partial_r \toeps \tilde u \partial_r \tilde u \,\d(r,s)\right|= \left|\int_{\Sigma_{2\delta}}\tfrac{\partial_r (\toeps^{2k-1}\partial_r \toeps)}{\eps^2}\tilde u^2  \,\d(r,s)\right|\leq C\|\tfrac{\toeps^{k}}{\eps^2} \tilde u\|_{L^2(\Sigma_{2\delta})}\|\toeps^{k-2} \tilde u\|_{L^2(\Sigma_{2\delta})}.
      \end{align*}
      Hence
        \begin{align}\label{eq:EquivLr1}
 \|(\toeps^{k}\partial_r^2 \tilde u,\tfrac{\toeps^{k}}{\eps^2} \tilde u,\tfrac{\toeps^k}\eps \partial_r \tilde u)\|_{L^2(\Sigma_{2\delta})}      &\leq   c_k\left(   \|\toeps^{k}\widetilde{\mathcal{L}}_\eps \tilde u\|_{L^2(\Sigma_{2\delta})} +\|\toeps^{k-2}\tilde u\|_{L^2(\Sigma_{2\delta})}\right),\\ \label{eq:EquivLr2}
            \|\toeps^{k}\widetilde{\mathcal{L}}_\eps \tilde u\|_{L^2(\Sigma_{2\delta})}&\leq C_k \|(\toeps^{k}\partial_r^2 \tilde u,\tfrac{\toeps^{k}}{\eps^2} \tilde u)\|_{L^2(\Sigma_{2\delta})}
      \end{align}
for some $c_k, C_k>0$ independent of $\tilde u\in H^2(\Sigma_{2\delta})\cap H^1_0(\Sigma_{2\delta})$ and $\eps \in (0,1]$ provided $k\geq 2$.

      \subsection{Elliptic Regularity Close to the  Interface}
      We consider the auxiliary elliptic problem  in  $\Sigma_{2\delta}$
\begin{align}\label{eq:TransformElliptic}
	- \widetilde{\Delta}_{\Sigma} \tilde{u}+\widetilde{\mathcal{L}}_\eps \tilde{u} =\tilde{g},\ \ \tilde u\in H^2(\Sigma_{2\delta})\cap H^1_0(\Sigma_{2\delta}),
\end{align}
with $\tilde g \in L^2(\Sigma_{2\delta})$,  where 
\begin{align}
	\widetilde{\Delta}_{\Sigma}\tilde{u}(r,s) &=  \sum_{i,j=1}^{d}(\nabla_\Sigma)_i\left(\nabla S_i\cdot \nabla S_j|_{x=X(r,s,t)}  (\nabla_\Sigma)_j  \tilde{u}\right),\label{eq:Deltaeps}
\end{align}
and $\widetilde{\mathcal{L}}_\eps$ is as in \eqref{eq:Leps}.

Now we present elliptic estimates regarding \eqref{eq:TransformElliptic}.
\begin{prop}
	\label{prop:EllipticEstim}
	Let $ \tilde{u} $ be a solution to \eqref{eq:TransformElliptic} and $ k \in \bbn $. Then it follows
	\begin{align}
		& \|(\toeps^k\nabla_\Sigma^2\tilde{u}, \toeps^k\partial_r \nabla_\Sigma \tilde{u}, \toeps^k\widetilde{\mathcal{L}}_\eps \tilde u,\tfrac{\toeps^k}{\eps} \nabla_\Sigma \tilde{u})\|_{L^2(\Sigma_{2\delta})}^2\nonumber\\
		\label{eq:EllipticEstim}
		&\quad \leq C_k\left(\|\toeps^k \tilde{g}\|_{L^2(\Sigma_{2\delta})}^2+\|(\toeps^{k-1}\tilde u,\toeps^{k-1}\nabla_\Sigma \tilde{u}, \toeps^k\partial_r \tilde{u})\|_{L^2(\Sigma_{2\delta})}^2\right),
	\end{align} 
	where $ C_k > 0 $ does not depend on $ \eps \in (0,1] $.
      \end{prop}
\begin{proof}
	We test \eqref{eq:TransformElliptic} with $-\toeps^{2k}(\nabla_{\Sigma}(\nabla_{\Sigma}\tilde{u})_l)_l=-\toeps^{2k}\partial^2_{s_l}\tilde{u}$, $l=1,\ldots,d$, $k\in\N$, and integrate with respect to $(r,s)\in \Sigma_{2\delta}$. Using integration by parts, this yields
	\begin{align*}
		& \sum_{i,j=1}^d\int_{\Sigma_{2\delta}}a_{i,j}(\nabla_\Sigma(\nabla_{\Sigma}\tilde{u})_l)_i (\nabla_\Sigma(\nabla_{\Sigma}\tilde{u})_l)_j \toeps^{2k}\, \d(r,s)+B_\eps^k(\nabla_{\Sigma}\tilde{u}) \\ 
		&= -\sum_{i,j=1}^d\int_{\Sigma_{2\delta}}
                  b_{i,j,l}(\nabla_{\Sigma}\tilde{u})_i(\nabla_\Sigma(\nabla_{\Sigma}\tilde{u})_l)_j \toeps^{2k} \, \d(r,s)\\
		&\quad + 2k\int_{\Sigma_{2\delta}} \toeps^k \partial_r \tilde u\,  \toeps^{k-1}\tilde u\,  \partial_r \toeps \,\d(r,s) -\int_{\Sigma_{2\delta}}\toeps^k \tilde{g}\toeps^k(\nabla_{\Sigma}(\nabla_{\Sigma}\tilde{u})_l)_l\, \d(r,s)\\
          &\quad + \int_{\Sigma_{2\delta}} f'''(\theta_0(\tfrac{r}\eps-h_\eps))\tfrac{\toeps^2}{\eps^2}\theta_0'(\tfrac{r}\eps-h_\eps)(\nabla_\Sigma h_\eps)_l \toeps^{2k-2} \tilde u (\nabla_{\Sigma}\tilde{u})_l\, \d(r,s)
	\end{align*}
        for some uniformly bounded $b_{i,j,l}\colon \Sigma_{2\delta}\to \R$ arising from derivatives of $a_{i,j}$ and terms related to $\Sigma$.
	Here $\frac{\toeps^2}{\eps^2}\theta_0'(\tfrac{r}\eps-h_\eps)(\nabla_\Sigma h_\eps)_l$ is uniformly bounded as well.
	Summation with respect to $l=1,\ldots,d$, using Young's inequality, and the ellipticity of $\widetilde{\Delta}_{\Sigma}$ yields
	\begin{align*}
		& \|\toeps^k\nabla_\Sigma^2\tilde{u}\|_{L^2(\Sigma_{2\delta})}^2+ B_\eps^k(\nabla_{\Sigma}\tilde{u})  \leq C\left(\|\toeps^k \tilde{g}\|_{L^2(\Sigma_{2\delta})}^2+\|(\toeps^{k-1}\tilde u,\toeps^{k-1}\nabla_\Sigma \tilde{u}, \toeps^k \partial_r\tilde{u})\|_{L^2(\Sigma_{2\delta})}^2\right).
	\end{align*} 
	Furthermore, we use \eqref{eq:TransformElliptic} to estimate $\widetilde{\mathcal{L}}_\eps \tilde{u}$ and
	\begin{align*}
		B_\eps^k(\nabla_{\Sigma}\tilde{u})        &\geq 
		c_k\|\toeps^k \partial_r \nabla_\Sigma \tilde{u}\|_{L^2(\Sigma_{2\delta})}^2
		+ c_k\alpha \|\tfrac{\toeps^k}{\eps} \nabla_\Sigma \tilde{u}\|_{L^2(\Sigma_{2\delta})}^2
		-C_k\|\toeps^{k-1}\nabla_{\Sigma}\tilde{u}\|_{L^2(\Sigma)}^2.
	\end{align*}
	We then obtain the desired estimate \eqref{eq:EllipticEstim}, which completes the proof.
\end{proof}

      \begin{rem}
      In the case $k\geq 2$ \eqref{eq:EllipticEstim} implies
        	\begin{align}
		& \|(\toeps^k\nabla^2\tilde{u},\tfrac{\toeps^k}{\eps} \nabla \tilde{u}, \tfrac{\toeps^k}{\eps^2} \tilde{u} )\|_{L^2(\Sigma_{2\delta})}^2\nonumber\\
		\label{eq:EllipticEstim2}
		&\quad \leq C_k\left(\|\toeps^k \tilde{g}\|_{L^2(\Sigma_{2\delta})}^2+\|(\toeps^{k-2}\tilde u,\toeps^{k-1}\nabla_\Sigma \tilde{u}, \toeps^k\partial_r \tilde{u})\|_{L^2(\Sigma_{2\delta})}^2\right),
	\end{align} 
	where $ C_k > 0 $ does not depend on $ \eps \in (0,1] $, because of \eqref{eq:EquivLr1}.
      \end{rem}

\section{Energy-Type Estimates for a Modified Linearized System}\label{linearized system}
      In the first step we will consider the following modified linear system
\begin{subequations}
	\label{eq:linNSAC'}
	\begin{alignat}{2}\nonumber
		\partial_t \ol{\we} +\ve_A\cdot \nabla \ol\we&+\ol\we\cdot \nabla \ve_A-\Div(2\nu(c_A)D\ol\we) - \Div(2\nu'(c_A)\ol uD\ve_A) +\nabla \ol{q}\\\label{eq:linNSAC1'}
		& = -\eps \Div (\nabla \ol{u} \otimes \nabla c_A+\nabla c_A \otimes \nabla \ol{u})+\ol{\mathbf{r}}_1&\ & \text{in}\ Q_T,\\\label{eq:linNSAC2'}
		\Div \ol{\we}& =\ol{r}_2&\ & \text{in}\ Q_T,\\\label{eq:linNSAC3'}
		\partial_t \ol{u} +\ve_A\cdot \nabla \ol{u}&+ (\ol{\we} -\ol{\we}|_{\Gamma_t})\cdot \nabla c_A  =\Delta \ol{u} - \tfrac1{\eps^{2}} f''(c_A)\ol{u}+\ol{r}_3 &\ & \text{in}\ Q_T,\\
		\label{eq:linNSAC4'}
		(\ol{\we},\ol{u})|_{\partial\Omega}&= (\boldsymbol{0},0)&\ & \text{on }S_T,\\
		\label{eq:linNSAC5'}
		(\ol\we,\ol{u}) |_{t=0}& = (\ol{\we}_0,\ol{u}_0)&\ & \text{in }\Omega,
	\end{alignat}
      \end{subequations}
      where $T\in (0,T_0]$ and $\ol{\we}|_{\Gamma_t}:= \ol{\we}\circ P_{\Gamma_t}$.
      Here the additional term $\ol{\we}|_{\Gamma_t}\cdot \nabla c_A$ in \eqref{eq:linNSAC3'} leads to a decoupling of the Navier--Stokes system and Allen--Cahn equation ``in highest order'' concerning the asymptotics as $\eps\to 0$. We will reduce to this system in Section~\ref{sec:FullSystem} below by subtracting a certain leading part $(\we_A,u_A)$ of the solution $(\we,u)$ of \eqref{eq:linNSAC}.

      For the following analysis we assume that
      \begin{equation}
        \label{eq:v-1-eps}
        \ve_A= \ve_0 + \eps \ve_{1,\eps},
      \end{equation}
      where $\|\ve_{1,\eps}\|_{C^1(\ol\Omega)}$ is uniformly bounded in $\eps\in (0,1]$ and $t\in[0,T]$,
      \begin{align*}
        \mathbf{v}_0(x,t)&=\zg(x,t) \hv_0(\rho,x,t)+(1-\zg(x,t) )\left(\ve^{+}(x,t)\chi_{\Omega^+}(x,t) +\ve^-(x,t)\chi_{\Omega^-}(x,t)\right),\\
        \hv_0(\rho,x,t)&=\mathbf{v}^+(x,t)\eta(\rho)+\mathbf{v}^-(x,t)(1-\eta(\rho)),
      \end{align*}
 where $\rho$ is as in \eqref{eq:StretchedVariable}. In particular, if $\ve_A$ is as in Remark~\ref{rem:PropertiesAproxSol}, this is satisfied.
 \begin{rem}\label{rem:vAC1bounded}
   We note that, since $\ve^\pm, h_{\varepsilon}$ are smooth and uniformly bounded, we have
 \begin{align}
 	|\nabla\ve_0|&\leq \frac{1}{\varepsilon}\zg\left|\frac{\ve^+-\ve^-}{d_{\Gamma}}\cdot d_{\Gamma}\eta_{}'(\rho)\right|+\zg\big|(\ve^+-\ve^-)\eta_{}'(\rho) \nabla^{\Gamma}h_{\varepsilon}\big|+C\nonumber\\
 	&\leq C\zg\big|\eta_{}'(\rho)(\rho+h_{\varepsilon})\big|+C\leq \tilde{C}\qquad \text{in }Q_{T_0}.
 \end{align}
 which implies that  $\nabla \ve_A$ is uniformly bounded.  Similarly, we have
\begin{align*}
	\hv_0(\rho,x,t)-\bv^+(x,t)
	&=
	(1-\eta(\rho))
	\bigl(\bv^-(x,t)-\bv^+(x,t)\bigr)\\
	&=
	-(1-\eta(\rho))d_\Gamma(x,t)\bu_0(x,t) = O(\eps), \quad x\in\Om^+(t)\cap\Gamma_t(2\delta), \\
    \hv_0(\rho,x,t)-\bv^-(x,t)
	&=
	\eta(\rho)\bigl(\bv^+(x,t)-\bv^-(x,t)\bigr)\\
	&=
	\eta(\rho)d_\Gamma(x,t)\bu_0(x,t) = O(\eps), \quad x\in\Om^-(t)\cap\Gamma_t(2\delta),
\end{align*}
by the exponential decay of $\eta(\rho)$ and $d_\Gamma = \eps(\rho + h_\eps)$, where $\ue_0$ is as in Remark~\ref{rem:PropertiesAproxSol}.
Consequently, 
 \begin{align*}
    \mathbf{v}_A=\ve^{+}\chi_{\Omega^+} +\ve^-\chi_{\Omega^-}+O(\varepsilon) \quad \text{ in } L^\infty(\Omega \times (0,T_0)).
 \end{align*} 
 \end{rem}

In this section we will show:
\begin{theorem}\label{thm:main-0}
  Let $\eps\in (0,1]$, let 
  \begin{align*}
  \ol\we_0&\in H^1_0(\Omega)^d, \quad \ol u_0\in L^2(\Omega),\quad  \ol{\mathbf{r}}_1\in L^2(0,T;V'_\sigma),  \quad   \ol r_3 \in L^2(Q_T), \\
    \ol r_2 &\in L^2(0,T;H^1_{(0)}(\Omega))\cap H^{1}(0,T;H^{-1}_{(0)}(\Omega))\quad \text{with }\ol r_2|_{t=0}= \Div \ol\we_0,
  \end{align*}
  and $(\ol \we, \ol u)$ with
  \begin{align*}
     \ol \we \in C([0,T];L^2(\Omega)^d)\cap L^2(0,T; H^1_0(\Omega)^d),\quad \ol u \in L^2(0,T;H^1_0(\Omega))\cap H^1(0,T;H^{-1}(\Omega))
  \end{align*}
  be the unique weak solution of \eqref{eq:linNSAC'}. 
   Then there exist $C>0$ and $\eps_1\in (0,1]$ independent of $(\ol \we,\ol u)$ and $(\ol\we_0,\ol u_0, \ol r_1,\ol r_2, \ol r_3)$ such that 
   \begin{align}
      	&\|(\ol{u},\ol{\we})\|_{L^\infty(0,T;L^2)}+\|\ol u\|_{L^2(0,T;\Veps)}+ \|\nabla_{\btau} \ol u\|_{L^2(0,T;L^2(\Gamma_t(2\delta)))}+\|(\oeps \nabla \ol u, \tfrac{\oeps}\eps \ol u,\nabla \ol{\we})\|_{L^2(Q_T)}
          \nonumber\\
        &\quad \leq C\left(\|\ol{u}_0\|_{L^2(\Omega)}
        +\|\ol\we_0\|_{H^1(\Omega)} + \|\ol{\mathbf{r}}_1\|_{L^2(0,T;V_\sigma')}
        + \|(\ol r_2,\oeps \nabla \ol r_2)\|_{L^2(Q_T)}\right.\nonumber\\
   &\qquad\qquad  \left. + \| \ol{r}_2\|_{H^1(0,T; H^{-1}_{(0)})} +\|\ol{r}_3\|_{L^2(0,T;(\Veps)')}+\|\oeps^2\ol{r}_3\|_{L^2(Q_T)}\right)\label{ineq:uw energy-1}
   \end{align}
   for all $\eps\in (0,\eps_1]$ and $T\in (0,T_0]$.
 \end{theorem}
 
We note that existence of a unique weak solution of the linear system \eqref{eq:linNSAC'} in the spaces stated in the theorem can be shown by standard methods.
First we prove an auxiliary results before we prove Theorem \ref{thm:main-0}.
\begin{prop}[Weighted estimate for the nonstationary Stokes equations]
	\label{prop:weighted-divergence-corrector}~\\
        Let $T\in(0,T_0]$ and $\varepsilon\in(0,1]$. Assume that
	\begin{align*}
		\ol{\bw}_0\in H^1_0(\Omega)^d, \quad \ol{r}_2\in L^2\bigl(0,T;H^1_{(0)}(\Omega)\bigr)
		\cap H^1\bigl(0,T;H^{-1}_{(0)}(\Omega)\bigr)
	\end{align*}
        with $\ol{r}_2|_{t=0}=\Div \ol{\bw}_0$ in $L^2_{(0)}(\Omega)$.
	Then there is a pair $(\bw_1,q_1)$ with
	\begin{align*}
		\bw_1\in L^2\bigl(0,T;H^2(\Omega)^d\cap H^1_0(\Omega)^d\bigr)
		\cap H^1\bigl(0,T;L^2(\Omega)^d\bigr),\quad
		q_1\in L^2\bigl(0,T;H^1_{(0)}(\Omega)\bigr)
	\end{align*}
	that satisfies
	\begin{align*}
		\partial_t\bw_1-\Delta\bw_1+\nabla q_1&=0
		&&\text{in }Q_T,\\
		\Div \bw_1&=-\ol{r}_2
		&&\text{in }Q_T,\\
		\bw_1|_{\partial \Omega}&=0
		&&\text{on }\partial\Omega\times (0,T),\\
		\bw_1|_{t=0}&=-\ol{\bw}_0
		&&\text{in }\Omega.
	\end{align*}
	Moreover, there are constant $C>0$, $\eps_1\in (0,1]$ independent of $\varepsilon\in(0,\eps_1]$ and $T\in(0,T_0]$, such that
	\begin{align}
		& \|(\nabla\bw_1,\partial_t\bw_1,\oeps\nabla^2\bw_1)\|_{L^2(Q_T)} \notag\\
		& \qquad \leq C\Bigl(
		\|(\ol{r}_2,\oeps\nabla \ol{r}_2)\|_{L^2(Q_T)}
		+\|\partial_t \ol{r}_2\|_{L^2(0,T;H^{-1}_{(0)})}
		+\|\ol{\bw}_0\|_{H^1(\Omega)}
		\Bigr).
		\label{eq:w1Estim}
	\end{align}
\end{prop}
\begin{proof}
	We define $ \mathbf z(t)=\mathcal R\bigl(-\ol{r}_2(t)\bigr) $ by Lemma \ref{lem:weighted-Dirichlet-Stokes-lifting}, where for $g\in L^2_{(0)}(\Omega) $  $ \mathcal R g:= \mathbf{z}_g$  is the solution of \eqref{eq:Dirichlet-Stokes-lifting}. We note that $\mathcal{R}$ extend to a linear operator from  $H^{-1}_{(0)}(\Omega)$ to $L^2_{(0)}(\Omega)^d$  by \eqref{eq:negative-Stokes-lifting}. Since $\mathcal R\in\mathcal L(H^{-1}_{(0)}(\Omega),L^2(\Omega)^d)$ is independent of time and
	$\ol{r}_2\in H^1(0,T;H^{-1}_{(0)}(\Omega))$, the standard composition property for
	Bochner--Sobolev functions yields $ \mathbf z\in H^1\bigl(0,T;L^2(\Omega)^d\bigr) $ and
	\begin{align*}
		\partial_t\mathbf z=\mathcal R\bigl(-\partial_t \ol{r}_2\bigr)
		\quad\text{in }L^2\bigl(0,T;L^2(\Omega)^d\bigr).
	\end{align*}
	In particular,
	\begin{align}
		\|\partial_t\mathbf z\|_{L^2(Q_T)}
		\leq C\|\partial_t \ol{r}_2\|_{L^2(0,T;H^{-1}_{(0)}(\Omega))}.
		\label{eq:z-time-derivative-unweighted}
	\end{align}
	Combining \eqref{eq:z-time-derivative-unweighted} with  \eqref{eq:Dirichlet-Stokes-lifting-low-order} for fixed time, \eqref{eq:Dirichlet-Stokes-lifting-weighted} gives
	\begin{align}
		\|(\nabla\mathbf z,\oeps\nabla^2\mathbf z,\partial_t\mathbf z)\|_{L^2(Q_T)}
		\leq C\Bigl(
		\|\ol{r}_2\|_{L^2(Q_T)}
		+\|\oeps\nabla \ol{r}_2\|_{L^2(Q_T)}
		+\|\partial_t \ol{r}_2\|_{L^2(0,T;H^{-1}_{(0)})}
		\Bigr).
		\label{eq:z-time-dependent-estimate}
	\end{align}
	Next, let $\mathbf y$ solve
	\begin{subequations}
		\label{eq:y-equation}
		\begin{alignat}{3}
			\partial_t\mathbf y-\Delta\mathbf y+\nabla\varpi&=-\partial_t\mathbf z
			&&\quad \text{in }Q_T,
			\label{eq:y-equation-one}\\
			\Div \mathbf y&=0
			&&\quad \text{in }Q_T,
			\label{eq:y-equation-two}\\
			\mathbf y|_{\partial \Omega}&=0
			&&\quad \text{on }\partial\Omega\times (0,T),
			\label{eq:y-equation-three}\\
			\mathbf y|_{t=0}&=-\ol{\bw}_0-\mathbf z|_{t=0}
			&&\quad \text{in }\Omega.
			\label{eq:y-equation-four}
		\end{alignat}
	\end{subequations}
	The initial value in \eqref{eq:y-equation-four} is divergence free since $\Div \bigl(-\ol{\bw}_0-\mathbf z|_{t=0}\bigr)
	=-\ol{r}_2|_{t=0}+\ol{r}_2|_{t=0}=0$.
	Moreover, \eqref{eq:Dirichlet-Stokes-lifting-low-order} and the compatibility condition give
	\begin{align}
		\|\mathbf z|_{t=0}\|_{H^1(\Omega)}
		\leq C\|\ol{r}_2|_{t=0}\|_{L^2(\Omega)}
		=C\|\Div \ol{\bw}_0\|_{L^2(\Omega)}
		\leq C'\|\ol{\bw}_0\|_{H^1(\Omega)}.
		\label{eq:z-initial-estimate}
	\end{align}
	By standard results for the non-stationary Stokes system (e.g., by \cite[Theorem~1.2]{SolonnikovLpStokes}), we have
	\begin{align}
		\|\partial_t\mathbf y\|_{L^2(Q_T)}
		+\|\mathbf y\|_{L^2(0,T;H^2)}
		+\|\nabla\mathbf y\|_{L^2(Q_T)}
		\leq C\Bigl(
		\|\partial_t\mathbf z\|_{L^2(Q_T)}
		+\|\ol{\bw}_0\|_{H^1(\Omega)}
		\Bigr).
		\label{eq:y-estimate}
	\end{align}
	
	Finally, we define $ (\bw_1,q_1) = (\mathbf z+\mathbf y,\pi+\varpi) $. Then $(\bw_1,q_1)$ solves the asserted system. Combining
	\eqref{eq:z-time-dependent-estimate} and \eqref{eq:y-estimate} proves
	\eqref{eq:w1Estim}. 
\end{proof}

 Now we define $\widetilde{\we}=\ol{\we}+\we_1$. Then the system for $\widetilde{\we}$ reads
 \begin{subequations}
 	\label{eq:linNSAC'-1}
 	\begin{alignat}{2}\nonumber
 		\partial_t \widetilde{\we} +\ve_A\cdot \nabla \widetilde\we&+\widetilde\we\cdot \nabla \ve_A-\Div(2\nu(c_A)D\widetilde\we)- \Div(2\nu'(c_A)\ol uD\ve_A) +\nabla \ol{q}\\\label{eq:linNSAC1'-1}
 		& = -\eps \Div (\nabla \ol{u} \otimes \nabla c_A+\nabla c_A \otimes \nabla \ol{u})+\ol{\mathbf{r}}_1 +\mathbf{g}
 		 &\quad & \text{in}\ Q_T,\\ \label{eq:linNSAC2'-1}
 		\Div \widetilde{\we}& = 0&\ & \text{in}\ Q_T,\\ \label{eq:linNSAC3'-1}
 		\widetilde{\we}|_{\partial\Omega}&= 0&\ & \text{on }S_T,\\
 		\label{eq:linNSAC5'-1}
 		\widetilde\we|_{t=0}& = 0&\ & \text{in }\Omega,
 	\end{alignat}
 \end{subequations}
 where
$
   \mathbf{g}:=\partial_t \we_1+\ve_A\cdot \nabla \we_1+\we_1\cdot \nabla \ve_A-\Div(2\nu(c_A)D\we_1).
$
 \begin{lem} Under the assumptions of Theorem~\ref{thm:main-0} we have
  \begin{align}
 	&\frac{1}{2}\frac{\d}{\dt}\int_{\Omega} |	\widetilde{\we}|^2dx+2\int_{\Omega} \nu(c_A) |D	\widetilde{\we}|^2dx
\leqslant C\int_{\Omega}\left| \widetilde\we \right|^2dx
 	\nonumber\\&\quad
 	+C\bigg(\|\ol{u}\|_{\Veps}+\|\oeps\nabla \ol u\|_{L^2(\Omega)}+\|\partial_t\we_1\|_{V'_\sigma}+ \|\nabla\we_1\|_{L^2(\Omega)}+\|\ol{\mathbf{r}}_1\|_{V'_\sigma}\bigg)\|\nabla\widetilde{\we}\|_{L^2(\Omega)}\label{wenergyequ-old-1-1}
 \end{align}
for almost every $t\in(0,T)$ and $\eps\in (0,\eps_1]$ if $\eps_1\in (0,1]$ is sufficiently small. Here $C$ is independent of $T\in (0,T_0]$.
  \end{lem}
  \begin{proof}
 Testing \eqref{eq:linNSAC1'-1} with $\widetilde{\we}$ and integrating by parts yield that
 \begin{align}
   &\frac{1}{2}\frac{\d}{\dt}\int_{\Omega} |	\widetilde{\we}|^2\dx+2\int_{\Omega} \nu(c_A) |D	\widetilde{\we}|^2\dx \nonumber\\
   &\quad \leqslant C\int_{\Omega}\left| \widetilde\we \right|^2\sd x+\eps\int_{\Omega}\left(\nabla \ol{u}\otimes_s  \nabla c_A:\nabla	\widetilde{\we}\right)\sd x+\int_{\Omega}\ol{\mathbf{r}}_1\cdot \widetilde{\we}\sd x+ \weight{\partial_t \we_1, \widetilde{\we}}_{V'_\sigma,V_\sigma}\nonumber\\
&\qquad +\int_{\Omega}(\ve_A\cdot \nabla \we_1+\we_1\cdot\nabla \ve_A )\cdot\widetilde{\we}\sd x+\int_\Omega2(\nu(c_A) D\we_1-\nu'(c_A)\ol uD\ve_A):\nabla\widetilde{\we}\sd x. 
\label{wenergyequ-old}
 \end{align} 
 Here we have used that $\Div\ve_A$ and  $\nabla\ve_A$ are bounded, cf.\ Remark~\ref{rem:vAC1bounded}. To estimate the right-hand side the essential step is to show:

 \medskip
 
 \noindent
 \emph{Claim:} There is some constant $C>0$ such that
  \begin{align}
 	&\left|\eps\int_{\Omega}\nabla \ol u\otimes \nabla c_A:\nabla\widetilde{\we}\dx\right|\leq C\left(\|\ol u\|_{\Veps}+\|\oeps\nabla \ol u\|_{L^2(\Omega)}\right)\|\nabla\widetilde{\we}\|_{L^2(\Omega)}.\label{w1est-1}
 \end{align}
\emph{Proof of claim:} We split $\Omega$ into  $\Omega\backslash\Gamma_t(2\delta)$ and $\Gamma_t(2\delta)$ and then the proof of \eqref{w1est-1}
 consists of two parts.
 Firstly, since $\nabla c_A$ is bounded $\Omega\setminus \Gamma_t(2\delta)$ and \eqref{eq:VepsExterior} there holds
 \begin{align}
 	&\varepsilon\bigg|\int_{\Omega\backslash\Gamma_t(2\delta)}\nabla \ol{u}\otimes \nabla c_A:\nabla\widetilde{\we}\dx\bigg|
 	\leq C\eps\|\ol{u}\|_{\Veps}\|\nabla\widetilde{\we}\|_{L^2(\Omega)}.\label{w1est-11}
 \end{align}
 Secondly, noting that
 \begin{equation}\label{eq:LowerBdd}
 	\big|d_\Gamma(x,t) - \eps h_\eps(S(x,t),t)\big|\geq \frac{\delta}{2}\quad \text{for all } x\in \Gamma_t(2\delta)\backslash\Gamma_t(\delta), \eps\in (0,\eps_1]
 \end{equation}
 if $\eps_1\in (0,1]$ is sufficiently small,
 we have 
 \begin{align}
   \nabla c_A (x,t)&=\varepsilon^{-1}\zg\theta_0'(\rho)\nabla d_\Gamma+O(1)\quad \text{in }L^\infty(Q_T). 	\label{innappgrad}
 \end{align}
 Using \eqref{innappgrad} we obtain
 \begin{align}
 	&\left|\eps\int_{\Gamma_t(2\delta)}\nabla \ol{u}\otimes \nabla c_A:\nabla\widetilde{\we}\dx\right|
 	\leq \left|\int_{\Gamma_t(2\delta)}\zg\theta_0'(\rho)\partial_{\nn} \ol{u}\big(\nn\otimes\nn:\nabla\widetilde{\we}\big)\dx \right|
 	\nonumber\\&\qquad + C\|\nabla_{\btau} \ol{u}\|_{L^2(\Gamma_t(2\delta))}\|\nabla\widetilde{\we}\|_{L^2(\Omega)}+ C\varepsilon\|\nabla \ol{u}\|_{L^2(\Gamma_t(2\delta))}\|\nabla\widetilde{\we}\|_{L^2(\Omega)},
\label{w1est-3}
 \end{align}
 where
 \begin{equation*}
   \left|\int_{\Gamma_t(2\delta)}\zg\theta_0'(\rho)\partial_{\nn} \ol{u}\big(\nn\otimes\nn:\nabla\widetilde{\we}\big)\sd x\right| = \left|\int_{\Gamma_t(2\delta)}\zg\theta_0'(\rho)\partial_{\nn} \ol{u}\Div_{\btau}\widetilde{\we}\sd x\right|=:I
 \end{equation*}
 since  $0=\Div\widetilde{\we}=\big(\nn\otimes\nn:\nabla\widetilde{\we}\big)+\Div_{\btau}\widetilde{\we}$.

 Similarly as in the proof of \cite[Lemma 3.4]{AbelsMarquardt1} (mainly by decomposition of $\ol{u}$ and  \cite[Corollary 2.12]{AbelsMarquardt1})  one shows
 \begin{align}
 	I& \leq C_1 \left(\int_{\Gamma_t(2\delta)} \left(|\nabla \ol u|^2 +\tfrac1{\eps^2} f''(c_A) \ol u^2 \right) \sd x + C_L \|\ol{u}\|_{L^2(\Gamma_t(2\delta))}^2\right)^{\frac12}\|\nabla\widetilde{\we}\|_{L^2(\Omega)} \notag\\
    &\quad +C_2\|\ol{u}\|_{L^2(\Gamma_t(2\delta))}\|\nabla\widetilde{\we}\|_{L^2(\Omega)}.\label{w1est-7}
 \end{align}
 Here we need to point out that Lemma 3.4 in \cite{AbelsMarquardt1}  still holds in the case of three space dimensions due to the following embedding
 \begin{equation*}
H^1(\Gamma_t(2\delta))\hookrightarrow L^{p,\infty}(\Gamma_t(2\delta))\quad \text{for all } \ 1\leq p\leq 4.
\end{equation*}
Using \eqref{eq:VepsChar} and $\|\ol u\|_{L^2(\Omega)}\leq C\|\ol u\|_{\Veps} $ we arrive at
 \begin{align}
	I\leq C\|\ol{u}\|_{\Veps}\|\nabla\widetilde{\we}\|_{L^2(\Omega)}.\label{w1est-8}
\end{align}
Substituting \eqref{w1est-8} into \eqref{w1est-3} we conclude the claim \eqref{w1est-1}.

Moreover, it holds
\begin{align*}
	\abs{\int_{\Omega} 2\nu'(c_A)\ol uD\ve_A : \nabla \tw \sd x}
	\leq C \norm{\ol{u}}_{L^2(\Gamma_t(2\delta))} \norm{\nabla \tw}_{L^2(\Omega)},
\end{align*}
where we have used the boundedness of $ \nabla \bv_A $. For the source term we have
\begin{align*}
  \left|  \int_{\Omega} \ol{\mathbf{r}}_1\cdot \widetilde{\we}\sd x \right|
    \leq C \norm{\ol{\mathbf{r}}_1}_{V_\sigma'} \norm{\nabla \tw}_{L^2(\Omega)}.
\end{align*}

 \medskip

 It remains to deal with  the terms containing $\we_1$ in \eqref{wenergyequ-old}. One  easily has
 \begin{align}\nonumber
 	& \big|\weight{\partial_t \we_1, \widetilde{\we} }_{V'_\sigma,V_\sigma} \big| +\bigg|\int_{\Omega}\bigg(  (\ve_A\cdot \nabla \we_1+\we_1\cdot\nabla \ve_A)\cdot\widetilde{\we}  +\big(2\nu(c_A) D\we_1:\nabla\widetilde{\we} \big)\bigg)\dx\bigg|\\
 	&\quad \leq C\|\partial_t\we_1\|_{V'_\sigma}\|\nabla\widetilde{\we}  \|_{L^2(\Omega)}+C \|\nabla\we_1\|_{L^2(\Omega)}\|\nabla\widetilde{\we}\|_{L^2(\Omega)}
 	\label{w2est-5}
 \end{align}
here we have used the fact that $\ve_A$ and $\nabla\ve_A$ are bounded. Combining all the estimates above implies the result.
  \end{proof}

 \begin{lem} Under the assumptions of Theorem~\ref{thm:main-0} there are $C>0,\eps_1\in (0,1]$ independent of $T\in (0,T_0]$ such that
 	\begin{alignat}{2}
 	&\left|\int_{\Omega}(\ol{\we} -\ol{\we}|_{\Gamma_t})\cdot \nabla c_A  \ol{u} \sd x\right|	\nonumber\\&\quad \leq C\|\ol{r}_2\|_{L^2(\Omega)}\|\ol{u}\|_{L^2(\Omega)}
          + C\varepsilon^{\frac{1}{2}}\|\nabla\ol{\we}\|_{L^2(\Omega)}\bigg(\|\nabla_{\btau}\ol{u}\|_{L^2(\Gamma_t(2\delta))}+\|\ol{u}\|_{L^2(\Omega)}\bigg) ,\label{est:diffw1couple}\\
           	&\left|\int_{\Omega}\oeps^2(\ol{\we} -\ol{\we}|_{\Gamma_t})\cdot \nabla c_A  \ol{u} \dx\right|\leq C\|\ol{\we}\|_{H^1(\Omega)}\varepsilon\|\ol{u}\|_{L^2(\Omega)}\label{est:diffw1coupleWeighted}
 	\end{alignat}
        for all $\eps \in (0,\eps_1]$.
      \end{lem}
      \begin{proof}
        Using \eqref{innappgrad} we have
        \begin{align}
 		&	\int_{\Omega}(\ol{\we} -\ol{\we}|_{\Gamma_t})\cdot \nabla c_A  \ol{u} \sd x=J_1+J_{2},
 		\label{est:diffw1-1}
 	\end{align}
 	where
 	\begin{align}
 		J_1&:= \varepsilon^{-1}\int_{\Gamma_t(2\delta)}\zg(\ol{\we} -\ol{\we}|_{\Gamma_t})\cdot\nn\theta_0'(\rho)\ol{u}\sd x,
\quad |J_{2}|\leq  C\varepsilon\|\ol{u}\|_{L^2(\Omega)}\|\nabla\ol{\we}\|_{L^2(\Omega)}
 		.\label{est:j2}
 	\end{align}         
 	In the same way as in the proof of \cite[Lemma 5.1]{StokesAllenCahn} one obtains
 	\begin{align}
 		|J_1|&\leq \varepsilon^{-1} \left| \int_{-2\delta}^{2\delta}\int_{\Sigma}\int_0^r  \tilde{\ol{r}}_2(r',s,t)\dr'\theta_0'(\tfrac{r}{\eps} - h_\eps)\ol{u}(r,s,t)J(r,s,t)\sd s\sd r\right|
 		\nonumber\\&\quad+ C\varepsilon^{\frac{1}{2}}\bigg(\|\nabla_{\btau}\ol{u}\|_{L^2(\Gamma_t(2\delta))}+\|\ol{u}\|_{L^2(\Omega)}\bigg) \|\ol{\we}\|_{L^2(\Omega)}^{\frac{1}{2}}\|\ol{\we}\|_{H^1(\Omega)}^{\frac{1}{2}}
 	 		\nonumber\\	&\leq C\|\ol{r}_2\|_{L^2(\Omega)}\|\ol{u}\|_{L^2(\Omega)}+ C\varepsilon^{\frac{1}{2}}\bigg(\|\nabla_{\btau}\ol{u}\|_{L^2(\Gamma_t(2\delta))}+\|\ol{u}\|_{L^2(\Omega)}\bigg) \|\nabla\ol{\we}\|_{L^2(\Omega)}
 		\label{est:j1}
 	\end{align}
 Using \eqref{est:j1} and \eqref{est:j2} in
 \eqref{est:diffw1-1}  leads to the desired conclusion \eqref{est:diffw1couple}.

 Finally, one easily obtains
 	\begin{align*}
         \left| \int_{\Omega}\oeps^2(\ol{\we} -\ol{\we}|_{\Gamma_t})\cdot \nabla c_A  \ol{u} \sd x \right|
          &\leq \|\ol{\we} -\ol{\we}|_{\Gamma_t}\|_{L^2(\Gamma_t(3\delta))}\|\oeps^2\nabla c_A\|_{L^\infty(\Gamma_t(3\delta))}\|\ol u\|_{L^2(\Omega)}\\
          &\leq C \varepsilon \|\ol{\we}\|_{H^1(\Omega)} \|\ol{u}\|_{L^2(\Omega)}.
 	\end{align*}
\end{proof}

\begin{lem}
There are some $\eps_1\in (0,1]$, $C>0$, independent of $T\in (0,T_0]$ such that
  \begin{alignat}{2}
 	&\|\ol{u}\|_{L^\infty(0,T;L^2(\Omega))}+\|\ol u\|_{L^2(0,T;\Veps)}+ \|\nabla_{\btau} \ol u\|_{L^2(0,T;L^2(\Gamma_t(2\delta)))}+\|(\oeps \nabla \ol u, \tfrac{\oeps}\eps \ol u)\|_{L^2(Q_T)}
          \nonumber\\&\leq C\left(\|\ol{u}_0\|_{L^2(\Omega)}+\|\ol{r}_2\|_{L^2(0,T;L^2(\Gamma_t(2\delta)))} \label{ineq:u energy-000}
   +\|\ol{r}_3\|_{L^2(0,T;(\Veps)')} + \|(\oeps^2\ol{r}_3,\sqrt\eps\nabla\ol{\we})\|_{L^2(Q_T)}\right).
  \end{alignat}
  provided $\eps\in (0,\eps_1]$.
   \end{lem}
 	 \begin{proof}
 Multiplying \eqref{eq:linNSAC3'} by $ (1+\oeps^2) \ol{u} $ and integrating by  parts yield that
 \begin{alignat}{2}\nonumber
   &\frac{1}{2}\frac{\d}{\dt}\int_{\Omega}  (1+\oeps^2)\ol{u} ^2\dx+ \int_{\Omega}(1+\oeps^2)\left( |\nabla  \ol{u} |^2+\tfrac{f''(c_A)}{\eps^2} \ol{u}^2\right) \sd x\\
   &\quad = \int_\Omega \oeps \partial_t \oeps \ol u^2\dx 
  -\int_{\Omega}(1+\oeps^2)\bv_A\cdot \nabla \ol{u} \ol{u}\dx
   -\int_{\Omega}(1+\oeps^2)(\ol{\we} -\ol{\we}|_{\Gamma_t})\cdot \nabla c_A\ol{u}\dx \notag\\
   & \qquad+\int_{\Omega}(1+\oeps^2) \ol{r}_3 \ol{u}\sd x+  \int_\Omega 2\ol u \oeps \nabla \oeps \cdot \nabla \ol u\,\sd x.
     \label{ineq:u energy-00}
 \end{alignat}
 By the spectral estimate due to Theorem \ref{thm:Spectral} and \eqref{eq:CoercivBepsk} one gets
 \begin{alignat}{2}
   &\int_{\Omega}(1+ \oeps^2)\left( |\nabla  \ol{u} |^2+\frac{f''(c_A)}{\eps^2} \ol{u}^2\right) \sd x \geq-\max(C_L,C_1)\int_\Om(1+ \oeps^2) \ol{u}^2 \sd x\nonumber\\
   &\qquad + c_1  \left(\|\nabla_\btau \ol{u}\|^2_{L^2(\Gamma_t({2\delta}))}+ \|\oeps\nabla \ol u\|_{L^2(\Omega)}^2 + \|\tfrac{\oeps}\eps\ol u\|_{L^2(\Omega)}^2  \right).\label{spectral} 
 \end{alignat}
  Moreover, integration by parts yields
   \begin{equation*}
     \left| \int_{\Omega}(1+\oeps^2)\bv_A\cdot \nabla \ol{u}\, \ol{u}\dx\right|=
     \left| \frac12 \int_{\Omega}\Div ((1+\oeps^2)\bv_A)\ol{u}^2\dx\right|\leq C\|\ol u\|_{L^2(\Omega)}^2.
   \end{equation*}
 Thus we obtain
 \begin{alignat}{2}
   &\frac{1}{2}\frac{\d}{\dt}\int_{\Omega} (1+ \oeps^2)\ol{u}^2\sd x +c_1  \left(\|\nabla_\btau \ol{u}\|^2_{L^2(\Gamma_t({2\delta}))}+ \|\oeps\nabla \ol u\|_{L^2(\Omega)}^2 + \|\tfrac{\oeps}\eps\ol u\|_{L^2(\Omega)}^2  \right)\nonumber\\&\quad \leq(C_L+C_1)\int_\Om \ol{u}^2 \sd x+\left|\int_{\Omega}(1+ \oeps^2)(\ol{\we} -\ol{\we}|_{\Gamma_t})\cdot \nabla c_A\ol{u}\sd x\right|
   \nonumber\\\label{ineq:u energy-0}
   &\qquad +\|\ol{r}_3\|_{(\Veps)'}\|\ol{u}\|_{\Veps}
   +  C\left(\|\ol u\|_{L^2(\Omega)}^2 + \|\ol u\|_{L^2(\Omega)}\|\oeps \nabla \ol u\|_{L^2(\Omega)} + \|\oeps^2\ol{r}_3\|_{L^2(\Omega)}^2 \right).
 \end{alignat}
Using \eqref{est:diffw1couple} and \eqref{est:diffw1coupleWeighted} in \eqref{ineq:u energy-0} and Young's inequality one has
\begin{alignat}{2}
  &\frac{1}{2}\frac{\d}{\dt}\int_{\Omega} (1+ \oeps^2)\ol{u}^2\sd x-\max(C_L,C_1)\int_\Om (1+ \oeps^2)\ol{u}^2 \sd x\nonumber\\
  &\leq C\|\ol{r}_2\|_{L^2(\Omega)}\|\ol{u}\|_{L^2(\Omega)}+C\|\ol{r}_3\|_{(V_t^\eps)'}\|\ol{u}\|_{\Veps}+ C\varepsilon^{\frac{1}{2}}\bigg(\|\nabla_{\btau}\ol{u}\|_{\Gamma_t(2\delta)}+\|\ol{u}\|_{L^2(\Omega)}\bigg) \|\nabla\ol{\we}\|_{L^2(\Omega)}\nonumber\\\label{ineq:u energy}
  &\quad +  \left(\|\ol u\|_{L^2(\Omega)}^2 + \|\ol u\|_{L^2(\Omega)}\|\oeps \nabla \ol u\|_{L^2(\Omega)} +\eps \|\ol \we\|_{L^2(\Omega)}^2 + \|\oeps^2\ol{r}_3\|_{L^2(\Omega)}^2 \right).
 \end{alignat}
Similarly, by using the definition of $\|\ol u\|_{\Veps}$ and \eqref{ineq:u energy-00} there holds
 \begin{alignat}{2}
   &\frac{1}{2}\frac{\d}{\dt}\int_{\Omega} \ol{u}^2\sd x+\|\ol u\|_{\Veps}^2-(C_L+1)\int_{\Gamma_t({2\delta})} | \ol{u}|^2\sd x\leq C\|\ol{r}_2\|_{L^2(\Omega)}\|\ol{u}\|_{L^2(\Omega)}\nonumber\\
   &\quad +C\|\ol{r}_3\|_{(V_t^\eps)'}\|\ol u\|_{\Veps}
   \nonumber + C\varepsilon^{\frac{1}{2}}\bigg(\|\nabla_{\btau}\ol{u}\|_{\Gamma_t(2\delta)}+\|\ol{u}\|_{L^2(\Omega)}\bigg) \|\nabla\ol{\we}\|_{L^2(\Omega)}\\\label{ineq:u energy lambda}
 &\quad +  C\left(\|\ol u\|_{L^2(\Omega)}^2 + \|\ol u\|_{L^2(\Omega)}\|\oeps \nabla \ol u\|_{L^2(\Omega)} +\eps \|\ol \we\|_{L^2(\Omega)}^2+ \|\oeps^2\ol{r}_3\|_{L^2(\Omega)}^2 \right).
 \end{alignat}
Combining \eqref{ineq:u energy} with \eqref{ineq:u energy lambda} and using Young's inequality leads to  \eqref{ineq:u energy-000}.
   \end{proof}
   
   \noindent
   \begin{proof*}{of Theorem~\ref{thm:main-0}}
     Applying   Gronwall's and  Young's inequality  in  \eqref{wenergyequ-old-1-1} and using \eqref{eq:w1Estim} we immediately obtain
    \begin{align}
      &\|\ol{\we} \|_{L^\infty(0,T;L^2)}^2 +\|\nabla \ol{\we}\|_{L^2(Q_T)}^2\leq C\bigg( \|\ol{\we}_0\|^2_{H^1(\Omega)} +\|\ol{\mathbf{r}}_1\|^2_{L^2(0,T;V_\sigma')} 
          \nonumber\\
      &\qquad + \|(\ol r_2,\oeps\nabla \ol r_2) \|_{L^2(Q_T)} + \| \ol{r}_2\|^2_{H^1(0,T; H^{-1}_{(0)}(\Omega))}
+\|\ol{u}\|^2_{L^2(0,T;\Veps)}+\|\oeps \nabla \ol u\|^2_{L^2(Q_{T})}\bigg),\label{ineq:uw energy-0}
 \end{align}
   which together with \eqref{ineq:u energy-000} implies that \eqref{ineq:uw energy-1} holds for sufficiently small $\eps$. 
 \end{proof*}

\section{Higher Order Estimates for the Modified Linearized System}
\label{sec:higher-order}

In this section, our goal is to perform higher order estimates taking the singular behavior at the diffuse interface into account. More precisely, for the Allen--Cahn equation, we are going to get better weighted estimates close to the interface, where the weight $\oeps$ is as in \eqref{eq:weight}. It will compensate singularities of normal and higher order derivatives.
Let us recall that $\wW^m$ is normed by
\begin{align*}
  \|u\|_{\wW^m(\Omega)}^2 = \sum_{k+l\leq m}\left\|\tfrac{\oeps^{k+l}}{\eps^l}\nabla^k u\right\|_{L^2(\Omega)}^2. 
\end{align*}
Moreover, we denote for the following
  \begin{equation}\label{eq:Leps2}
  \mathcal{L}_\eps u(x) = -\partial_{\no}^2 u(x) + \tfrac1{\eps^2} f''(\theta_0(\rho(x,t)))u(x)\qquad \text{for all }x\in \Gamma_t(2\delta)
\end{equation}
for all $u\in H^2(\Gamma_t(2\delta))$, $t\in [0,T_0]$, which is the operator defined in \eqref{eq:Leps} with respect to the variable $x\in\Gamma_t(2\delta)$. Here $\no=\no_{\Gamma_t}$.
\subsection{Second Order Estimates for Linearized Convective Allen--Cahn Equation}

In this subsection we start with a result on optimal second order estimates in space, which is as follows:
\begin{thm}
  \label{thm:SecondOrderEstimate}
  Let $ \ol{u}\in L^2(0,T;H^3(\Omega))\cap H^1(0,T;L^2(\Omega)) $ fulfill the linearized convective Allen--Cahn equation \eqref{eq:linNSAC3'} and $\ol u|_{\partial\Omega}=0$ for some $\ol\we\in L^2(0,T;H^1(\Omega))$. Then there is some $C>0$ such that for all $ \eps \in(0,1] $ and $T\in (0,T_0]$, it holds
  \begin{align}
    &\left\|\ol u\right\|_{L^\infty(0,T;\wW^1(\Omega))}+
      \sum_{l=0,1}\left\|\nabla_\btau^l \ol u\right\|_{L^2 (0,T;\wW^1(\Gamma_t(2\delta)))}\nonumber\\
      &\quad +\left\|\ol u\right\|_{L^2 (0,T;\wW^2)} + \left\|\oeps\partial_t \ol u\right\|_{L^2 (Q_T)}\nonumber
      +\left\|\oeps\mathcal{L}_\eps \ol u\right\|_{L^2 (0,T;L^2(\Gamma_t(2\delta)))}\\
      \label{eq:UniformSecondOrderEstim}
    &\leq C\left(\|\oeps \ol{r}_3\|_{L^2(Q_T)} +\|\ol u\|_{L^\infty(0,T;L^2)}+ \|\ol u\|_{L^2(0,T;\Veps)}+ \|\oeps \ol \we \|_{L^2(0,T;H^1)}+ \|\ol u_0\|_{\wW^1(\Omega)}\right).
  \end{align}
\end{thm}
\begin{rem}
  We note that the  estimate \eqref{eq:UniformSecondOrderEstim} is optimal in the  sense that
  $$
  \|\oeps \ol{r}_3- \oeps(\ol{\we} - \ol{\we}|_{\Gamma_t})\cdot \nabla c_A\|_{L^2(Q_T)}
  $$
  (as well as $\|\ol u\|_{L^2(0,T;\wW^1)}$ and $\|\ol u\|_{\wW^1}$) can be estimated by the norms on the left-hand side of \eqref{eq:UniformSecondOrderEstim} uniformly in $\eps\in (0,1]$. This is obvious in $\Omega\setminus \Gamma_t(\delta)$ and in $\Gamma_t(2\delta)$ one uses that $-\Delta u +\tfrac{1}{\eps^2} f''(c_A)u$ transforms to
  \begin{equation*}
    -\widetilde{\Delta}_{\Sigma} \tilde u +\widetilde{\mathcal{L}}_\eps \tilde u \qquad \text{in }\Sigma_{2\delta} 
  \end{equation*}
  up to lower order terms, cf.~\eqref{eq:TransformConvAC} below. Here $\|\oeps \widetilde{\Delta}_{\Sigma} \tilde u\|_{L^2(\Sigma_{2\delta})}\leq C \|\nabla_{\btau} u\|_{\wW^1(\Omega)}$.
  In particular, if $\ol\we, \ol u_0 \equiv 0$, after transformation in $\Gamma_t(2\delta)$ the terms
  \begin{equation*}
    \tilde\omega_\eps\partial_t \ol u, -\tilde \omega_\eps\widetilde{\Delta}_{\Sigma} \tilde u, \tilde\oeps\mathcal{L}_\eps \tilde u
  \end{equation*}
  in \eqref{eq:linNSAC3'} can be estimated uniformly in $L^2((0,T)\times \Sigma_{2\delta})$ by $\|\oeps \ol r_3\|_{L^2(Q_T)}$. In $Q_T\setminus\Gamma(\delta)$
  \begin{equation*}
    \partial_t u,\nabla^2 u, \tfrac{f''(c_A)}{\eps^2} u
  \end{equation*}
  are estimated uniformly in $L^2(Q_T\setminus \Gamma(\delta))$ by $\|\oeps \ol r_3\|_{L^2(Q_T)}$. In that sense we obtain a kind of \emph{uniform maximal regularity result}.
\end{rem}
\begin{proof*}{of Theorem~\ref{thm:SecondOrderEstimate}} 
\emph{Step 1: Suboptimal weighted higher order estimates for $ \ol{u} $ in $ H^2 $.}
First we test \eqref{eq:linNSAC3'} with $ - \Div (\oeps^4\nabla  \ol{u})+\oeps^4\frac{\alpha}{\eps^2} \ol u$. This yields
\begin{align}\nonumber
	&\int_\Omega \partial_t \nabla \ol{u}\cdot \nabla \ol u \oeps^4\,\sd x
	+ \int_\Omega (\partial_t \ol{u})\ol u   \oeps^4 \tfrac{\alpha}{\eps^2} \,\sd x 
	+ \int_\Omega \ve_A \cdot \nabla \ol{u}  \left(- \Div(\oeps^4\nabla  \ol{u})+\oeps^4\tfrac{\alpha}{\eps^2} \ol u\right)\,\sd x\\\nonumber
	&\qquad + \int_\Omega (\ol{\we}-\ol\we|_{\Gamma})\cdot \nabla c_A(- \Div (\oeps^4\nabla  \ol{u})+\oeps^4\tfrac{\alpha}{\eps^2} \ol u)\,\sd x+  \|\oeps^2\Delta \ol{u}\|_{L^2(\Omega)}^2
	+ \alpha^2 \|\tfrac{\oeps^2}{\eps^2} \ol{u}\|_{L^2(\Omega)}^2 \\\nonumber
	&\quad = 
	-  \int_\Omega \Delta \ol{u} \nabla \ol{u} \cdot \nabla(\oeps^4)\,\sd x
	-  \int_{\Gamma_t(2\delta)} \tfrac1{\eps^2}(V(\tfrac{d_{\Gamma_t}}{\eps}-h_\eps)+a_\eps) (- \Div(\oeps^4\nabla  \ol{u})+\oeps^4\tfrac{\alpha}{\eps^2} \ol u) \ol u\,\sd x 
	\\\nonumber
	&\qquad 
	- 2  \alpha \|\tfrac{\oeps^2}{\eps} \nabla \ol{u}\|_{L^2(\Omega)}^2 
	-  \int_\Omega \tfrac{\alpha}{\eps^2} \nabla \ol{u} \cdot \nabla(\oeps^4) \ol{u}\,\sd x
	+ \int_\Omega \ol{r}_3 (- \Div(\oeps^4\nabla  \ol{u})+\oeps^4\tfrac{\alpha}{\eps^2} \ol u)\, \sd x, 
\end{align}
where
\begin{align*}
	 \abs{\int_\Omega \Delta \ol{u} \nabla \ol{u} \cdot \nabla(\oeps^4)\,\sd x}
	& \leq 4  \int_\Omega \abs{\oeps^2\Delta \ol{u} \oeps \nabla \ol{u} \cdot \nabla \oeps}\,\sd x 
	 \leq \tfrac{1}{2} \|\oeps^2\Delta \ol{u}\|_{L^2(\Omega)}^2
	+ C  \|\oeps \nabla \ol{u}\|_{L^2(\Omega)}^2,\\
    \abs{\int_\Omega \tfrac{\alpha}{\eps^2} \nabla \ol{u} \cdot \nabla (\oeps^4) \ol{u}\,\sd x}
	& \leq 4  \alpha \int_\Omega \abs{\tfrac{\oeps^2}{\eps^2} \nabla \ol{u} \cdot \nabla \oeps \oeps \ol{u}}\,\sd x \leq \tfrac{ \alpha^2}{2} \|\tfrac{\oeps^2}{\eps^2} \ol{u}\|_{L^2(\Omega)}^2
	+ C  \|\oeps \nabla \ol{u}\|_{L^2(\Omega)}^2,
\end{align*}
and
\begin{align*}
	\abs{\int_\Omega \bv_A \cdot \nabla \ol{u} \Div (\oeps^4\nabla  \ol{u})\,\sd x}
	& \leq C 
	\norm{\oeps \nabla \ol{u}}_{L^2(\Omega)} \big(
	\norm{\oeps^2 \Delta \ol{u}}_{L^2(\Omega)} 
	+ \norm{\oeps \ptial{r} \ol{u}}_{L^2(\Omega)}\big), \\
	\abs{\int_\Omega \bv_A \cdot \nabla \ol{u} \alpha \tfrac{\oeps^4}{\eps^2} \ol{u}\,\sd x}
	& \leq C 
	\norm{\oeps \nabla \ol{u}}_{L^2(\Omega)}
	\normm{\tfrac{\oeps^2}{\eps^2} \ol{u}}_{L^2(\Omega)}, \\
	\abs{\int_\Omega (\ol{\we}-\ol\we|_{\Gamma})\cdot \nabla c_A \Div (\oeps^4\nabla  \ol{u})\,\sd x}
	& \leq C \norm{\nabla\ol\bw}_{L^2(\Omega)} \big(
	\norm{\oeps^2 \Delta \ol{u}}_{L^2(\Omega)}
	+ \norm{\oeps \ptial{r} \ol{u}}_{L^2(\Omega)}\big), \\
	\abs{\int_\Omega (\ol{\we}-\ol\we|_{\Gamma})\cdot \nabla c_A \alpha \tfrac{\oeps^4}{\eps^2} \ol{u}\,\sd x}
	& \leq C \norm{\nabla\ol\bw}_{L^2(\Omega)}
	\normm{\tfrac{\oeps^2}{\eps^2} \ol{u}}_{L^2(\Omega)}
\end{align*}
since $\|\oeps \nabla c_A\|_{L^\infty(\Omega)}$ is uniformly bounded.
For the second term on the right-hand side, we derive
	\begin{align*}
		&\left|
		\int_{\Gamma_t(2\delta)}
		\tfrac1{\eps^2} (V(\tfrac{d_{\Gamma}}{\eps}-h_\eps)+a_\eps)
		\left(
		-\Div (\oeps^4\nabla\ol u)
		+\alpha \tfrac{\oeps^4}{\eps^2}\ol u
		\right)\ol u\,\dx
		\right|  
		\leq
		C \|(\oeps^2\Delta\ol u, \oeps\nabla\ol  u, \ol u)\|_{L^2(\Omega)}\|\ol u\|_{L^2(\Omega)}
\end{align*}
because of $-\Div (\oeps^4\nabla\ol u) = -\oeps^4 \Delta\ol u -\nabla(\oeps^4)\cdot\nabla\ol u$ and $|\nabla(\oeps^4)|\leq C\oeps^3$,
where we have employed Lemma \ref{lem:f''-decomposition} and \eqref{eq:VEstim}, to obtain
	\begin{align*}
		\left\|
		\tfrac{\oeps^2}{\eps^2} (V(\tfrac{d_{\Gamma}}{\eps}-h_\eps)+a_\eps) \right\|_{L^\infty(\Gamma_t(2\delta))}
		\leq C.
	\end{align*}
Moreover,
\begin{align*}
	&\int_\Omega \partial_t \nabla \ol{u}\cdot \nabla \ol u   \oeps^4\,\sd x+ \int_\Omega (\partial_t \ol{u})\ol u \oeps^4 \tfrac{\alpha}{\eps^2}\,\sd x\\
	&\quad = \frac{\d}{\dt} \int_\Omega \left(\tfrac{|\nabla \ol u|^2}2 + \tfrac{\alpha|\ol u|^2}{2\eps^2}\right)\oeps^4\, \sd x
	- \int_\Omega \left(\tfrac{|\nabla \ol u|^2}2 + \tfrac{\alpha|\ol u|^2}{2\eps^2}\right)4\oeps^3 \partial_t \oeps\, \sd x,  
\end{align*}
where
\begin{equation*}
	\left|\int_0^{T}\int_\Omega \left(\tfrac{|\nabla \ol u|^2}2 + \tfrac{\alpha|\ol u|^2}{2\eps^2}\right)4\oeps^3 \partial_t \tilde r\, \sd x\dt\right|\leq C \left(\|\oeps \nabla \ol u\|_{L^2(Q_{T})}^2 + \|\tfrac{\oeps}{\eps} \ol u\|_{L^2(Q_{T})}^2\right).
\end{equation*}
Combining the estimates before, integrating in time, and using \eqref{eq:linNSAC3'} to estimate $\oeps^2 \pt \ol{u} \in L^2(Q_{T})$  we finally obtain
\begin{align*}
	&\|(\oeps^2 \nabla \ol u,\tfrac{\oeps^2}{\eps} \ol u)\|_{L^\infty(0,T;L^2)}+ \|(\oeps^2 \pt \ol{u},\oeps^2 \Delta \ol u,\tfrac{\oeps^2}{\eps} \nabla \ol u,\tfrac{\oeps^2}{\eps^2} \ol u)\|_{L^2(Q_T)}\\
	&\quad \leq C  \left(\|(\ol u, \oeps \nabla \ol u,\tfrac{\oeps}\eps \ol u) \|_{L^2(Q_T)}
	+ \|\oeps^2 \ol r_3\|_{L^2(Q_T)}
	+ \|\nabla \ol \bw\|_{L^2(Q_T)}+ \|(\oeps^2\nabla u_0, \tfrac{\oeps}\eps u_0)\|_{L^2(\Omega)}\right).
\end{align*}

\medskip

\noindent
\emph{Step 2: Refined estimates close to the interface.}
The starting point is
\begin{align}\label{eq:ConvAC}
	\partial_t \ol{u} +\ve_A\cdot \nabla \ol{u}&+ (\ol{\we} -\ol{\we}|_{\Gamma_t})\cdot \nabla c_A  =\left(\Delta \ol{u} - \tfrac{1}{\eps^{2}} f''(c_A)\ol{u}\right) +\ol{r}_3.
\end{align}
In $\Gamma(2\delta)$ we use the coordinate transform given by $X(.,.,t)\colon (-2\delta,2\delta)\times \Sigma\to \Gamma_t(2\delta)$. Moreover, we localize close to $\Gamma$ with the aid of a cut-off function $\zeta\in C^\infty_0(\R)$ with $\supp \zeta \subset [-2\delta,2\delta]$ and $\zeta(r)=1$ for all $r\in [-\delta,\delta]$ as before. 
More precisely, we define
\begin{align*}
	\tilde{u}(r,s,t):= \zeta(r)\ol{u}(X(r,s,t),t),\         \widetilde{\we}(r,s,t):= \ol{\we}(X(r,s,t),t), \ \tilde{r}_3(r,s,t):= \zeta(r)\ol{r}_3(X(r,s,t),t)
\end{align*}
for all $(r,s,t)\in \Sigma_{2\delta}\times [0,T]$ and $\eps\in (0,1]$.
 This yields the relations
\begin{align*}
  \zeta(d_{\Gamma_t}(x))\partial_t \ol{u}(x,t) &= \partial_t d_{\Gamma_t} (x)\partial_r \tilde{u}(r,s,t) + \partial_t S(x,t)\cdot \nabla_\Sigma \tilde{u}(r,s,t)+ \partial_t \tilde{u}(r,s,t)\\
  &\quad -\partial_t d_{\Gamma_t}(x)\zeta'(d_{\Gamma_t}(x)) \ol{u}(x,t),\\
	\zeta(d_{\Gamma_t}(x))\nabla \ol{u}(x,t) &= \nabla d_{\Gamma_t}(x) \partial_r \tilde{u}(r,s,t) +D_xS(x,t)^T\nabla_\Sigma \tilde{u}(r,s,t)- \nabla d_{\Gamma_t}(x)\zeta'(d_{\Gamma_t}(x)) \ol{u}(x,t),\\
  \zeta(d_{\Gamma_t}(x))\Delta \ol{u}(x,t) &= \left(\Delta d_{\Gamma_t}(x)\right)\partial_r \tilde{u}(r,s,t) + \Delta S(x,t)\cdot \nabla_\Sigma \tilde{u}(r,s,t) +\tilde{g}_1(r,s,t)\\
    &\quad +\sum_{i,j=1}^d\nabla S_i(x,t)\cdot \nabla S_j(x,t) (\nabla_\Sigma)_i (\nabla_\Sigma)_j  \tilde{u}(r,s,t) + \partial_r^2 \tilde{u}(r,s),\\
  \tilde{g}_1(r,s,t)&:= - \Div_x (\nabla d_{\Gamma_t}(x,t)\zeta'(r) \ol{u}(x,t)).
\end{align*}
for all $(r,s)\in \Sigma_{2\delta}, t\in [0,T]$, and $x=X(r,s,t)$.
Using $c_A= \theta_0 + \eps^2c_{A,2+}$ in $ \Gamma_t(2\delta) $ and a Taylor expansion of  $f''(c_A)$ at $\theta_0$ 
we obtain that \eqref{eq:ConvAC} transforms to
\begin{align}\label{eq:TransformConvAC}
	\partial_t \tilde{u} - \widetilde{\Delta}_{\Sigma} \tilde{u}+ \widetilde{\mathcal{L}}_\eps \tilde{u} =\tilde{g} \quad \text{in } \Sigma_{2\delta}\times (0,T), 
\end{align}
where $\widetilde{\Delta}_{\Sigma} \tilde{u}$ is defined in \eqref{eq:Deltaeps}, $\widetilde{\mathcal{L}}_\eps \tilde{u} $ is defined in \eqref{eq:Leps}, and 
\begin{align*}
	\tilde{g}(r,s,t)&= \tilde{r}_3(r,s,t)- \tilde{\mathbf{a}}\cdot\nabla_\Sigma \tilde{u}-\partial_r \tilde{u}\left(\partial_t d_{\Gamma_t}+ \ve_A\cdot \nabla d_{\Gamma_t} - \Delta d_{\Gamma_t}\right) - \tfrac1{\eps^2}(f''(c_A) - f''(\theta_0))\tilde u\\
&\quad - (\widetilde{\we}-\widetilde{\we}|_{r=0}) \cdot \nabla c_A - \tilde{g}_1(r,s,t) +(\partial_t d(x,t)+\ve_A(x,t)\cdot \nabla d(x,t))\zeta'(r) \ol{u}(x,t),\\
  	\tilde{\mathbf{a}} &= \partial_tS+ \ve_A\cdot (D_xS)^T-\Delta S|_{(x,t)}-\sum_{i,j=1}^d(\nabla_\Sigma(\nabla S_i\cdot \nabla S_j|_{(x,t)}))_i e_j.
\end{align*}
Here $\partial_t d_{\Gamma_t}+ \ve_A\cdot \nabla d_{\Gamma_t} - \Delta d_{\Gamma_t}=O(|r|+\eps) = O(\toeps)$ and $\tfrac1{\eps^2}(f''(c_A) - f''(\theta_0)), (\widetilde{\we}-\widetilde{\we}|_{r=0}) \cdot \nabla c_A=O(1)$ in $L^\infty(Q_T)$. Therefore $\tilde{g}$ can be estimated as
\begin{align}\notag
  &\|\toeps \tilde g\|_{L^2(0,T;L^2(\Sigma_{2\delta}))}\\\label{eq:gEstim1}
  &\quad \leq  C\left(\|\toeps \ol r_3\|_{L^2(\Omega\times (0,T))}+ \|\toeps\ol \we\|_{L^2(0,T; H^1)}+ \|(\toeps \ol u, \toeps \nabla \ol u, \toeps^2 \nabla ^2 \ol u)\|_{L^2(Q_T)} \right),\\\notag
    &\|(\toeps^2 \nabla_\Sigma \tilde g, \toeps^2 \partial_r \tilde g, \tfrac{\toeps^2}\eps \tilde g) \|_{L^2(0,T;L^2(\Sigma_{2\delta}))}\leq  C\left(\|(\toeps^2 \nabla \ol r_3, \tfrac{\toeps^2}\eps \ol r_3)\|_{L^2(\Omega\times (0,T))}\right.\\\label{eq:gEstim2}
  &\quad\left. + \|\toeps \ol \we\|_{L^2(0,T; H^1(\Omega))}
  + \|(\tfrac{\toeps^2}\eps\ol u, \tfrac{\toeps^2}\eps \nabla \ol u, \toeps^2 \nabla^2 \ol u, \toeps^3 \nabla ^3 \ol u)\|_{L^2(\Omega\times (0,T))} \right).
\end{align}
For the following estimate it is essential to use an equivalent form of \eqref{eq:TransformConvAC}. Namely, since $\widetilde{\cL}_\eps \tilde{u} = \widetilde{\cL}_\eps (P\tilde{u})$, we have
    \begin{align}
        \label{eq:TransformConvAC-proj}
	       \partial_t \tilde{u} - \widetilde{\Delta}_{\Sigma} \tilde{u}+ \widetilde{\mathcal{L}}_\eps (P\tilde{u}) =\tilde{g} \quad \text{in } \Sigma_{2\delta}\times (0,T),
    \end{align}
    where $P\tilde{u} = \tilde{u} - Q\tilde{u}$ denotes the projection defined by \eqref{eq:DefnPu} in Corollary \ref{cor:Pu}, which is orthogonal to $\theta_0'$ in $L^2(\Sigma_{2\delta})$.
    Now we test \eqref{eq:TransformConvAC-proj} with $\toeps^2\partial_t\tilde{u} = \toeps^2 \pt (P \tilde{u} + Q \tilde{u})$. This yields
	\begin{align*}
		&\frac12\frac{\d}{\dt} \left[B_\Sigma^1(\tilde{u}) + B_\eps^1 (P\tilde{u})\right]
		+ \|\toeps\partial_t \tilde{u}\|_{L^2(\Sigma_{2\delta})}^2 = \frac12\sum_{i,j=1}^d\int_{\Sigma_{2\delta}}\partial_t a_{i,j}(\nabla_{\Sigma}\tilde{u})_i (\nabla_\Sigma\tilde{u})_j\toeps^2 \, \d(r,s)\\
        & \qquad
        + \frac12 \int_{\Sigma_{2\delta}} \tfrac{\toeps^2}{\eps^2} f'''(\theta_0(\tfrac{r}{\eps} - h_\eps)) \theta_0'(\tfrac{r}{\eps} - h_\eps) (P\tilde{u})^2 (-\pt h_\eps) \,\d(r,s) \\
        &\qquad - \int_{\Sigma_{2\delta}} \pr (P\tilde{u}) 2 r \pt \tilde{u}\,\d(r,s) 
        - \int_{\Sigma_{2\delta}} \widetilde{\cL}_\eps (P\tilde{u}) \pt(Q\tilde{u}) \toeps^2 \,\d(r,s)
        + \int_{\Sigma_{2\delta}} \toeps \tilde{g} \toeps \partial_t \tilde{u}\, \d(r,s),
	\end{align*}
        where $B_\Sigma^1$ and $ B_\eps^1 $ are defined in \eqref{bi-linear-3}. By Corollary \ref{cor:Pu} it holds
        \begin{align*}
            \abs{\int_{\Sigma_{2\delta}} \pr (P\tilde{u}) 2 r \pt \tilde{u}\,\d(r,s)}
            \leq C \normm{\pr(P\tilde{u})}_{L^2(\Sigma_{2\delta})} \normm{\toeps \pt \tilde{u}}_{L^2(\Sigma_{2\delta})} \leq C \normm{u}_{V_t^\eps(2\delta)} \normm{\toeps \pt \tilde{u}}_{L^2(\Sigma_{2\delta})}.
        \end{align*}
        and
        \begin{align*}
            \left|\frac12 \int_{\Sigma_{2\delta}} \tfrac{\toeps^2}{\eps^2} f'''(\theta_0(\tfrac{r}{\eps} - h_\eps)) \theta_0'(\tfrac{r}{\eps} - h_\eps) (P\tilde{u})^2 (-\pt h_\eps) \,\d(r,s)\right| \leq C\|P\tilde u\|_{H^1(\Sigma_{2\delta})}^2
        \end{align*}
        since $\tfrac{\toeps^2}{\eps^2}\theta_0'(\tfrac{r}\eps-h_\eps)$ is uniformly bounded.
    Moreover, 
    \begin{align*}
        \int_{\Sigma_{2\delta}} \widetilde{\cL}_\eps (P\tilde{u}) \pt(Q\tilde{u}) \toeps^2 \,\d(r,s)
        & = \int_{\Sigma_{2\delta}} \widetilde{\cL}_\eps (P\tilde{u}) (Q\pt\tilde{u}) \toeps^2 \,\d(r,s)
        + \int_{\Sigma_{2\delta}} \widetilde{\cL}_\eps (P\tilde{u}) [\pt,Q]\tilde{u} \toeps^2 \,\d(r,s),
    \end{align*}
    where by integration by parts we have
    \begin{align*}
        & \left|\int_{\Sigma_{2\delta}} \widetilde{\cL}_\eps (P\tilde{u}) (Q\pt\tilde{u}) \toeps^2 \,\d(r,s)\right| \\
        &\quad \leq \left|\int_{\Sigma_{2\delta}} \pr (P\tilde{u}) \pr \big(\toeps^2 Q\pt\tilde{u}\big) \,\d(r,s)\right|
        + \left| \int_{\Sigma_{2\delta}} \tfrac{f''(\theta_0)}{\eps^2} P\tilde{u} (Q\pt\tilde{u}) \toeps^2 \,\d(r,s)\right| \\
        &\quad \leq C \normm{\pr(P \tilde{u})}_{L^2(\Sigma_{2\delta})} \normm{\pr \big(\toeps^2 Q\pt\tilde{u}\big)}_{L^2(\Sigma_{2\delta})}
            + \tfrac{C}{\eps}\normm{P \tilde{u}}_{L^2(\Sigma_{2\delta})} \normm{\tfrac{\toeps^2}{\eps} Q \pt \tilde{u}}_{L^2(\Sigma_{2\delta})} \\
        &\quad \leq C \big(\normm{P \tilde{u}}_{H^1(\Sigma_{2\delta})} 
        + \tfrac{1}{\eps}\normm{P \tilde{u}}_{L^2(\Sigma_{2\delta})}\big) \normm{\toeps \pt \tilde{u}}_{L^2(\Sigma_{2\delta})}
    \end{align*}
    because of
    \begin{align*}
        \int_{\Sigma_{2\delta}} \abs{\tfrac{\toeps^2}{\eps} Q\pt\tilde{u}}^2 \d(r,s)
        \leq C\normm{\eps\pt \tilde{u}}_{L^2(\Sigma_{2\delta})}^2 
        \leq C' \normm{\toeps \pt \tilde{u}}_{L^2(\Sigma_{2\delta})}^2 
    \end{align*}
    since $\left\|\tfrac{\toeps^2}{\eps} \tfrac1{\sqrt{\sigma \eps}}\theta'_0(\tfrac{.}{\eps}-h_\eps)\right\|_{L^2(-2\delta,2\delta)}\leq C\eps$.
    The commutator $[\partial_t,Q]$ can be controlled in the following standard way as it is lower-order:
    \begin{align*}
        &\left|\int_{\Sigma_{2\delta}} \widetilde{\cL}_\eps (P\tilde{u}) [\pt,Q]\tilde{u} \toeps^2 \,\d(r,s)\right|\\
        &\quad \leq \left|\int_{\Sigma_{2\delta}} \pr (P\tilde{u}) \tfrac{1}{\eps \sigma} 
        \pr(\theta_0''(\tfrac{r}\eps-h_\eps) \toeps^2) (-\pt h_\eps) \int_{-2\delta}^{2\delta}\theta_0'(\tfrac{\tilde r}\eps-h_\eps)\tilde{u}(\tilde r,s) \d \tilde r \,\d(r,s)\right| \\
        & \qquad
        + \left|\int_{\Sigma_{2\delta}} \pr (P\tilde{u}) \tfrac{1}{\eps \sigma} 
        \pr(\theta_0'(\tfrac{r}\eps-h_\eps) \toeps^2)  \int_{-2\delta}^{2\delta}\theta_0''(\tfrac{\tilde r}\eps-h_\eps)(-\pt h_\eps) \tilde{u}(\tilde r,s) \d \tilde r \,\d(r,s)\right| \\
        & \qquad 
        + \left|\int_{\Sigma_{2\delta}} \tfrac{\toeps^2}{\eps^2} f''(\theta_0(\tfrac{r}\eps-h_\eps)) P\tilde{u}(\tilde r,s) \tfrac{1}{\eps \sigma} 
        \theta_0''(\tfrac{r}\eps-h_\eps) (-\pt h_\eps)  \int_{-2\delta}^{2\delta}\theta_0'(\tfrac{\tilde r}\eps-h_\eps) \tilde{u}(\tilde r,s) \d \tilde r \,\d(r,s)\right| \\
        & \qquad 
        + \left|\int_{\Sigma_{2\delta}} \tfrac{\toeps^2}{\eps^2} f''(\theta_0(\tfrac{r}\eps-h_\eps)) P\tilde{u} \tfrac{1}{\eps \sigma} 
        \theta_0'(\tfrac{r}\eps-h_\eps)  \int_{-2\delta}^{2\delta}\theta_0''(\tfrac{\tilde r}\eps-h_\eps)(-\pt h_\eps) \tilde{u}(\tilde{r},s) \d \tilde r \,\d(r,s)\right| \\
        &\quad \leq C \Big(
        \normm{P\tilde{u}}_{H^1(\Sigma_{2\delta})}^2
        + \normm{\tilde{w}}_{H^1(\Sigma)}^2
        \Big).
    \end{align*}
    Hence integrating over $(0,T)$ together with Young's inequality and Corollary \ref{cor:Pu} leads to
	\begin{align*}
		&
        \sup_{t\in [0,T]} \left[B_\Sigma^1 (\tilde{u}(t)) + B_\eps^1 (P\tilde{u}(t))\right] + \|\toeps\partial_t\tilde{u}\|_{L^2(\Sigma_{2\delta}\times (0,T))}^2\\
		&\quad \leq C\left(\|\toeps \tilde{g}\|_{L^2(\Sigma_{2\delta}\times (0,T))}^2+\|\toeps  \nabla_\Sigma\tilde{u}\|_{L^2(\Sigma_{2\delta}\times (0,T))}^2
        +\|u\|_{L^2(0,T;\Veps(2\delta))}^2+ B_\Sigma^1 (\tilde{u}_0)
        + B_\eps^1 (P\tilde{u}_0)\right).
	\end{align*}
    In view of the elliptic regularity \eqref{eq:EllipticEstim} close to the interface, one infers from \eqref{eq:TransformConvAC}
        \begin{align}\nonumber
		& \|(\toeps\nabla_\Sigma^2\tilde{u}, \toeps\partial_r \nabla_\Sigma \tilde{u}, \toeps\widetilde{\mathcal{L}}_\eps \tilde u,\tfrac{\toeps}{\eps} \nabla_\Sigma \tilde{u})\|_{L^2(\Sigma_{2\delta})}^2\\
		&\quad \leq C\left(\|\toeps \tilde{g}\|_{L^2(\Sigma_{2\delta})}^2
        +\|\toeps \pt \tilde{u}\|_{L^2(\Sigma_{2\delta})}^2
        +\|(\tilde u,\nabla_\Sigma \tilde{u}, \toeps\partial_r \tilde{u})\|_{L^2(\Sigma_{2\delta})}^2\right).
        \label{eq:second-elliptic}
	\end{align}
    where the right-hand side terms are well controlled by above estimates.

	Now a combination with \eqref{eq:CoercivBepsk}, 
     \eqref{ineq:uw energy-1},
    \eqref{eq:gEstim1}, and \eqref{eq:second-elliptic} yields
	\begin{align*}
          &\|(\toeps \nabla \tilde{u}, \tfrac{\toeps}\eps \tilde u)\|_{L^\infty(0,T;L^2(\Sigma_{2\delta}))}
          + \|(\toeps\partial_t\tilde{u},\toeps \nabla\nabla_\Sigma \tilde{u},\tfrac{\toeps}\eps\nabla_\Sigma\tilde{u}, \toeps\mathcal{L}_\eps \tilde{u})\|_{L^2(\Sigma_{2\delta}\times (0,T))}\\
		&\quad \leq C\left(\|\toeps \ol{r}_3 \|_{L^2(\Omega\times (0,T))}+\|(\toeps \nabla  \ol u, \toeps^2\nabla^2 \ol{u})\|_{L^2(\Omega\times (0,T))}+\|\ol u\|_{L^\infty(0,T;L^2(\Omega))}\right.\\
          &\qquad\qquad \left. +\|u\|_{L^2(0,T;\Veps(2\delta))}^2 + \|\toeps \ol \we \|_{L^2(0,T;H^1(\Omega))}+ \|(\toeps \nabla \tilde{u}_0, \tfrac{\toeps}\eps \tilde u_0)\|_{L^2(\Sigma_{2\delta})}\right).
	\end{align*}
        A collection with the estimates from the first step and Gronwall's lemma finishes the proof of the theorem.
\end{proof*}

\subsection{Weighted Second Order Estimates for $\ol{\bw}$}
In this section, we derive the $\oeps$-weighted second-order estimates for $ \ol{\bw} $. Before heading to the main result, let us introduce a weighted second-order estimate for the Stokes equation, which might be of independent interest.
\begin{lem}
	\label{lem:weighted-stokes-estimate}
	Let $M>0$, $c_A\colon Q_{T_0}\to \R$ be continuously differentiable such that $\|\oeps \nabla c_A\|_{L^\infty(Q_{T_0})}\leq M$ and let $(\tw,\ol q)\in H^1_{0}(\Omega)^d \times L^2_{(0)}(\Omega)$ be a solution of
	\begin{subequations}
	    \begin{alignat}{2}
		\label{eq:stokes-tw-qbar}
		-\Div\bigl(2\nu(c_A)D\tw\bigr)+\nabla\ol q&=\mathbf f
		&\qquad &\text{in } \Omega,\\
		\Div\tw&=0 &&\text{in } \Omega,\\
		\tw&=0&	&\text{on } \partial \Omega
	\end{alignat}
	\end{subequations}
	for some $\mathbf f\in L^2(\Omega)^d$.
	Then $\tw\in H^2 (\Omega)^d$, $\ol q\in H^1(\Omega)$ and satisfies
	\begin{equation}
		\label{eq:weighted-pressure-laplacian-estimate}
		\|\oeps \nabla^2\tw\|_{L^2(\Omega)}
		+
		\|\oeps\nabla\ol q\|_{L^2(\Omega)}
		\leq
		C\left(
		\|\oeps \mathbf f\|_{L^2(\Omega)}
		+
		\|\nabla\tw\|_{L^2(\Omega)}
		\right),
	\end{equation}
	where $C=C(M)$ is independent of $\eps$, $t\in [0,T_0]$, $\tw, \ol q$, $\mathbf f$ and $c_A$ satisfying the assumptions.
\end{lem}

\noindent
\begin{proof}
First of all, $\tw\in H^2(\Omega)^d$ and $\ol q\in H^1(\Omega)$ follows from standard regularity results for the Stokes equation with variable viscosity and fixed $\eps>0$, cf.\ e.g.\ \cite[Lemma~4]{ModelH}. In view of Proposition~\ref{prop:weightedEllipticReg} it is sufficient to prove the estimate for $\|\oeps \nabla \ol q\|_{L^2(\Omega)}$. 
It is enough to show the estimate for smooth solutions; the general case follows by a standard approximation argument. 
Since $\Div\tw=0$, the equation \eqref{eq:stokes-tw-qbar} may be written as
\begin{equation}
	\label{eq:reduced-stokes-localized-fixed}
	-\nu(c_A)\Delta\tw+
	\nabla\ol q
	=
	\mathbf f+2D\tw\,\nabla\nu(c_A)
	=:\widetilde{\mathbf f}.
\end{equation}
Since $\oeps \nabla c_A$ is bounded by $M$ and $\nu \in C^1(\R)$,
\begin{equation}
	\label{eq:ftilde-weighted-bound-localized-fixed}
	\|\oeps\widetilde{\mathbf f}\|_{L^2(\Omega)}
	\leq
	C(M) \big(\|\oeps \mathbf f\|_{L^2(\Omega)}
	+
	\|\nabla\tw\|_{L^2(\Omega)}\big).
\end{equation}

\noindent
\emph{Step 1: Interior pressure estimate.}
Choose $\eta\in C_0^\infty(\Omega)$ such that
\[
0\leq\eta\leq 1,
\qquad
\eta=1\quad\text{on }\bigcup_{t\in [0,T_0]} \ol{\Gamma_t(\delta)}.
\] 
Testing \eqref{eq:reduced-stokes-localized-fixed} by $\eta^2\oeps^2\nabla\ol q$ gives
\begin{align}
	\int_\Omega \eta^2\oeps^2|\nabla\ol q|^2\dx
	&=
	\int_\Omega \widetilde{\mathbf f}\cdot
	\eta^2\oeps^2\nabla\ol q\dx
	+
	\int_\Omega \nu(c_A)\Delta\tw\cdot
	\eta^2\oeps^2\nabla\ol q\dx 
	=: I_1+I_2 .
	\label{eq:localized-pressure-identity-fixed}
\end{align}
Clearly,
\begin{equation}
	\label{eq:I1-localized-fixed}
	|I_1|
	\leq
	\|\eta\oeps\widetilde{\mathbf f}\|_{L^2(\Omega)}
	\|\eta\oeps\nabla\ol q\|_{L^2(\Omega)} .
\end{equation}

To estimate $I_2$, we set
$
a:=\nu(c_A)\eta^2\oeps^2 .
$
Since $a$ is compactly supported in $\Omega$, integration by parts gives
\begin{align*}
	I_2
	&=
	-\sum_{i,j=1}^d\int_\Omega a\,\partial_{x_j}\tw_i\,\partial_{x_j}\partial_{x_i}\ol q\dx
	-\sum_{i,j=1}^d
	\int_\Omega \partial_{x_j} a\,\partial_{x_j}\tw_i\,\partial_{x_i}\ol q\dx .
\end{align*}
Integrating the first term by parts in $x_i$ and using
$\Div\tw=0$, we obtain
\begin{align*}
	&-\sum_{i,j=1}^d\int_\Omega a\,\partial_{x_j}\tw_i\,\partial_{x_j}\partial_{x_i}\ol q\dx
=\sum_{i,j=1}^d
	\int_\Omega \partial_{x_i} a\,\partial_{x_j}\tw_i\,\partial_{x_j}\ol q\dx .
\end{align*}
Hence
\begin{equation}
	\label{eq:I2-identity-fixed}
	I_2
	=\sum_{i,j=1}^d
	\int_\Omega \partial_{x_i} a\,\partial_{x_j}\tw_i\,\partial_{x_j}\ol q\dx
	-\sum_{i,j=1}^d
	\int_\Omega \partial_{x_j} a\,\partial_{x_j}\tw_i\,\partial_{x_i}\ol q\dx .
\end{equation} 
Moreover,
\[
\nabla a
=
\eta^2\nabla\bigl(\nu(c_A)\oeps^2\bigr)
+
2\eta\nu(c_A)\oeps^2\nabla\eta .
\]
With the choice of $\eta$, we get
\begin{equation}
	\label{eq:gradient-a-bound-fixed}
	|\nabla a|
	\leq
	C\eta\oeps .
\end{equation}
Consequently,
\begin{equation}
	\label{eq:I2-localized-fixed}
	|I_2|
	\leq
	C
	\|\nabla\tw\|_{L^2(\Omega)}
	\|\eta\oeps\nabla\ol q\|_{L^2(\Omega)} .
\end{equation}
Combining \eqref{eq:localized-pressure-identity-fixed}, \eqref{eq:I1-localized-fixed}, \eqref{eq:I2-localized-fixed}, and \eqref{eq:ftilde-weighted-bound-localized-fixed}, we obtain
\begin{equation}
	\label{eq:interior-pressure-weighted-final-fixed}
	\|\eta\oeps\nabla\ol q\|_{L^2(\Omega)}
	\leq
	C \big(\|\oeps \mathbf f\|_{L^2(\Omega)}
	+
	\|\nabla\tw\|_{L^2(\Omega)}\big).
\end{equation}

\noindent
\emph{Step 2: Pressure estimate close to the boundary.}
Let $\Omega_\kappa$, $\kappa \in (0,2\delta']$ and $\delta'>0$ be as in the proof of Proposition~\ref{prop:weightedEllipticReg} and choose some $\eta\in C^\infty(\ol\Omega)$ with $\eta \equiv 1$ in $\Omega_{\delta'}$ and $\eta\equiv 0$ in $\Omega\setminus \Omega_{2\delta'}$. Multiplying \eqref{eq:reduced-stokes-localized-fixed} by $\eta$, we obtain
\begin{alignat*}{2}
	-\Delta(\nu(c_A) \eta\tw)+
	\nabla(\eta \ol q-m)
  &= \eta \widetilde{\mathbf f}- 2\nabla (\eta \nu(c_A))\cdot \nabla \tw - \tw \Delta (\eta \nu(c_A))
  + q \nabla \eta &\quad &\text{in }\Omega_{2\delta'},\\
  \Div (\nu(c_A) \eta\tw) &= \tw\cdot \nabla (\nu(c_A) \eta)&\quad &\text{in }\Omega_{2\delta'},\\
  \nu(c_A) \eta\tw|_{\partial\Omega} &=0 &&\text{in }\partial\Omega_{2\delta'},
\end{alignat*}
where  $m= \tfrac1{|\Omega|}\int_{\Omega}\eta \ol q\sd x$. Hence standard results for the Stokes system yield
\begin{align*}
  &\|\nabla^2 \tw\|_{L^2(\Omega_{\delta'})}+          \|\nabla \ol q\|_{L^2(\Omega_{\delta'})}\leq C\left(\|\nu(c_A) \eta \tw\|_{H^2(\Omega_{2\delta'})}+          \|\eta\ol q-m\|_{H^1(\Omega_{2\delta'})} \right)\\
  & \qquad \leq C\left(\|\widetilde{\mathbf f}\|_{L^2(\Omega_{2\delta'})}+ \|\nabla \tw\|_{L^2(\Omega)}\right)\leq  C\left(\|\oeps \mathbf f\|_{L^2(\Omega)}+ \|\nabla \tw\|_{L^2(\Omega)}\right),
\end{align*}
since $\nabla (\eta \nu(c_A))$ and $\Delta (\eta \nu(c_A))$ are uniformly bounded. Since $\oeps \in [1,2]$ in $\Omega_{\delta'}$, this finishes the proof.
\end{proof}

With the aid of the previous lemma we prove the following higher-order estimate for $ \ol{\bw} $.
\begin{thm}\label{thm:HigherOrderW}
	Let $ \ol \bw \in L^2(0,T;H^2(\Omega)\cap H^1_0(\Omega))^d \cap H^1(0,T;L^2(\Omega))^d $ be a solution to \eqref{eq:linNSAC1'}-\eqref{eq:linNSAC2'} with $\ol\bw|_{t=0}= \ol\bw_0$. Then it holds
	\begin{align*}
		& \norm{\oeps \nabla \ol{\bw}}_{L^\infty(0,T;L^2)}
		+ \norm{\oeps \pt \ol{\bw}}_{L^2(Q_{T})}
		+ \norm{\oeps \nabla^2 \ol{\bw}}_{L^2(Q_{T})}
		 \leq C \Big(\|(\ol u,\oeps \nabla \ol u,\nabla \tw)\|_{L^2(Q_{T})}\\
        &\ + \|\oeps \ol{\mathbf{r}}_1\|_{L^2(Q_{T})}
                  + \| \ol{\mathbf{r}}_1 \|_{L^2(0,T;V_\sigma')}
              + \|(\ol r_2,\oeps \nabla \ol r_2)\|_{L^2(Q_T)}
          + \|\partial_t \ol r_2\|_{L^2(0,T;H^{-1}_{(0)}(\Omega))}+\|\ol{\we}_0\|_{H^1(\Omega)} \Big),
	\end{align*}
	where $ C > 0 $ is independent of $ \eps \in (0,\eps_1] $ and $T\in (0,T_0]$.
\end{thm}
\begin{proof}
	To justify the estimate, we recall \eqref{eq:linNSAC1'-1} and \eqref{eq:linNSAC2'-1}:
	\begin{subequations}
	    \begin{align}
    		\nonumber
    		\partial_t \widetilde{\we} +\ve_A\cdot \nabla \widetilde\we&+\widetilde\we\cdot \nabla \ve_A-\Div(2\nu(c_A)D\widetilde\we) - \Div(2\nu'(c_A)\ol uD\ve_A) +\nabla \ol{q}\\ 
    		\label{eq:NS-tw-2}
    		& = -\eps \Div (\nabla \ol{u} \otimes \nabla c_A+\nabla c_A \otimes \nabla \ol{u})+\ol{\mathbf{r}}_1
    		+ \mathbf{g}, \\ \label{eq:NS-tw-3}
    		\Div \tw & = 0,
    	\end{align}
	\end{subequations}
	and $\tw|_{\partial \Omega} = 0$, $\tw|_{t=0} = 0$, where $ \mathbf{g} \coloneqq \partial_t \we_1+\ve_A\cdot \nabla \we_1+\we_1\cdot \nabla \ve_A -\Div(2\nu(c_A)D\we_1) $. By \eqref{eq:w1Estim}
	\begin{align*}
        & \|(\nabla \we_1,\partial_t \we_1,  \oeps \nabla^2\we_1)\|_{L^2(Q_T)} \leq C \left(\|(\ol r_2,\oeps \nabla \ol r_2)\|_{L^2(Q_T)}
          + \|\partial_t \ol r_2\|_{L^2(0,T;H^{-1}_{(0)})}
          +\|\ol\we_0\|_{H^1(\Omega)}\right),
	\end{align*}
	which implies
        \begin{align*}
          & \|\oeps\mathbf{g}\|_{L^2(Q_T)}
          + \|\mathbf{g}\|_{L^2(0,T;V_\sigma')}  \leq C(T_0) \left(\|(\ol r_2,\oeps \nabla \ol r_2)\|_{L^2(Q_T)}
          + \|\partial_t \ol r_2\|_{L^2(0,T;H^{-1}_{(0)})}
          +\|\ol\we_0\|_{H^1(\Omega)}\right).
        \end{align*}
        
        	Now testing \eqref{eq:NS-tw-2} with $ P_\sigma(\oeps^2 \pt \tw) $, where $P_\sigma\colon L^2(\Omega)^d\to L^2_\sigma(\Omega)$ denotes the Helmholtz projection, yields 
	\begin{align}
		& \int_\Omega 2 \nu(c_A) D \tw : D (\oeps^2 \pt \tw) \sd x - \int_\Omega 2 \nu(c_A) D \tw : \nabla^2 p \dx
		+ \int_\Omega \oeps^2 \abs{\pt \tw}^2 \dx \nonumber\\
		& \quad= \int_{\partial\Omega} 2\nu(c_A)D\tw : \no_{\partial\Omega}\otimes \nabla p\sd\sigma - \int_\Omega (\bv_A \cdot \nabla) \tw \cdot P_\sigma(\oeps^2 \pt \tw) \dx
		- \int_\Omega (\tw \cdot \nabla) \bv_A \cdot P_\sigma(\oeps^2 \pt \tw) \dx \nonumber\\
		& \qquad + \int_\Omega \Div (2\nu'(c_A)\ol uD\ve_A) \cdot (\oeps^2 \pt \tw) \dx
		+ \int_\Omega 2\nu'(c_A)\ol uD\ve_A:  \nabla^2 p \dx \nonumber\\
		& \qquad - \int_\Omega \eps  \Div (\nabla \ol{u} \otimes_s \nabla c_A) \cdot P_\sigma(\oeps^2 \pt \tw) \dx
		+ \int_\Omega (\ol{\mathbf{r}}_1
                  + \mathbf{g}) \cdot P_\sigma(\oeps^2 \pt \tw) \dx\label{energy-0000}
	\end{align}
        since $P_\sigma (\oeps^2 \pt \tw)= \oeps^2 \pt \tw-\nabla p$, where
        \begin{alignat*}{2}
          \Delta p &= 2\oeps \nabla  \oeps\cdot \partial_t \tw &\qquad & \text{in }\Omega,\\
          \no\cdot \nabla p|_{\partial\Omega}&=0 &\qquad &  \text{on }\partial\Omega.
        \end{alignat*}
        Hence
        \begin{equation*}
          \|\nabla p\|_{H^1(\Omega)}\leq C\|\oeps \pt \tw\|_{L^2(\Omega)}
        \end{equation*}
        and therefore
        \begin{align*}
          \left|\int_\Omega 2 \nu(c_A) D \tw : \nabla^2 p \dx\right|&\leq C\|\nabla \tw\|_{L^2(\Omega)}\|\oeps \pt \tw\|_{L^2(\Omega)}, \\
          \left|\int_\Omega 2\nu'(c_A)\ol uD\ve_A: \nabla^2 p \dx\right|&\leq C\|\ol{u}\|_{L^2(\Omega)}\|\oeps \pt \tw\|_{L^2(\Omega)}.
        \end{align*}
	Moreover, we have
	\begin{align*}
		\int_\Omega 2 \nu(c_A) D \tw : D (\oeps^2 \pt \tw) \dx
		& = \ddt \int_\Omega \nu(c_A) \abs{\oeps D \tw}^2 \dx 
		+ \int_\Omega 2 \oeps \nu(c_A) D \tw : (\nabla \oeps \otimes_s \pt \tw) \dx \\
		& \quad 
		- \int_\Omega \big(2 \nu'(c_A) \pt c_A \oeps^2 + 4 \nu(c_A) \oeps \pt \oeps\big) \abs{D \tw}^2 \dx, 
	\end{align*}
        where $2 \nu'(c_A) \pt c_A \oeps^2 + 4 \nu(c_A) \oeps \pt \oeps$ is bounded in $L^\infty(Q_{T})$,
    and
	\begin{align*}
		\abs{\int_\Omega \Div (2\nu'(c_A)\ol uD\ve_A) \cdot \oeps^2 \pt \tw \dx}
		& \leq C(\norm{\ol{u}}_{L^2(\Omega)} + \norm{\oeps \nabla \ol{u}}_{L^2(\Omega)}) \norm{\oeps \pt \tw}_{L^2(\Omega)}, \\
		\abs{\int_\Omega \bv_A \cdot \nabla \tw \cdot P_\sigma( \oeps^2 \pt \tw) \dx}
		& \leq \norm{\bv_A}_{L^\infty(\Omega)} \norm{\nabla \tw}_{L^2(\Omega)} \norm{\oeps \pt \tw}_{L^2(\Omega)} , \\
		\abs{\int_\Omega \tw \cdot \nabla \bv_A \cdot P_\sigma(\oeps^2 \pt \tw) \dx}
		& \leq \norm{\tw}_{L^2(\Omega)} \norm{\nabla \bv_A}_{L^\infty(\Omega)} \norm{\oeps \pt \tw}_{L^2(\Omega)},\\
          \abs{\int_\Omega (\ol{\mathbf{r}}_1+\mathbf{g}) \cdot P_\sigma(\oeps^2 \pt \tw) \dx}&\leq C (\norm{\oeps (\ol{\mathbf{r}}_1+\mathbf{g})}_{L^2(\Omega)}+ \norm{\ol{\mathbf{r}}_1+\mathbf{g}}_{H^1(\Omega)'}) \norm{\oeps \pt \tw}_{L^2(\Omega)}.
	\end{align*}
        Since $\nabla \oeps$ vanish in a $\delta$-neighborhorhood $\partial\Omega_{\delta}= \{x\in\Omega: \dist(x,\Omega)<\delta\}$ for sufficiently small $\delta>0$, we have
        \begin{equation*}
          \|\nabla p\|_{H^1(\partial\Omega_\delta)}\leq C_\delta \|\nabla p\|_{L^2(\Omega)}\leq C_\delta'\|\oeps \pt \tw\|_{L^2(\Omega)}.
        \end{equation*}
        Using this together with $\|a\|_{L^2(\partial\Omega)}\leq C\|a\|_{L^2(\partial\Omega_\delta)}^{\frac12}\|a\|_{H^1(\partial\Omega_\delta)}^{\frac12}$, we obtain
        \begin{align*}
          &\left|\int_{\partial\Omega} 2\nu(c_A)D\tw : \no_{\partial\Omega}\otimes \nabla p\sd\sigma\right|\leq C\|D\tw\|_{L^2(\partial\Omega_\delta)}^{\frac12}\|D\tw\|_{H^1(\partial\Omega_\delta)}^{\frac12}\|\nabla p\|_{H^1(\partial\Omega_\delta)} \\
          &\qquad \leq C\left(\|\oeps D\tw\|_{L^2(\Omega)}^{\frac12}\|\oeps \nabla^2\tw\|_{L^2(\Omega)}^{\frac12}+ \|\oeps D\tw\|_{L^2(\Omega)}\right)\norm{\oeps \pt \tw}_{L^2(\Omega)}.
        \end{align*}
	Furthermore, we have
	\begin{align}\nonumber
		-\eps \Div (\nabla \ol{u} \otimes \nabla c_A+\nabla c_A \otimes \nabla \ol{u})
		& = \Big(-\eps \Delta \ol{u} + \tfrac{f''(c_A)}{\eps} \ol{u}\Big) \nabla c_A
		+ \Big(-\eps \Delta c_A + \tfrac{f'(c_A)}{\eps}\Big)\nabla \ol{u} \\ \label{eq:CapId}
		& \quad - \nabla \Big(\tfrac{f'(c_A)}\eps \ol{u} + \eps\nabla \ol{u} \cdot \nabla c_A\Big),
	\end{align}
    where $-\eps \Delta c_A + \tfrac{f'(c_A)}{\eps}$ is bounded in $L^\infty(Q_{T})$ due to \eqref{eq:OptProfile}, cf.\ \eqref{eq:DeltacA} below.
	Then 
	\begin{align*}
		&\abs{\int_\Omega \eps \Div (\nabla \ol{u} \otimes_s \nabla c_A) \cdot P_\sigma(\oeps^2 \pt \tw) \dx}\\
		& \quad \leq C \left(\norm{(\oeps \nabla \ol{u},\oeps \cL_\eps \ol{u},\oeps \nabla_\btau\nabla \ol{u})}_{L^2(\Gamma_t(2\delta))} \norm{\eps \nabla c_A}_{L^\infty(\Omega)}  + \norm{\oeps \nabla \ol{u}}_{L^2(\Omega)}\right)
		\norm{\oeps \pt \tw}_{L^2(\Omega)},
	\end{align*}
    where $\cL_\eps$ is defined in \eqref{eq:Leps2}.

	Hence integrating in \eqref{energy-0000} over $ (0,T) $ together with Gronwalls', Cauchy--Schwarz's, and Young's inequality yields
	\begin{align}\nonumber
		&\norm{\oeps D \tw}_{L^\infty(0,T;L^2)}
		+ \norm{\oeps \pt \tw}_{L^2(Q_{T})} -\tilde\delta\|\oeps \nabla^2\tw\|_{L^2(\Omega)}^2 \\
		&\quad \leq  
		C_{\tilde\delta} \Big(\|(\oeps \nabla \ol u,\ol u,\nabla \tw)\|_{L^2(Q_{T})}
		+ \norm{\oeps \cL_\eps \ol{u}}_{L^2(0,T;L^2(\Gamma_t(2\delta)))} 
		+ \|(\oeps \ol{\mathbf{r}}_1,\oeps \mathbf{g}) \|_{L^2(Q_{T})}\Big), \label{eq:oeps-Dw}
	\end{align}
	where $\tilde\delta>0$ is arbitrary and $\|\nabla \tw\|_{L^2(Q_{T})}$ is bounded by the energy estimate in Section \ref{linearized system} and $ \norm{\oeps \nabla \ol{u}}_{L^2(Q_{T})} $, $ \norm{\oeps \cL_\eps \ol{u}}_{L^2(Q_{T})} $ are controlled by the refined estimates above.
	Now by \eqref{eq:NS-tw-2} and \eqref{eq:CapId} with $\ol q$ replaced by $\ol q + \tfrac{f'(c_A)}\eps \ol{u} + \eps\nabla \ol{u} \cdot \nabla c_A$ and invoking Lemma \ref{lem:weighted-stokes-estimate}, we find
	\begin{align*}
		\|\oeps \nabla^2\tw\|_{L^2(\Omega)}
		\leq
		C\left(
		\|\oeps \mathbf f\|_{L^2(\Omega)} 
		+
		\|\nabla\tw\|_{L^2(\Omega)}
		\right),
	\end{align*}
	where
        \begin{align*}
            \mathbf f &= - \pt \tw - \bv_A \cdot \nabla \tw - \tw \cdot \nabla \bv_A 
            + \Div(2\nu'(c_A)\ol uD\ve_A)
            \\
            &\quad+\left(-\eps \Delta \ol{u} + \tfrac{f''(c_A)}{\eps} \ol{u}\right) \nabla c_A 
            + \left(-\eps \Delta c_A + \tfrac{f'(c_A)}{\eps}\right) \nabla \ol{u} + \ol{\br}_1 + \mathbf g,
        \end{align*}
        which by the arguments above and \eqref{eq:oeps-Dw} satisfies
	\begin{align*}
		\norm{\oeps \mathbf f}_{L^2(\Omega)}
		& \leq C \Big(\norm{(\oeps \nabla_\btau \nabla \ol u, \oeps \cL_\eps \ol{u})}_{L^2(Q_{T})}
		+ \|\nabla \tw\|_{L^2(Q_{T})}
		+ \|\oeps \ol{\mathbf{r}}_1 \|_{L^2(Q_{T})}
		+ \|\oeps \mathbf{g} \|_{L^2(Q_{T})}\Big).
	\end{align*} 
	Finally, choosing $\tilde\delta>0$ sufficiently small, one arrives at
	\begin{align*}
		& \norm{\oeps  D \tw}_{L^\infty(0,T;L^2)}
		+ \norm{\oeps \pt \tw}_{L^2(Q_{T})}
		+ \norm{\oeps  \nabla^2 \tw}_{L^2(Q_{T})}
		+ \norm{\oeps \nabla \ol{q}}_{L^2(Q_{T})} \\
		& \quad \leq C \Big(\|(\oeps \nabla \ol u,\ol u)\|_{L^2(Q_{T})}
		+ \norm{\oeps \tw}_{L^2(Q_{T})}
		+ \|\nabla \tw\|_{L^2(Q_{T})}
		+ \|\oeps \ol{\mathbf{r}}_1 \|_{L^2(Q_{T})}
		+ \|\oeps \mathbf{g} \|_{L^2(Q_{T})}\Big)
	\end{align*}
        which together with the estimate \eqref{eq:w1Estim} for $\we_1$ finishes the proof.
\end{proof}

\subsection{Higher Order Estimates for Linearized Convective Allen--Cahn Equation}
With the aid of the second-order estimates we derived above, now we record the third-order spatial estimates for the convective Allen--Cahn equation, which is essential for closing the arguments.
\begin{thm}
  \label{thm:HighOrderEstimate}
  Let $ \ol{u}\in L^2(0,T;H^3(\Omega)\cap H^1_0(\Omega))\cap H^1(0,T;H^1_0(\Omega)) $ fulfill the linearized convective Allen--Cahn equation \eqref{eq:linNSAC3'} for some $\ol{r}_3\in L^2(0,T;H^1_0(\Omega))$, $\ol\we\in L^2(0,T;H^2(\Omega)^d\cap H^1_0(\Omega))$. Then there are $C>0, \eps_1\in (0,1]$ such that for all $ \eps \in(0,\eps_1] $ and $T\in (0,T_0]$, it holds
  \begin{align}
    &\left\|\ol u\right\|_{L^\infty(0,T;\wW^2)}
    + \left\|\oeps\partial_t \ol u\right\|_{L^2 (0,T;\wW^1)}\nonumber\\\nonumber
    &\quad
    +\left\|\nabla_\btau \ol u\right\|_{L^2 (0,T;\wW^2(\Gamma_t(2\delta)))} + \left\|\oeps\mathcal{L}_\eps \ol u\right\|_{L^2 (0,T;\wW^1(2\delta))}
    + \|(\oeps^2 \nabla^2\ol u, \oeps^3 \nabla^3\ol u )\|_{L^2 (Q_T)}\nonumber\\\nonumber
    &\leq C\left(\|\oeps\ol{r}_3\|_{L^2(0,T;\wW^1)}
    + \|(\ol u,\oeps \nabla^2 \ol{\bw})\|_{L^2(Q_{T})} + \|\ol u\|_{L^2(0,T;\Veps)}
    + \|\oeps \ol\we \|_{L^2(0,T;H^1)}+ \left\|\ol u_0\right\|_{\wW^2}\right).
  \end{align}
\end{thm}
\begin{proof}
\noindent
\emph{Step 1: Suboptimal weighted third order estimates for $ \ol{u} $.} First of all, we note that $\ol u|_{\partial\Omega}= \partial_t \ol u|_{\partial\Omega}=0$ by the homogeneous Dirichlet boundary conditions for $\ol u$. Therefore, using \eqref{eq:linNSAC3'}, $\ve_A|_{\partial\Omega}=\ol\we|_{\partial\Omega}=0$, and $\ol r_3|_{\partial\Omega}=0$ by assumption, we also have $\Delta \ol u|_{\partial\Omega}=0$. Now
testing \eqref{eq:ConvAC} with $\Delta(\oeps^6 \Delta \ol{u}) - \Div(\oeps^6 \tfrac{\alpha}{\eps^2} \nabla \ol{u})$ in $\Omega$ and using $\ol u|_{\partial\Omega}=\Delta \ol u|_{\partial\Omega}=\ol r_3|_{\partial\Omega}=0$  yields
\begin{align}\nonumber
	&\frac12\ddt \left(
	\int_\Omega \abs{\oeps^3\Delta \ol{u}}^2 \sd x
	+  \int_\Omega \oeps^6\tfrac{\alpha}{\eps^2}\abs{\nabla \ol{u}}^2 \sd x 
	\right)
	+ \int_\Omega \ve_A \cdot \nabla \ol{u}\,  \left(\Delta(\oeps^6 \Delta \ol{u}) - \Div(\oeps^6 \tfrac{\alpha}{\eps^2} \nabla \ol{u})\right)\sd x\\\nonumber
	&\qquad + \int_\Omega (\ol{\we}-\ol\we|_{\Gamma})\cdot \nabla c_A\, \left(\Delta(\oeps^6 \Delta \ol{u}) - \Div(\oeps^6 \tfrac{\alpha}{\eps^2} \nabla \ol{u})\right)\sd x
          - \frac12\int_\Omega (\partial_t \oeps^6)\left((\Delta \ol u)^2+ \tfrac{\alpha}{\eps^2}|\nabla \ol u|^2\right)\sd x\\\nonumber
	&\quad = -  \|\oeps^3 \nabla\Delta \ol{u}\|_{L^2(\Omega)}^2- \int_\Omega \Delta \ol u \nabla \Delta \ol u \cdot \nabla (\oeps^6)\sd x  
	- \alpha^2 \|\tfrac{\oeps^3}{\eps^2} \nabla \ol{u}\|_{L^2(\Omega)}^2
	- 2  \alpha \|\tfrac{\oeps^3}{\eps} \Delta \ol{u}\|_{L^2(\Omega)}^2 
	\\\nonumber
	&\qquad 
          -  \int_{\Gamma_t(2\delta)} \tfrac1{\eps^2}(V(\tfrac{d_{\Gamma}}{\eps}-h_\eps)+a_\eps) \ol{u} \left(\Delta(\oeps^6 \Delta \ol{u}) - \Div(\oeps^6 \tfrac{\alpha}{\eps^2} \nabla \ol{u})\right)\sd x \\
          &\qquad 
	+ \int_\Omega \nabla \ol{r}_3 \cdot (-\nabla ( \oeps^6  \Delta \ol{u})+\tfrac{\oeps^6}{\eps^2} \alpha \nabla \ol u)\, dx-\int_\Omega \Delta \ol u (\nabla \oeps^6)\cdot\tfrac{\alpha}{\eps^2} \nabla \ol u\sd x. 
\end{align}
     Here it holds 
\begin{align*}
   & \abs{\int_\Omega \ve_A \cdot \nabla \ol{u}  \Delta (\oeps^6 \Delta\ol{u})\sd x}
	 = \abs{\int_\Omega \nabla(\bv_A \cdot \nabla \ol{u}) \cdot \nabla(\oeps^6  \Delta \ol{u}) \dx} \\
	&\quad  \leq C \Big(\norm{\nabla \bv_A}_{L^\infty(\Omega)} \norm{\oeps^3 \nabla \ol{u}}_{L^2(\Omega)} + \norm{\bv_A}_{L^\infty(\Omega)} \norm{\oeps^3 \Delta \ol{u}}_{L^2(\Omega)}\Big)
	\norm{(\oeps^3 \nabla \Delta \ol{u},\oeps^2\Delta \ol u)}_{L^2(\Omega)},
\end{align*}
\begin{align*}
    & \abs{\int_\Omega (\ol{\we}-\ol\we|_{\Gamma}) \cdot \nabla c_A \Delta(\oeps^6 \Delta \ol{u})\sd x}
	= \abs{\int_\Omega \nabla \big((\ol{\we}-\ol\we|_{\Gamma}) \cdot \nabla c_A\big)  \nabla(\oeps^6 \Delta \ol{u})\sd x} \\
	& \qquad \qquad \leq \Big(\norm{\oeps^2 \nabla^2 \ol{\bw}}_{L^2(\Omega)} \norm{\oeps \nabla c_A}_{L^\infty(\Omega)} + \norm{\ol{\bw}}_{H^1(\Omega)} \norm{\oeps^3 \Delta c_A}_{L^\infty(\Omega)}\Big)
	\norm{(\oeps^3 \nabla \Delta \ol{u},\oeps^2\Delta \ol u)}_{L^2(\Omega)}, 
\end{align*}
and analogously
\begin{align*}
   & \abs{\int_\Omega \ve_A \cdot \nabla \ol{u}  \Div (\oeps^6 \tfrac{\alpha}{\eps^2} \nabla\ol{u})\sd x}
	 = \abs{\int_\Omega \nabla(\bv_A \cdot \nabla \ol{u}) \cdot \oeps^6 \tfrac{\alpha}{\eps^2} \nabla \ol{u} \dx} \\
	&\quad  \leq C \Big(\norm{\nabla \bv_A}_{L^\infty(\Omega)} \norm{\oeps^3 \nabla \ol{u}}_{L^2(\Omega)} + \norm{\bv_A}_{L^\infty(\Omega)} \norm{\oeps^3 \Delta \ol{u}}_{L^2(\Omega)}\Big)
	\norm{ \tfrac{\oeps^3}{\eps^2}\nabla \ol{u}}_{L^2(\Omega)},
\end{align*}
\begin{align*}
    & \abs{\int_\Omega (\ol{\we}-\ol\we|_{\Gamma}) \cdot \nabla c_A \Div(\oeps^6 \tfrac{\alpha}{\eps^2}\nabla \ol{u})\sd x}
	= \abs{\int_\Omega \nabla \big((\ol{\we}-\ol\we|_{\Gamma}) \cdot \nabla c_A\big) \cdot \oeps^6\tfrac{\alpha}{\eps^2} \nabla \ol{u})\sd x} \\
	& \qquad \qquad \leq \Big(\norm{\oeps^2 \nabla^2 \ol{\bw}}_{L^2(\Omega)} \norm{\oeps \nabla c_A}_{L^\infty(\Omega)} + \norm{\ol{\bw}}_{H^1(\Omega)} \norm{\oeps^3 \Delta c_A}_{L^\infty(\Omega)}\Big)
	\norm{\tfrac{\oeps^3}{\eps^2}\nabla \ol u}_{L^2(\Omega)}. 
\end{align*}
Moreover, we know
\begin{align*}
    \abs{\int_{\Gamma_t(2\delta)} \tfrac{1}{\eps^2}(V(\rho(x,t))+a_\eps) \ol{u} \Delta (\oeps^6 \Delta \ol{u}) \sd x}
	& =  \abs{\int_{\Gamma_t(2\delta)} \tfrac1{\eps^2}\nabla \big( (V(\rho(x,t)) + a_\eps) \ol{u} \big) \cdot \nabla (\oeps^6  \Delta \ol{u}) \sd x} \\
	& \leq C \norm{(\ol{u},\oeps \nabla \ol{u})}_{L^2(\Omega)}
	\norm{(\oeps^3 \nabla \Delta \ol{u},\oeps^2\Delta \ol{u})}_{L^2(\Omega)}, \\
    \abs{\int_{\Gamma_t(2\delta)} (V(\rho(x,t))+a_\eps) \ol{u} \Div (\oeps^6 \tfrac{\alpha}{\eps^2} \nabla \ol u) \sd x} 
	& \leq C \norm{(\ol{u},\oeps \nabla \ol{u})}_{L^2(\Omega)}
	\norm{\oeps^3\tfrac{\alpha}{\eps^2} \nabla  \ol{u}}_{L^2(\Omega)},
\end{align*}
since $\tfrac{\oeps^2}{\eps^2}V(\rho)$ is uniformly bounded and $V(\rho)+a_\eps|_{\partial\Gamma_t(2\delta)}= f''(c_A)-f''(\pm 1)|_{\partial\Gamma_t(2\delta)}=0$.
Here by Lemma \ref{lem:f''-decomposition}, we have used $ a_\eps \in O(\eps^2) $ in $ L^\infty(Q_{T}) $ and $ \nabla a_\eps = O(\eps^{1+1/p}) $ in $ L^p(Q_{T}) $ for $ 1 \leq p \leq \infty $.
In addition,
\begin{align*}
  &\abs{\int_\Omega \nabla \ol{r}_3 \cdot (-\nabla ( \oeps^6  \Delta \ol{u})+\oeps^4 \alpha \nabla \ol u)\, \sd x}\\
  &\quad\leq C\norm{\oeps^3 \nabla \ol{r}_3}_{L^2(\Omega)} \Big(\norm{(\oeps^3 \nabla \Delta \ol{u},\oeps^2\Delta \ol u)}_{L^2(\Omega)} + \norm{\oeps^3\tfrac{\alpha}{\eps^2} \nabla \ol{u}}_{L^2(\Omega)}\Big)
        \end{align*}
        and
        \begin{align*}
          \left|\frac12\int_\Omega (\partial_t \oeps^6)\left((\Delta \ol u)^2+ \tfrac{\alpha}{\eps^2}|\nabla \ol u|^2\right)\sd x\right|&\leq C\|(\oeps^2\Delta \ol u, \tfrac{\oeps^2}{\eps}\nabla \ol u)\|_{L^2(\Omega)}^2,\\
          \left|\int_\Omega \Delta \ol u \nabla \Delta \ol u \cdot \nabla (\oeps^6)\sd x\right|&\leq C\|\oeps^2\Delta \ol u\|_{L^2(\Omega)}\|\oeps^3\nabla \Delta\ol u\|_{L^2(\Omega)},\\
          \left|\int_\Omega \Delta \ol u (\nabla \oeps^6)\cdot\tfrac{\alpha}{\eps^2} \nabla \ol u\sd x \right|&\leq C \|\tfrac{\oeps^3}\eps\Delta \ol u\|_{L^2(\Omega)}\|\tfrac{\oeps^2}\eps \nabla\ol u\|_{L^2(\Omega)}.
        \end{align*}
Then Young's and Gronwall's inequality implies
\begin{align}
	\label{eqs:eps3-Delta-u}
	&\|(\oeps^3 \Delta \ol u,\tfrac{\oeps^3}{\eps} \nabla \ol u)\|_{L^\infty(0,T;L^2)}
	+ \|(\oeps^3 \nabla \Delta \ol u,\tfrac{\oeps^3}{\eps} \Delta \ol u,\tfrac{\oeps^3}{\eps^2} \nabla \ol u)\|_{L^2(Q_{T})} \\
	\nonumber
	&\quad \leq \|(\oeps^3 \Delta \ol u_0,\tfrac{\oeps^3}{\eps} \nabla \ol u_0)\|_{L^2(\Omega)}
    + C  \left(\|(\ol u, \oeps^2 \Delta \ol u, \tfrac{\oeps^2}\eps \nabla \ol u)\|_{L^2(Q_{T})}
	+ \|(\oeps^3 \nabla \ol r_3,\oeps^2 \nabla^2 \ol{\bw})\|_{L^2(Q_{T})}\right).
    \end{align}
Furthermore, it follows from the equation \eqref{eq:ConvAC} that $\oeps^3 \pt \nabla \ol{u} \in L^2(Q_{T})$ can be bounded in the same way.

\smallskip

\noindent
\emph{Step 2: Refined third order estimates for $\ol u$.}
	Finally, we test  \eqref{eq:TransformConvAC} with $\toeps^4\partial_t(-\widetilde{\Delta}_{\Sigma} \tilde u +\widetilde{\mathcal L}_\eps \tilde u)$ and obtain
	\begin{align*}
		&\frac12\frac{\mathrm{d}}{\mathrm{d} t} \int_{\Sigma_{2\delta}} (-\widetilde{\Delta}_{\Sigma} \tilde u +\widetilde{\mathcal L}_\eps \tilde u)^2\toeps^4\,\d(r,s)+ B_\Sigma^2(\partial_t \tilde u)+ B_\eps^2 (\partial_t \tilde u)\\
		&= \sum_{i,j= 1}^d\int_{\Sigma_{2\delta}} (\partial_ta_{i,j})(\nabla_\Sigma \tilde u)_i(\nabla_\Sigma \tilde g - \nabla_\Sigma \partial_t \tilde u)_j \toeps^4\,\d(r,s) 
        + B^2_\Sigma(\tilde{g}, \partial_t \tilde{u})+ B^2_\eps(\tilde{g}, \partial_t \tilde{u}) 
          \\
        & \quad 
        + 4\int_{\Sigma_{2\delta}} (\tilde{g} - \pt \tilde u) \partial_r\partial_t\tilde u\,\toeps^2 r \,\d(r,s)
        - \int_{\Sigma_{2\delta}} f'''(\theta_0) \theta_0'(\tfrac{r}\eps-h_\eps)\tfrac{\toeps^2}{\eps^2} \partial_t h_\eps  \tilde u\, \oeps^2 (\tilde g - \pt \tilde u) \,\d(r,s).
	\end{align*}
	Here
    \begin{align*}
		& \abs{\sum_{i,j= 1}^d \int_{\Sigma_{2\delta}} (\partial_ta_{i,j})(\nabla_\Sigma \tilde u)_i(\nabla_\Sigma \tilde g - \nabla_\Sigma \partial_t \tilde u)_j \toeps^4\,\d(r,s)} \\
        & \qquad \leq
		C \|\toeps \nabla_\Sigma \tilde u\|_{L^2(\Sigma_{2\delta})}\|(\toeps^2 \nabla_\Sigma \tilde g, \toeps^2 \nabla_\Sigma \partial_t\tilde u)\|_{L^2(\Sigma_{2\delta})},
    \end{align*}
    and
	\begin{align*}
        \abs{\int_{\Sigma_{2\delta}} (\tilde g - \pt \tilde u) \partial_r\partial_t\tilde u\,\toeps^2 r \,\d(r,s)}
		&\leq C\|(\tfrac{\toeps^2}{\eps} \tilde g,\toeps \pt \tilde{u})\|_{L^2(\Sigma_{2\delta})}\|\toeps r\partial_r\partial_t \tilde u\|_{L^2(\Sigma_{2\delta})}, \\
        \abs{\int_{\Sigma_{2\delta}} f'''(\theta_0) \theta_0'(\tfrac{r_\eps}\eps)\tfrac{\oeps^2}{\eps^2} \partial_t h_\eps  \tilde u\, \oeps^2 (\tilde g - \pt \toeps^2) \,\d(r,s)}
        & \leq C \|\tfrac{\toeps^2}{\eps^2}\tilde u\|_{L^2(\Sigma_{2\delta})} \|(\tfrac{\toeps^2}{\eps} \tilde g,\toeps \pt \tilde{u})\|_{L^2(\Sigma_{2\delta})},
	\end{align*}
	where
	\begin{equation*}
		\|\toeps r\partial_r\partial_t \tilde u\|_{L^2(\Sigma_{2\delta})}^2\leq \|\toeps^2\partial_r\partial_t \tilde u\|_{L^2(\Sigma_{2\delta})}^2 \leq B^2_\eps (\partial_t \tilde u) + C\|\toeps \partial_t \tilde{u}\|_{L^2(\Sigma_{2\delta})}^2.  
	\end{equation*}
	Hence using Young's inequality we obtain
	\begin{align*}
		&\frac12\frac{\mathrm{d}}{\mathrm{d} t} \int_{\Sigma_{2\delta}} (-\widetilde{\Delta}_{\Sigma} \tilde u +\widetilde{\mathcal L}_\eps \tilde u)^2\toeps^4\,\d(r,s)+ \|\toeps^2\nabla_\Sigma\partial_t \tilde u\|_{L^2(\Sigma_{2\delta})}^2+ B_\eps^2 (\partial_t \tilde u)\\
		&\quad \leq  C\left(\|\toeps \nabla_\Sigma \tilde u\|_{L^2(\Sigma_{2\delta})}^2+ \|\tilde u\|_{L^2(\Sigma_{2\delta})}^2+ \|\toeps \partial_t \tilde u\|_{L^2(\Sigma_{2\delta})}^2  
                 + \left\|(\toeps^2\nabla_\Sigma \tilde{g},\toeps^2 \partial_r \tilde{g}, \tfrac{\toeps^2}\eps \tilde{g}\right\|_{L^2(\Sigma_{2\delta})}^2\right).
	\end{align*}
	This yields 
	\begin{align}\nonumber
		&\|\toeps^2 (-\widetilde{\Delta}_{\Sigma} \tilde u +\widetilde{\mathcal L}_\eps \tilde u)\|_{L^\infty(0,T;L^2(\Sigma_{2\delta}))}+ \left\|\left(\toeps^2\nabla\partial_t\tilde{u}, \tfrac{\toeps^2}\eps \partial_t \tilde u\right)\right\|_{L^2(\Sigma_{2\delta}\times (0,T))}\\\nonumber
		&\quad \leq C\Big(\|\toeps^2(\nabla_\Sigma\tilde{g}, \partial_r\tilde{g}, \tfrac1\eps\tilde{g})\|_{L^2(\Sigma_{2\delta}\times (0,T))}+\|(\tilde u,\toeps\nabla_\Sigma \tilde{u},\toeps \partial_t \tilde u)\|_{L^2(\Sigma_{2\delta}\times (0,T))}\\\label{eq:HigherOrder1}
          &\qquad\quad   + \|\toeps^2(-\widetilde{\Delta}_{\Sigma}+ \widetilde{\mathcal L}_\eps) \tilde{u}_0\|_{L^2(\Sigma_{2\delta})}\Big).
	\end{align}

	Next we apply $\nabla_\Sigma$ to \eqref{eq:TransformConvAC} and obtain
	\begin{equation*}
		-\widetilde{\Delta}_{\Sigma} (\nabla_\Sigma \tilde u)   +\widetilde{\mathcal{L}}_\eps(\nabla_\Sigma \tilde u) = \nabla_\Sigma \left(-\partial_t\tilde u +\tilde g \right)+ [\nabla_\Sigma,\widetilde{\Delta}_{\Sigma}]\tilde u + \tfrac1{\eps^2} (\nabla_\Sigma f''(\theta_0)) \tilde u,     
	\end{equation*}
	where $\tfrac{\toeps^2}{\eps^2} (\nabla_\Sigma f''(\theta_0))$ is bounded and all terms on the right-hand side multiplied by $\toeps^2$ are bounded in $L^2(\Sigma_{2\delta}\times (0,T))$ by the previous estimates. Hence \eqref{eq:EllipticEstim2} yields
	\begin{align}\nonumber
		& \|(\toeps^2\nabla^2\nabla_\Sigma\tilde{u}, \tfrac{\toeps^2}{\eps} \nabla \nabla_\Sigma \tilde{u}, \tfrac{\toeps^2}{\eps^2} \nabla_\Sigma \tilde{u})\|_{L^2(\Sigma_{2\delta}\times (0,T))}^2\leq C\left(\|\toeps^2(\nabla_\Sigma\tilde{g}, \partial_r\tilde{g}, \tfrac1\eps\tilde{g})\|_{L^2(\Sigma_{2\delta}\times (0,T))}^2\right.\\\nonumber
		&\qquad\quad \left.+\|(\tilde u,\toeps\nabla_\Sigma \tilde{u},\toeps \partial_t \tilde u)\|_{L^2(\Sigma_{2\delta}\times (0,T))}^2+ \|\toeps^2(-\widetilde{\Delta}_{\Sigma}+ \mathcal L_\eps) \tilde{u}_0\|_{L^2(\Sigma_{2\delta})}^2\right).
	\end{align}
    Note that by \eqref{eq:EllipticEstim} we additionally derive
	\begin{align*}
		\|\toeps^2\partial_r^2\nabla_\Sigma \tilde u\|_{L^2(0,T;L^2(\Sigma_{2\delta}))}&\leq \|\toeps^2 \mathcal{L}_\eps(\nabla_\Sigma \tilde u)\|_{L^2(0,T;L^2(\Sigma_{2\delta}))}+C\|\tfrac{\toeps^2}{\eps^2} \nabla_\Sigma \tilde u \|_{L^2(0,T;L^2(\Sigma_{2\delta}))}.
	\end{align*}
	Using the equation \eqref{eq:TransformConvAC} we also obtain by comparison a control of
	\begin{equation*}
		\toeps^2 \partial_r \widetilde{\mathcal{L}}_\eps(\tilde u)\in L^2(0,T;L^2(\Sigma_{2\delta})).
	\end{equation*}
	Furthermore, using $- \toeps^3 \partial_r^3 \tilde u = \toeps^3 \partial_r \widetilde{\mathcal{L}}_\eps \tilde u -\tfrac{\toeps^2}{\eps^2} \toeps \partial_r (f''(\theta_0(\tfrac{r}{\eps}-h_\eps)\tilde u)$, we obtain additionally
	\begin{align}\nonumber
		& \|(\toeps^2\partial_r \widetilde{\mathcal{L}}_\eps \tilde u,\toeps^3 \partial_r^3 \tilde{u})\|_{L^2(\Sigma_{2\delta}\times (0,T))}^2 \leq C\left(\|\toeps^2(\nabla_\Sigma\tilde{g}, \partial_r\tilde{g}, \tfrac1\eps\tilde{g})\|_{L^2(\Sigma_{2\delta}\times (0,T))}^2\right.\\\nonumber
          &\quad \qquad \left.+\|(\tilde u,\nabla_\Sigma \tilde{u},\toeps \partial_t \tilde u)\|_{L^2(\Sigma_{2\delta}\times (0,T))}^2+ \|\toeps^2(-\widetilde{\Delta}_{\Sigma}+ \widetilde{\mathcal L}_\eps) \tilde{u}_0\|_{L^2(\Sigma_{2\delta})}^2\right)
	\end{align}
	Finally, combining the $L^\infty$-estimate with respect to $t\in (0,T)$ in \eqref{eq:HigherOrder1} with \eqref{eq:EllipticEstim2} yields
	\begin{align}\nonumber
		& \|(\toeps^2\nabla^2\tilde{u}, \tfrac{\toeps^2}{\eps} \nabla \tilde{u},  \tfrac{\toeps^2}{\eps^2} \tilde{u})\|_{L^\infty(0,T;L^2(\Sigma_{2\delta}))}\leq C\left(\|\toeps^2(\nabla_\Sigma\tilde{g}, \partial_r\tilde{g}, \tfrac1\eps\tilde{g})\|_{L^2(\Sigma_{2\delta}\times (0,T))}\right.\\\nonumber
          &\quad \qquad \left. +\|(\tilde u,\nabla_\Sigma \tilde{u},\toeps \partial_t \tilde u)\|_{L^2(\Sigma_{2\delta}\times (0,T))}
          + \|\toeps^2(-\widetilde{\Delta}_{\Sigma}+ \widetilde{\mathcal{L}}_\eps) \tilde{u}_0\|_{L^2(\Sigma_{2\delta})}\right).
	\end{align}
        Here the norm of $\tilde{g}$ can be estimated by \eqref{eq:gEstim2}.
\end{proof}

\section{Results for the Full Linearized System}\label{sec:FullSystem}

The goal is to show suitable uniform estimates of the solutions of \eqref{eq:linNSAC} by reducing it to the modified system \eqref{eq:linNSAC'}.
The idea is to use the ansatz
\begin{alignat}{2}
	\label{eq:w_A}
	\we&= \we_A + \ol{\we},&\quad \we_A&\coloneqq \we_A^{\mathrm{out}}
	- \zeta_\Gamma \bw_A^{\mathrm{in}},\quad\\
	\label{eq:q_A}
	q&= q_A +\ol q,&\quad   q_A&\coloneqq q_A^{\mathrm{out}} - \zg q_A^{\mathrm{in}}, \\ \label{eq:u_A}
	u &= u_A + \ol{u}, &\quad u_A &\coloneqq  - \zg u_A^{\mathrm{in}},
\end{alignat}
where 
\begin{alignat*}{2}
	\bw_A^{\mathrm{out}} & = (1-\zg)(\we^+\chi_{\Omega^+}+\we^- \chi_{\Omega^-})+ \zg(\we^+\eta(\rho)+ \we^-(1-\eta(\rho))), \\
  q_A^{\mathrm{out}} & = (1-\zg) (q^+\chi_{\Omega^+}+q^- \chi_{\Omega^-}) + \zg(q^+\eta(\rho)+ q^-(1-\eta(\rho))), \\
  \bw_A^{\mathrm{in}} &= \tfrac{h}\eps \prho \hv_A^{\mathrm{in}} , 
                        \qquad                       q_A^{\mathrm{in}}  
	= \tfrac1\eps \prho \hp_0 h
	-\eps \nabla u^{\mathrm{in}}_A\cdot \nabla c^{\mathrm{in}}_A-\tfrac1\eps f'(c_A^{\mathrm{in}}) u_A^{\mathrm{in}}
	, \\
	u_A^{\mathrm{in}} & = \tfrac1\eps(\theta_0'(\rho)+\eps^2\partial_\rho \hc_2) h,
\end{alignat*}
$\hv_0,\hv_1,\hc_2$ are as in Remark~\ref{rem:PropertiesAproxSol}, and $(\we^\pm, h)$ solves the sharp interface limit of the linearized system \eqref{eq:linNSAC}:
\begin{subequations}
	\label{eq:Limit}
	\begin{alignat}{3}
		\label{eq:Limit1}
		\p_t \we^\pm+\ve^\pm\cdot\nabla\we^\pm+\we^\pm\cdot\nabla\ve^\pm&-\nu^\pm\Delta \we^\pm  +\nabla q^\pm = 0 &\quad &\text{in }\Omega^\pm_t, t\in (0,T),\\\label{eq:Limit2}
		\Div \we^\pm &= 0 &&\text{in }\Omega^\pm_t, t\in (0,T),\\\label{eq:Limit3}
		\llbracket 2\nu^\pm D\we^\pm -q^\pm \tn{I}\rrbracket\no_{\Gamma_t} &= 
		-\sigma(\Delta_{\Gamma_t}h\,\no_{\Gamma_t}- H_{\Gamma_t}\nabla_{\Gamma_t} h)
		&& \text{on }\Gamma_t, t\in (0,T),\\ \label{eq:Limit4}
		\llbracket\we^\pm \rrbracket &=0 && \text{on }\Gamma_t, t\in (0,T),\\
		\we^-|_{\partial\Omega}&= 0&&\text{on }\partial\Omega\times(0,T),\\
		\we|_{t=0} &= 0 &&\text{in }\Omega,
	\end{alignat}
	where $\nu^\pm\coloneqq\nu(\pm 1)$, together with
	\begin{alignat}{2}\nonumber
		\partial_t^{\Gamma} h&-(\no_{\Gamma_t} \cdot \we^\pm)|_{\Gamma_t}\circ X_0^{-1}+ \ve^\pm\circ X_0^{-1} \cdot \nabla_{\Gamma_t}h + g_0 h  - \Delta_{\Gamma_t}h \\\label{eq:linTwoPhase5}
		&+(\no_{\Gamma_t}\cdot\partial_\rho \hv_1)|_{\Gamma_t}\circ X_0^{-1}h  = (\no_{\Gamma_t}\cdot \ol{\we})|_{\Gamma_t}\circ X_0^{-1}
		&\quad &\text{on }\Sigma\times (0,T),\\\label{eq:linTwoPhase6}
		h|_{t=0}& = 0&&\text{on } \Sigma.
	\end{alignat}
\end{subequations}
        Here $g_0$ is as in Remark~\ref{rem:PropertiesAproxSol} and $(\we_A,u_A)$ is the leading order of the linearized system \eqref{eq:linNSAC} as $\eps\to 0$ as will be shown in Theorem~\ref{thm:FullLinearizedSystem} below.

\begin{rem}\label{rem:uA}
	We note that $\partial_\rho \hat{\ve}_0= \ue_0d_\Gamma \eta'=0$ on $\Gamma_t$, cf.\ Remark~\ref{rem:PropertiesAproxSol}. Hence \eqref{eq:linTwoPhase5}-\eqref{eq:linTwoPhase6} are equivalent to
	\begin{alignat*}{2}
          \partial_t^{\Gamma} h+ \ve^\pm\circ X_0^{-1} \cdot \nabla_{\Gamma_t}h + g_0 h - \Delta_{\Gamma_t}h 
		&= (\no_{\Gamma_t}\cdot \we)|_{\Gamma_t}\circ X_0^{-1}
		&\quad &\text{on }\Sigma\times (0,T),\\
		h|_{t=0}& = 0&&\text{on } \Sigma,
	\end{alignat*}
	where $\we$ is known. Therefore existence of a solution to \eqref{eq:Limit} follows easily from known results for these linear systems. Moreover, one obtains
	\begin{align}
		& \|h\|_{L^2(0,T;H^{\frac52}(\Sigma))}
		+\|\partial_t h\|_{L^2(0,T;H^{\frac12}(\Sigma))}
		+ \norm{h}_{L^\infty(0,T;H^{\frac32}(\Sigma))} \leq C(T_0)\|\we\|_{L^2(0,T;H^1)}.
		\label{eq:h-w-estimate2}
	\end{align}
	e.g.\ by applying \cite[Theorem~2.9]{StokesAllenCahn}.  Actually, the latter result is stated in the case $\Sigma=\T^1$, $d=2$. But the result and its proof directly carries over to the present situation. Moreover, the extra zero order term $(\no_{\Gamma_t}\cdot\partial_\rho \hv_1)|_{\Gamma_t}\circ X_0^{-1}h $ can be controlled with aid of Gronwall's inequality.
\end{rem}
\begin{prop}
	\label{prop:wpm-h}
	Let $ T \in (0,T_0] $ and 
	 $ \ol{\bw} \in L^2(0,T;H^1(\Omega))$.
	Then there is a unique solution $ (\bw^\pm,q^\pm,h) $ to \eqref{eq:Limit} fulfilling
	\begin{gather*}
		\widetilde{\bw}_A^{\mathrm{out}}\coloneqq\bw^+\chi_{\Omega^+}+\bw^-\chi_{\Omega^-} \in L^2(0,T; H^2(\Omega \setminus \Gamma_t)) \cap H^1(0,T; L^2(\Omega)), \\
		  \widetilde{q}_A^{\mathrm{out}}\coloneqq q^+\chi_{\Omega^+}+q^-\chi_{\Omega^-} \in L^2(0,T; H^1_{(0)}(\Omega \setminus \Gamma_t)), \\
		h \in X_{T,h}\coloneqq L^2(0,T; H^{\frac52}(\Sigma)) \cap H^1(0,T; H^{\frac12}(\Sigma)).
	\end{gather*}
	Moreover, there is a constant $C(T_0)$, independent of $T\in (0,T_0]$, solution and data, such that
	\begin{align}
		\|h\|_{X_{T,h}} & + \|\pt \widetilde{\bw}_A^{\mathrm{out}}\|_{L^2(0,T;L^2)} \nonumber \\
		& + \|\widetilde{\bw}_A^{\mathrm{out}}\|_{L^2(0,T; H^2(\Omega\setminus \Gamma_t))}
		+ \|\widetilde{\bw}_A^{\mathrm{out}}\|_{L^\infty(0,T; H^1(\Omega\setminus\Gamma_t))} 
		\leq C(T_0)\|\ol{\we}\|_{L^2(0,T;H^1)},
		\label{eq:h-w-estimate}
	\end{align}
	where
	\begin{align*}
		\|h\|_{X_{T,h}}\coloneqq \|h\|_{L^2(0,T;H^{\frac52}(\Sigma))}
		+\|\partial_t h\|_{L^2(0,T;H^{\frac12}(\Sigma))}
		+ \norm{h}_{L^\infty(0,T;H^{\frac32}(\Sigma))}.
	\end{align*}
\end{prop}     
\begin{proof}
	The proof essentially follows from \cite[Theorem A.10]{AbelsFei}. In our case, it is even simpler as we do not have extra datum. Moreover, we have by interpolation
	\begin{gather*}
		L^2(0,T; H^2(\Omega \setminus \Gamma_t)) \cap H^1(0,T; L^2(\Omega)) 
		\hookrightarrow L^\infty(0,T; H^1(\Omega \setminus \Gamma_t)), \\
		L^2(0,T; H^{\frac52}(\Sigma)) \cap H^1(0,T; H^{\frac12}(\Sigma))
		\hookrightarrow L^\infty(0,T; H^{\frac32}(\Sigma)).
	\end{gather*}
	This completes the proof.
      \end{proof}

      We remark that $\we_A^{\mathrm{out}}$, which is used in the definition of $\we_A$, is a ``smoothed version'' of $\widetilde{\we}_A^{\mathrm{out}}$. In particular, $\we_A^{\mathrm{out}}\in L^2(0,T;H^2(\Omega)^d)$ while $\widetilde{\we}_A^{\mathrm{out}}\in L^2(0,T;H^2(\Omega\setminus\Gamma_t)^d\cap H^1(\Omega)^d)$ only. The error can be estimated by the following:
      \begin{lem}\label{lem:wout}
        	Let $ T \in (0,T_0] $ and 
                $\we_A^{\mathrm{out}}, \we_A^{\mathrm{in}}, \widetilde{\we}_A^{\mathrm{out}}$ as before. Then there are $C>0$, $\eps_1\in (0,1]$ such that
                \begin{align}
                  \|\we_A^{\mathrm{out}}-\widetilde{\we}_A^{\mathrm{out}}\|_{L^{\infty}(0,T;H^1)}+     \|(\partial_t\we_A^{\mathrm{out}}-\partial_t \widetilde{\we}_A^{\mathrm{out}}, q_A^{\mathrm{out}}-\widetilde{q}_A^{\mathrm{out}}\|_{L^2(Q_T)} &\leq C\sqrt{\eps}\|\ol{\we}\|_{L^2(0,T;H^1)} \label{eq:woutEstim}\\
                                    \|\Div (\we_A^{\mathrm{out}})\|_{L^{\infty}(0,T;H^1)}+ \|\partial_t \Div (\we_A^{\mathrm{out}})\|_{L^2(Q_T)} &\leq C\sqrt{\eps}\|\ol{\we}\|_{L^2(0,T;H^1)} \label{eq:DivwoutEstim}\\
                       \|(\we_A^{\mathrm{in}},\oeps \nabla \we_A^{\mathrm{in}})\|_{L^\infty(0,T;L^2)}+\|(\oeps \partial_t \we_A^{\mathrm{in}},\oeps^2 \nabla^2 \we_A^{\mathrm{in}}) \|_{L^2(Q_T)}&\leq C\sqrt{\eps}\|\ol{\we}\|_{L^2(0,T;H^1)}\label{eq:winEstim} 
                \end{align}
                for all $T\in (0,T_0]$ and $\eps\in (0,\eps_1]$.
      \end{lem}
\begin{proof}
     First of all, we have that
    \begin{equation*}
      \we_A^{\mathrm{out}}-\widetilde{\we}_A^{\mathrm{out}} = -\zg\left(\we^+(\chi_{\Omega^+}- \eta(\rho))+\we^-(\chi_{\Omega^-}-(1-\eta(\rho)))\right)
    \end{equation*}
    and therefore
    \begin{align*}
      &\|\we_A^{\mathrm{out}}-\widetilde{\we}_A^{\mathrm{out}}\|_{L^2(Q_T)}\\
      &\quad \leq C\left(\|\zg(\chi_{\Omega^+}- \eta(\rho))\|_{L^2(\Omega)} + \|\zg(\chi_{\Omega^-}- (1-\eta(\rho)))\|_{L^2(\Omega)}\right)\|\ol{\we}\|_{L^2(0,T;H^1)}\\
      &\quad \leq C\sqrt{\eps}\|\ol{\we}\|_{L^2(0,T;H^1)}
    \end{align*}
    since $\eta(\rho)= 1$ if $\rho\geq 1$ and $\eta(\rho)=0$ if $\rho\leq -1$. In the same way, one obtains
    \begin{equation*}
     \|q_A^{\mathrm{out}}-\widetilde{q}_A^{\mathrm{out}}\|_{L^2(Q_T)}\leq C\sqrt{\eps}\|\ol{\we}\|_{L^2(0,T;H^1)}.
    \end{equation*}
    Moreover,
      \begin{align*}
        \nabla (\we_A^{\mathrm{out}}-\widetilde{\we}_A^{\mathrm{out}}) &= \zg\left(\nabla \we^+(\chi_{\Omega^+}- \eta(\rho))+\nabla \we^-(\chi_{\Omega^-}-(1-\eta(\rho)))+\nabla (\eta(\rho))\otimes (\we^+-\we^-)\right)\\
        &\quad + \nabla \zg\otimes \left(\we^+(\chi_{\Omega^+}- \eta(\rho))+\we^-(\chi_{\Omega^-}-(1-\eta(\rho)))\right)
    \end{align*}
    since $\we^+|_{\Gamma}= \we^-|_{\Gamma}$. Here the first term can be estimated as before and the third term vanishes for sufficiently small $\eps$ since $\eta(\rho)$ is constant for $|\rho|\geq 1$. Using $(r,s)$-coordinates in $\Gamma(2\delta)$ and $\widetilde{\we}^\pm(r,s,t)\in L^2((0,T)\times \Sigma;C^1(-2\delta,2\delta)))$ by Sobolev-embedding as well as $\widetilde{\we}^\pm|_{r=0}=0$, one obtains
    \begin{align*}
      \left\|\zg \nabla (\eta(\rho))\otimes (\we^+-\we^-) \right\|_{L^2(Q_T)}\leq C\sqrt{\eps}\|\ol{\we}\|_{L^2(0,T;H^1)} 
    \end{align*}
    and finally \eqref{eq:woutEstim}. In the same way one shows
      \begin{equation*}
           \|\partial_t \we_A^{\mathrm{out}}-\partial_t\widetilde{\we}_A^{\mathrm{out}}\|_{L^2(Q_T)}\leq C\sqrt{\eps}\|\ol{\we}\|_{L^2(0,T;H^1)} 
         \end{equation*}
         For the divergence we obtain
           \begin{equation*}
        \Div (\we_A^{\mathrm{out}}) = \nabla (\eta(\rho))\cdot (\we^+-\we^-)+ \nabla \zg\cdot \left(\we^+(\chi_{\Omega^+}- \eta(\rho))+\we^-(\chi_{\Omega^-}-(1-\eta(\rho)))\right)
    \end{equation*}
    since $\Div \we^\pm =\widetilde{\we}_A^{\mathrm{out}}=0$ in $Q_T$. Now we use that $\no_\Gamma \cdot \nabla (\we^+-\we^-)|_{\Gamma_t} = -\Div_{\tau} (\we^+-\we^.)|_{\Gamma_t} =0$ since $\Div \we^\pm=0$ and $\we^+|_{\Gamma_t}=\we^-|_{\Gamma_t}$. This implies   $\nabla d_\Gamma\cdot \partial_t\we^+|_{\Gamma_t}=\nabla d_\Gamma\cdot \partial_t\we^-|_{\Gamma_t}$.
    From this one derives \eqref{eq:DivwoutEstim} easily.
         
         Finally, using $\partial_\rho\hat{\ve}_0(\rho,x,t)= d_\Gamma\eta'(\rho) \ue_0(x,t)$ one can show \eqref{eq:winEstim}
         in a straight forward manner.
      \end{proof}

We define 
\begin{align}\nonumber
	-\mathbf{r}_{1,A}&\coloneqq \partial_t\we_A +\ve_A\cdot\nabla\we_A+\we_A\cdot\nabla\ve_A-\Div\left(2\nu(c_A)D\we_A\right)
	- \Div(2\nu'(c_A) u_A D \bv_A)
	\\
	&\quad +\nabla q_A+\eps\Div\left(\nabla u_A\otimes\nabla c_A+\nabla c_A\otimes\nabla u_A\right),\label{eq:r1}\\\label{eq:r2}
    - r_{2,A}&\coloneqq\Div \bw_A, \\\label{eq:r3}
	- r_{3,A}&\coloneqq\partial_t u_A+ \ve_A\cdot \nabla u_A+\we_A\cdot \nabla c_A-\Delta u_A+ \tfrac1{\eps^2}f''(c_A) u_A+ \zg\ol\we|_{\Gamma}\cdot \nabla c_A.
\end{align}
where $(\ve_A, c_A)$ are as in Theorem~\ref{thm:ApproximateSolution}. We note that for the following analysis only the properties in Remark~\ref{rem:PropertiesAproxSol} are important. The form of the other higher order terms does not matter.

Then $(\ol{\we}, \ol{u})$ solves the modified linearized system \eqref{eq:linNSAC'},
where $\ol{\mathbf{r}}_1= \mathbf{r}_1+\mathbf{r}_{1,A}$, $\ol{r}_2= r_2+r_{2,A}$, $\ol{r}_3=r_3+r_{3,A}$, and $\ol\we_0=\we_0$ with perturbation of the data $(\mathbf{r}_{1,A}, r_{2,A}, r_{3,A})$.
Using this we will show the main result of this section:
\begin{theorem}\label{thm:FullLinearizedSystem}
  Let $0 < T\leq T_0$, $(\ve_A, c_A)$ be as in Theorem~\ref{thm:ApproximateSolution}, $\we_0\in H^1(\Omega)^d$, $u_0 \in H^2(\Omega)$, $\mathbf{r}_1\in L^2(Q_T)^d$, $r_2\in L^2(0,T;H^1_{(0)}(\Omega))\cap C([0,T];L^2(\Omega))\cap H^1(0,T;H_{(0)}^{-1}(\Omega))$, $r_3\in L^2(0,T;H^1_{0}(\Omega))$
  with $\Div \we_0 = r_2|_{t=0}$. Moreover, let $\we=\we_A+\ol \we, u= u_A+\ol u$ be the solution of \eqref{eq:linNSAC} with $\ol \we, \we_A,\ol u, u_A$ defined as before. Then there are constants  $\eps_1 \in(0,1]$ and $C(T_0)$ independent of $T\in (0,T_0]$ and $\eps\in (0,\eps_1]$ such that
	\begin{align}
		&\left\|\ol u\right\|_{L^\infty(0,T;\wW^2)}+
		\left\|\nabla_\btau \ol u\right\|_{L^2 (0,T;\wW^2(\Gamma_t(2\delta)))} + \|\oeps \partial_t \ol u\|_{L^2 (0,T;\wW^1)}\nonumber\\
		& \quad + \left\|(\oeps \nabla_{\btau}\nabla \ol u, \oeps\mathcal{L}_\eps \ol u)\right\|_{L^2 (0,T;\wW^1(\Gamma_t(2\delta))}+ \|(\oeps^2 \nabla^2\ol u, \oeps^3 \nabla^3\ol u )\|_{L^2 (Q_T)}\nonumber\\
          & \quad+ \|(\ol\we,\oeps \nabla \ol\we)\|_{L^\infty(0,T;L^2)}+ \|(\nabla \ol\we, \oeps\nabla^2 \ol\we, \oeps\partial_t \ol\we)\|_{L^2(Q_T)}  \nonumber \\
		& \quad + \|(\we_A,\oeps^{\frac12} \nabla \we_A)\|_{L^\infty(0,T;L^2)}+ \|(\oeps^{\frac12}\nabla \we_A, \oeps^{\frac{3}2}\nabla^2 \we_A, \oeps\partial_t \we_A)\|_{L^2(Q_T)}       \label{eq:MainUbarEstim}
		 \leq C(T_0)D_\eps,
	\end{align}
	where
	\begin{align*}
		D_\eps&\coloneqq\left(\left\| u_0\right\|_{\wW^2} +\|\we_0\|_{H^1(\Omega)}+\|\mathbf{r}_1\|_{L^2(0,T;V_\sigma')}+\|\oeps\mathbf{r}_1\|_{L^2(Q_T)} 
        + \|(\ol r_2,\oeps \nabla \ol r_2)\|_{L^2(Q_T)}\right.\\
		&\qquad \left.  + \| r_2\|_{H^1(0,T; H^{-1}_{(0)}(\Omega))}+\|r_3\|_{L^2(0,T;(\Veps)')}+\|\oeps r_3\|_{L^2(0,T;\wW^1)}\right).
	\end{align*}
	Moreover, for every $j\in\N_0$, $1\leq q<\infty$ and any $1 \leq  p \leq 4$ if $d = 3$, $1\leq p <\infty$ if $d=2$, respectively, $u_A$ satisfies
	\begin{align}\nonumber
		& \eps^{\frac12} \|u_A\|_{L^\infty(Q_T)}
		+ \eps^{\frac12-\frac1q-j} \|(\oeps^j u_A, \oeps^j\nabla_\btau u_A, \oeps^{1+j} \nabla u_A)\|_{L^\infty(0,T;L^q(\Omega))}\\\nonumber
		&\qquad
		+ \eps^{\frac12-\frac1p} \|\omega_\eps\partial_t u_A\|_{L^2(0,T;L^p(\Omega))}
		+ \|\omega_\eps\partial_t u_A\|_{L^4(0,T;L^2(\Omega))}\\\nonumber
		&\qquad
		+ \eps^{\frac12-\frac1p} \|(\omega_\eps \nabla_{\btau}\nabla u_A,\omega_\eps^2 \nabla^2 u_A)\|_{L^2(0,T;L^p(\Omega))}
		+ \|\omega_\eps^2 \nabla^2 u_A\|_{L^4(0,T;L^2(\Omega))}\\\label{eq:uAEstim2}
		& \quad \leq C(T_0,p,q,j)\eps^{-\frac12} \|\we \|_{L^2(0,T;H^1(\Omega))}
	\end{align}
	for  some constant $C(T_0,p,q,j)$ independent of $\we$, $T\in (0,T_0]$, and $\eps \in (0,\eps_1]$.
\end{theorem}
The proof is given at the end of this section.

For the following we denote
\begin{alignat*}{2}
	u_A^{\mathrm{in}}(x,t)&\coloneqq \frac1\eps(\theta_0'(\rho)+\eps^2\partial_\rho \hc_2(\rho,x,t)) h &\quad& \text{for all }(x,t)\in \Gamma(2\delta),
\end{alignat*}
and $h_\eps=h_1+\eps h_2+O(\eps^2)$.

We will frequently use the following identities:
\begin{align}
	\nabla c_A^{\mathrm{in}}&=(\theta_0'(\rho)+\eps^2\partial_\rho \hc_2)\left(\tfrac{\nabla d_{\Gamma}}{\varepsilon} - \nabla^\Gamma h_\varepsilon\right)+\eps^2 \nabla_x \hc_2\label{eq:nablacA}
	=\theta_0'(\rho)\tfrac{\nabla d_{\Gamma}}{\varepsilon} - \theta_0'(\rho)\nabla^\Gamma h_1+O(\eps),\\\nonumber
	\Delta c_A^{\mathrm{in}}& = (\theta_0''(\rho) + \varepsilon^2 \prho^2 \hc_2) \left(\tfrac{1}{\varepsilon^2} + \abs{\nabla^\Gamma h_\varepsilon}^2\right) + (\theta_0'(\rho) + \varepsilon^2 \prho \hc_2) \left(\tfrac{\Delta d_\Gamma}{\varepsilon} - \Delta^\Gamma h_\varepsilon\right) \\\nonumber
	& \quad 
	+ 2 \varepsilon^2 \nablax \prho \hc_2 \cdot \left(\tfrac{\nabla d_{\Gamma}}{\varepsilon} - \nabla^\Gamma h_\varepsilon\right) 
	+ \varepsilon^2 \Deltax \prho \hc_2 \\\label{eq:DeltacA}
	& = \tfrac{1}{\varepsilon^2} \theta_0''(\rho)
	+ \tfrac{1}{\varepsilon} \theta_0'(\rho) \Delta d_\Gamma
	+ \left(\prho^2 \hc_2
	+ \theta_0''(\rho) \absm{\nabla^\Gamma h_1}^2
	- \theta_0'(\rho) \Delta^\Gamma h_1\right) 
	+ O(\varepsilon)
\end{align}
in $L^\infty(Q_T)$ as well as
\begin{align}\nonumber
	\pt u_A^{\mathrm{in}} & = \tfrac{h}{\varepsilon} (\theta_0''(\rho) + \varepsilon^2 \prho^2 \hc_2) \left(\tfrac{\pt d_\Gamma}{\varepsilon} - \pt^\Gamma h_\varepsilon\right)+ \varepsilon h \pt \prho \hc_2 	+ \tfrac1{\varepsilon} (\theta_0'(\rho) + \varepsilon^2 \prho \hc_2)\pt^\Gamma h\\ \label{eq:partialtuA}
	&= \tfrac{h}{\varepsilon^2} \theta_0''(\rho) (\pt d_{\Gamma}- \eps \pt^\Gamma h_1-\eps^2\pt^\Gamma h_2)
	+ \tfrac1\eps\theta_0'(\rho) \pt^\Gamma h
	+  h \prho^2 \hc_2 \pt d_{\Gamma}+ O_{1,\eps},\\\nonumber
	\nabla u_A^{\mathrm{in}}&=\tfrac{h}\eps (\theta_0''(\rho) + \varepsilon^2 \prho^2 \hc_2) \left(\tfrac{\nabla d_{\Gamma}}{\varepsilon} - \nabla^\Gamma h_\varepsilon\right)
	+ \varepsilon h \nablax \hc_2 
	+  \tfrac1\eps\left(\theta_0'(\rho) + \varepsilon^2 \prho \hc_2\right)\nabla^\Gamma h, \\\label{eq:nablauA}
	\eps \nabla u_A^{\mathrm{in}} & = \tfrac{h}{\varepsilon} \theta_0''(\rho) \nabla d_{\Gamma}
	- h \theta_0''(\rho) \nabla^\Gamma h_1
	+ \theta_0'(\rho) \nabla^\Gamma h +O_{2,\eps}
\end{align}
where $O_{j,\eps}= O(\eps\|h\|_{X_{T,h}})$ in $L^2(0,T;L^\infty(\Gamma_t(2\delta))$ and $O_{j,\eps}= O(\eps^{\frac32}\|h\|_{X_{T,h}})$ in $L^\infty(0,T;L^2(\Gamma_t(2\delta)))$ for $j=1,2$, and
\begin{align}\nonumber
	\nabla^2 u_A^{\mathrm{in}}& = \tfrac{h}\eps (\theta_0'''(\rho) + \varepsilon^2 \prho^3 \hc_2) \left(\tfrac{\nabla d_\Gamma\otimes \nabla d_\Gamma}{\varepsilon^2} + \nabla^\Gamma h_\varepsilon\otimes\nabla^\Gamma h_\varepsilon \right) \\\nonumber
	& \quad + \tfrac{h}\eps (\theta_0''(\rho) + \varepsilon^2 \prho^2 \hc_2) \left(\tfrac{\nabla^2 d_\Gamma}{\varepsilon} - \nabla^\Gamma(\nabla^\Gamma h_\varepsilon)\right) \\\nonumber
	& \quad + \eps h  \left[\nablax (\prho^2 \hc_2) \otimes \left(\tfrac{\nabla d_\Gamma}{\varepsilon} - \nabla^\Gamma h_\varepsilon\right) +   \left(\tfrac{\nabla d_\Gamma}{\varepsilon} - \nabla^\Gamma h_\varepsilon\right)\otimes \nablax (\prho^2 \hc_2) \right] \\\label{eq:nabla2uA}
	& \quad -  \tfrac1\eps(\theta_0''(\rho) + \varepsilon \prho^2 \hc_2)\left[\nabla^\Gamma h \otimes \nabla^\Gamma h_\varepsilon+ \nabla^\Gamma h_\eps \otimes \nabla^\Gamma h\right] + \nabla^2_x \left[ (\theta_0'(\rho) + \varepsilon^2 \prho \hc_2) h \right]. 
\end{align}
In particular, we have
\begin{align}\nonumber
	\Delta u_A^{\mathrm{in}}
	& = \tfrac{h}{\varepsilon^3} \theta_0'''(\rho)
	+ \tfrac{h}{\varepsilon^2} \theta_0''(\rho) \Delta d_\Gamma
	+ \tfrac{h}\eps \prho \left(
	\prho^2 \hc_2
	+ \theta_0''(\rho) \absm{\nabla^\Gamma h_1}^2
	- \theta_0'(\rho) \Delta^\Gamma h_1
	\right) \\\label{eq:DeltauA}
	& \quad 
	- \tfrac{2}\eps \theta_0''(\rho) (\nabla^\Gamma h \cdot \nabla^\Gamma h_1) 
	+ \tfrac1\eps \theta_0'(\rho) \Delta^\Gamma h
	+ O(\sqrt{\eps}\|h\|_{X_{T,h}}) \qquad \text{in }L^2(\Gamma(2\delta)).      
\end{align}

\begin{prop}\label{prop:uAEstim}
	Let $u_A$ be defined as in \eqref{eq:u_A}. Then for every $j\in\N_0$, $1\leq q<\infty$ and any $1 \leq  p \leq 4$ if $d = 3$, $1\leq p <\infty$ if $d=2$, respectively, there is some $C_{j,p,q}>0$, independent of $T\in (0,T_0]$, $\eps\in (0,\eps_1]$, $u_A$ and $\ol \we$, such that 
	\begin{align}\nonumber
		&\eps^{\frac12} \|u_A\|_{L^\infty(Q_T)}
		+ \eps^{\frac12-\frac1q-j} \|(\oeps^j u_A, \oeps^j\nabla_\btau u_A, \oeps^{1+j} \nabla u_A)\|_{L^\infty(0,T;L^q(\Omega))}+ \|\omega_\eps\partial_t u_A\|_{L^4(0,T;L^2(\Omega))}\\\nonumber
		&\quad
		+ \eps^{\frac12-\frac1p} \|(\omega_\eps\partial_t u_A,\omega_\eps \nabla_{\btau}\nabla u_A,\omega_\eps^2 \nabla^2 u_A)\|_{L^2(0,T;L^p(\Omega))}
		+ \|\omega_\eps^2 \nabla^2 u_A\|_{L^4(0,T;L^2(\Omega))}\\\label{eq:uAEstim1}
		&\leq C_{j,p,q} \eps^{-\frac12} \|h\|_{X_{T,h}}.
	\end{align}
\end{prop}
\begin{proof}
	Using \eqref{eq:partialtuA}-\eqref{eq:nabla2uA}, $\omega_\eps^k \theta_0^{(l)}, \omega_\eps^k \prho^l \hc_2\in \mathcal{R}_{k,\alpha}$, Lemma~\ref{lem:rescale}, Proposition~\ref{prop:wpm-h}, $H^{1/2}(\Sigma) \hookrightarrow L^4(\Sigma)$ if $d=3$ and $H^{1/2}(\Sigma) \hookrightarrow L^q(\Sigma)$ for any $1\leq q<\infty$ if $d=2$, and the product rule, one can obtain the claimed estimate in a straight-forward manner.
\end{proof}

An essential result for the leading error is:
\begin{theorem}\label{th_leading_velocity}
	Let $\mathbf{r}_{1,A}$, $r_{2,A}$ be defined as in \eqref{eq:r1}, \eqref{eq:r2}.
	Then there are $\eps_1\in (0,1]$ and $C>0$ such that
	\begin{align*}
		\|\mathbf{r}_{1,A}\|_{L^2(0,T;V_\sigma')}+ \|\oeps\mathbf{r}_{1,A}\|_{L^2(Q_T)} \leq C\sqrt{\eps}\|\ol{\we}\|_{L^2(0,T;H^1)}, \\
        \norm{r_{2,A}}_{L^2(Q_{T})}
	       + \norm{r_{2,A}}_{H^{1}(0,T; H^{-1}_{(0)}(\Omega))}
	       + \norm{\oeps \nabla r_{2,A}}_{L^2(Q_{T})} \leq C\sqrt{\eps}\|\ol{\we}\|_{L^2(0,T;H^1)}
	\end{align*}
        for all $\eps \in (0,\eps_1]$, $T\in (0,T_0]$.
\end{theorem}
\begin{proof}
	For $\ve\coloneqq\ve^+\chi_{\Omega^+} + \ve^-\chi_{\Omega^-}$ it holds by construction
	\begin{align}\label{eq:expvA}
		\ve_A-\ve&=\zg (\hv_0(\rho,x,t)-\ve^+\chi_{\Omega^+} - \ve^-\chi_{\Omega^-})+O(\eps) =O(\eps^{\frac{1}{2}})
	\end{align}
	in $L^\infty(0,T;L^2(\Omega))$ by the matching conditions.
	
	First, it holds 
	\begin{align*}
		\partial_t\we_A&=\partial_t\we^+ \chi_{\Omega^+} +\partial_t\we^- \chi_{\Omega^-}
		- \zg \partial_t \bw_A^{\mathrm{in}} + R_1+ O(\sqrt{\eps}\|\ol \we\|_{L^2(0,T;H^1)}),
		\\ \nabla\we_A&=\nabla \we^+ \chi_{\Omega^+} +\nabla\we^- \chi_{\Omega^-} 
		- \zg \nabla \bw_A^{\mathrm{in}} + R_2+O(\sqrt{\eps}\|\ol \we\|_{L^2(0,T;H^1)})
	\end{align*}
        in $L^2(Q_T)$ due to Lemma~\ref{lem:wout} and $\llbracket \we^\pm\rrbracket=0$. Here $R_1,R_2$ are lower order terms which contain a factor $\zg'$, vanish in $\Gamma(\delta)$ as $\zg'$ supports in $\Gamma(2\delta) \setminus \Gamma(\delta)$ and decays exponentially to zero as $\eps \to 0$. Moreover, 
	\begin{align*}
	    \ve_A\cdot \nabla\we_A =\ve\cdot \nabla\we_A + (\ve_A-\ve)\cdot\nabla\widetilde{\we}_A^{\mathrm{out}}+ (\ve_A-\ve)\cdot\nabla(\we_A-\widetilde{\we}_A^{\mathrm{out}}),
	\end{align*}
	where $\ve_A-\ve=O(\sqrt{\eps})$ in $L^\infty(0,T;L^2)$ by \eqref{eq:expvA}, $\|\we_A-\widetilde{\we}_A^{\mathrm{out}}\|_{L^\infty(0,T;L^2)}\leq C\sqrt{\eps} \|\ol{\we}\|_{L^2(0,T;H^1)}$, and $\|\widetilde{\we}_A^{\mathrm{out}}\|_{L^2(0,T;L^\infty)}\leq C \|\ol{\we}\|_{L^2(0,T;H^1)}$ due to Proposition~\ref{prop:wpm-h} and Lemma \ref{lem:wout}. Therefore integration by parts yields
\begin{align*}
          &\|(\ve_A-\ve)\cdot\nabla(\we_A-\widetilde{\we}_A^{\mathrm{out}})\|_{L^2(0,T;H^{-1})}+ \|(\ve_A-\ve)\cdot\nabla \widetilde{\we}_A^{\mathrm{out}}\|_{L^2(0,T;H^{-1})} \\
    &\quad      \leq C\left(\|\ve_A-\ve\|_{L^2(0,T;W^{1,\infty})} \|\we_A-\widetilde{\we}_A^{\mathrm{out}}\|_{L^\infty(0,T;L^2)}  + \|\ve_A-\ve\|_{L^\infty(0,T;L^2)} \|\widetilde{\we}_A^{\mathrm{out}}\|_{L^2(0,T;L^\infty)}  \right)\\
  &\quad      \leq C\sqrt{\eps}\|\ol{\we}\|_{L^2(0,T;H^1)}.
	\end{align*}
	Additionally, we write
	\begin{align*}
	    \we_A\cdot\nabla\ve_A = \we_A\cdot\nabla\ve + \we_A\cdot\nabla(\ve_A-\ve),
	\end{align*}
	where 
	\begin{align*}
	    \int_\Omega(\we_A\cdot\nabla(\ve_A-\ve))\cdot\bphi\dx = -\int_\Omega (\ve_A-\ve) \we_A\cdot\nabla\bphi\dx
	+ \int_\Omega \Div \bw_A^{\mathrm{in}} (\bv_A - \bv) \cdot \bphi \dx
	\end{align*}
	for all $\bphi\in V_\sigma$ due to $\llbracket\we^\pm\rrbracket=0$ and $\Div\we^\pm=0$.
	Using \eqref{eq:expvA} and $\|\we_A\|_{L^\infty(0,T;L^2)}\leq C \sqrt{\eps} \|\ol{\we}\|_{L^2(0,T;H^1)}$ by Lemma \ref{lem:wout}, we obtain $\|\we_A\cdot\nabla(\ve_A-\ve)\|_{L^2(0,T;V_\sigma')}\leq C \sqrt{\eps}\|\ol{\we}\|_{L^2(0,T;H^1)}$. The terms of $\bv \cdot \nabla \bw_A + \bw_A \cdot \nabla \bv$ will be used to cancel the ``outer'' part of $\mathbf{r}_{1,A}$ via \eqref{eq:Limit1}.
	It is noted that by \eqref{eq:v-1-eps} we have employed the expansion
		\begin{align*}
			\Div \bw_A^{\mathrm{in}} (\bv_A - \bv)
			& = \frac{1}{\varepsilon} h \ptial{\rho}^2 \hv_0 \cdot \nabla d_\Gamma \bv_{1,\eps}
			+ \prho g \bv_{1,\eps} + O(\varepsilon) \\
			& = \frac{1}{\varepsilon} h \Big(\bu_0 \cdot \nabla d_\Gamma \prho^2(d_\Gamma \eta(\rho))\Big) \cdot \nabla d_\Gamma \bv_{1,\eps}
			+ \prho g \bv_{1,\eps} + O(\varepsilon) \\
			& = h (\rho + h_\varepsilon)\prho \Big(\bu_0 \cdot \nabla d_\Gamma \prho( \eta(\rho)) \Big) \cdot \nabla d_\Gamma \bv_{1,\eps}
			+ \prho g \bv_{1,\eps} + O(\varepsilon)
		\end{align*}
                in $L^2(\Gamma(2\delta)\cap Q_T)$
		for some given function $ g $, by taking $ \rho $-derivative of \eqref{eq:A.44} and using $ d_\Gamma = \varepsilon(\rho + h_\varepsilon) $. Hence one gets $ \Div \bw_A^{\mathrm{in}} (\bv_A - \bv) = O(\sqrt{\varepsilon})$ in $ L^2(Q_T) $.

	Furthermore, we use that 
	\begin{align*}
		 &- \Div(2\nu(c_A) D\we_A)
		+ \nabla q_A  = \Div(2\nu(c_A) D(\zg \bw_A^{\mathrm{in}}))\\ 
		&\qquad - \nabla (\zg q_A^{\mathrm{in}})  - \Div\left(2\nu^\pm D\widetilde{\bw}_A^{\mathrm{out}}\right) + \nabla \widetilde{q}_A^{\mathrm{out}}  - \Div\left(2(\nu(c_A)-\nu^\pm) D\widetilde\bw_A^{\mathrm{out}}\right)+O(\sqrt{\eps}\|\ol\we\|_{L^2(0,T;H^1)}).
	\end{align*}
        in $L^2(0,T;H^{-1}(\Omega)^d)$ because of \eqref{eq:woutEstim}. Here
	\begin{align*}
		& - \int_\Omega \Big(\Div(2\nu^\pm D\widetilde\bw_A^{\mathrm{out}})-\nabla \widetilde{q}_A^{\mathrm{out}}\Big) \cdot \bphi \dx \\
		& \quad = - \int_{\Gamma_t} \jump{2\nu^\pm D\bw^\pm - q^\pm \tn{I}} \bn_{\Gamma_t} \cdot \bphi \dsigma
		+ \int_{\Omega} \sum_\pm(-\nu^\pm\Delta\bw^\pm+\nabla q^\pm)\chi_{\Omega^\pm} \cdot \bphi \dx
	\end{align*}
	for all $\bphi\in V_\sigma$, i.e.,
	\begin{align}
		-\Div(2\nu^\pm D\widetilde\bw_A^{\mathrm{out}})+\nabla \widetilde q_A^{\mathrm{out}} & = -\jump{2\nu^\pm D\bw^\pm - q^\pm \tn{I}} \bn_{\Gamma_t} \delta_\Gamma+\sum_\pm(-\nu^\pm\Delta\bw^\pm+\nabla q^\pm)\chi_{\Omega^\pm}
		\label{eq:jump-outer}
	\end{align}
	in $L^2(0,T;V_\sigma')$, where
	\begin{equation*}
		\left\langle\delta_\Gamma, \psi \right\rangle\coloneqq \int_0^T\int_{\Gamma_t} \psi (x,t)\,\d\sigma(x)\,dt\quad \text{for }\psi \in L^1(\Gamma).
	\end{equation*}
	Here we have for $\bphi\in V_\sigma$
	\[
	\left\langle -\Div(2(\nu(c_A)-\nu^\pm) D\widetilde{\we}^{\mathrm{out}}), \bphi \right\rangle_{V_\sigma', V_\sigma} = \int_\Omega 2(\nu^\pm-\nu(c_A))D\widetilde{\we}^{\mathrm{out}} : \nabla\bphi\dx,
	\]
	where $\|(\nu^\pm-\nu(c_A))D\widetilde{\we}^{\mathrm{out}}\|_{L^2(0,T;L^2)}\leq C\sqrt{\eps}\|\ol{\we}\|_{L^2(0,T;H^1)}$ and therefore it holds
	\begin{equation*}
		\|\Div(2(\nu(c_A)-\nu^\pm) D\widetilde{\we}^{\mathrm{out}})\|_{L^2(0,T;V_\sigma')}\leq C\sqrt{\eps}\|\ol{\we}\|_{L^2(0,T;H^1(\Omega))}. 
	\end{equation*}

	In the following, we employ the definition to expand the terms with respect to $(\bw_A^{\mathrm{in}},q_A^{\mathrm{in}})$ in different orders of $\varepsilon$.
	First, in $\Gamma(2\delta)$, one calculates 
	\begin{align*}
		\pt \bw_A^{\mathrm{in}} & = \tfrac{h}{\varepsilon^2} \ptial{\rho}^2 \hv_0 \pt d_\Gamma 
		- \tfrac{1}{\varepsilon} h \Big(\ptial{\rho}^2 \hv_0 \pt^\Gamma h_1
		- \ptial{\rho} \pt \hv_0
		- \ptial{\rho}^2 \hv_1 \pt d_\Gamma\Big)
		+ \tfrac{1}{\varepsilon} \pt^\Gamma h \ptial{\rho} \hv_0 \\
		& \quad - h \Big(\prho^2 \hv_0 \pt^\Gamma h_2
		+ \prho^2 \hv_1 \pt^\Gamma h_1
		- \ptial{\rho} \pt \hv_1+\ptial{\rho}^2 \hv_2 \pt d_\Gamma\Big)
		+ \pt^\Gamma h \ptial{\rho} \hv_1
                  + O(\varepsilon \normm{h}_{X_{T,h}}), \\
          & = \tfrac{h}{\varepsilon^2} \ptial{\rho}^2 \hv_0 \pt d_\Gamma 
		- \tfrac{1}{\varepsilon} h \Big(\ptial{\rho}^2 \hv_0 \pt^\Gamma h_1
		- \ptial{\rho} \pt \hv_0
		- \ptial{\rho}^2 \hv_1 \pt d_\Gamma\Big)
		+ \tfrac{1}{\varepsilon} \pt^\Gamma h \ptial{\rho} \hv_0+ O(\sqrt\varepsilon \normm{h}_{X_{T,h}}) \\
		\nabla \bw_A^{\mathrm{in}} & = \tfrac{h}{\varepsilon^2} \ptial{\rho}^2 \hv_0 \otimes \nabla d_\Gamma 
		- \tfrac{h}{\eps} \Big(\ptial{\rho}^2 \hv_0 \otimes \nabla^\Gamma h_1
		- \ptial{\rho} \nablax \hv_0
                                             - \ptial{\rho}^2 \hv_1 \otimes \nabla d_\Gamma\Big)\\
          &\quad
		+ \tfrac{1}{\eps} \ptial{\rho} \hv_0 \otimes \nabla^\Gamma h + O(\sqrt\varepsilon \normm{h}_{X_{T,h}})\\
		\Div \bw_A^{\mathrm{in}} & = \tfrac{h}{\varepsilon^2} \ptial{\rho}^2 \hv_0 \cdot \nabla d_\Gamma
		- \tfrac{h}{\eps} \Big(\ptial{\rho}^2 \hv_0 \cdot \nabla^\Gamma h_1
		- \ptial{\rho} \Divx \hv_0
		- \ptial{\rho}^2 \hv_1 \cdot \nabla d_\Gamma\Big)\\
		&\quad + \tfrac{1}{\eps} \ptial{\rho} \hv_0 \cdot \nabla^\Gamma h + O(\sqrt\varepsilon \normm{h}_{X_{T,h}}), \\
		\nabla q_A^{\mathrm{in}} 
		& = \tfrac{h}{\varepsilon^2} \ptial{\rho}^2 \hp_0 \nabla d_\Gamma + \nabla (-\eps \nabla u^{\mathrm{in}}_A\cdot \nabla c^{\mathrm{in}}_A-\tfrac1\eps f'(c_A^{\mathrm{in}}) u_A^{\mathrm{in}})
	\end{align*}
        in $L^2(Q_T)$.
	By direct computations and $2D \hat{\mathbf{z}} = \prho \hat{\mathbf{z}} \otimes_s \left(\frac{\nabla d_\Gamma}\eps -\nabla ^\Gamma h_\eps\right) + 2D_x \hat{\mathbf{z}}$ with $ \hat{\mathbf{z}}(\rho(x,t),x,t) = \mathbf{z}(x,t) $ for some vector $ \mathbf{z}\colon \R\times \Gamma(2 \delta) \to \bbr^d $, it follows 
	\begin{align*}
		& 2\nu(c_A) D\bw_A^{\mathrm{in}}
		+ 2\nu'(c_A) u_A^{\mathrm{in}} D\bv_A = 2 \nu(\theta_0 + \varepsilon^2 \hc_2) D \left(\tfrac{h}{\varepsilon}(\prho \hv_0 + \varepsilon \prho \hv_1)\right) \\
		&\qquad 
		+ 2 \nu'(\theta_0 + \eps^2 \hc_2) \tfrac{h}{\varepsilon}(\prho \theta_0 + \varepsilon \prho^2 \hc_2) D(\hv_0 + \varepsilon \hv_1)+ O(\sqrt\varepsilon \normm{h}_{X_{T,h}}) \\
		&\quad = \tfrac{h}{\varepsilon} \prho \Big(
		\nu(\theta_0 + \eps^2 \hc_2) \prho (\hv_0 + \varepsilon \hv_1)
		\Big) \otimes_s \left(\tfrac{\nabla d_\Gamma}\eps -\nabla ^\Gamma h_\eps\right)\\
		& \qquad 
		+ \tfrac{h}{\varepsilon} \prho \Big(
		2 \nu(\theta_0 + \eps^2 \hc_2) D_x (\hv_0 + \varepsilon \hv_1)
		\Big) 
		+ \tfrac{\nabla^\Gamma h}{\varepsilon} \otimes_s \big(\nu(\theta_0 + \eps^2 \hc_2) (\prho \hv_0 + \varepsilon \prho \hv_1)\big) 
		+ O(\sqrt\varepsilon \normm{h}_{X_{T,h}})
	\end{align*}
        in $L^2(Q_T\cap \Gamma(2\delta))$.
	Then we have
	\begin{align*}
		& \Div(2\nu(c_A) D\bw_A^{\mathrm{in}})
		+ \Div(2\nu'(c_A) u_A^{\mathrm{in}} D\bv_A)= \ptial{\rho}\big(2\nu(c_A) D\bw_A^{\mathrm{in}}
		+ 2\nu'(c_A) u_A D\bv_A\big) \left(\tfrac{\nabla d_\Gamma}\eps -\nabla ^\Gamma h_\eps\right) \\
		&\qquad   
		+  \Divx \big(2\nu(c_A) D\bw_A^{\mathrm{in}}
		+ 2\nu'(c_A) u_A D\bv_A\big) + O(\sqrt\varepsilon \normm{h}_{X_{T,h}})\\
		&\quad = \tfrac{h}{\varepsilon} \prho^2 \Big(
		\nu(\theta_0(\rho)) \prho (\hv_0 + \varepsilon \hv_1)
		\Big) \Big(\tfrac{1}{\varepsilon^2} + \absm{\nabla^\Gamma h_1}^2\Big) \\
		& \qquad 
		+ \tfrac{h}{\varepsilon} \prho^2 \Big(
		\nu(\theta_0(\rho)) \prho (\hv_0 + \varepsilon \hv_1)
		\Big) \cdot \left(\tfrac{\nabla d_\Gamma}\eps -\nabla ^\Gamma h_1\right) \left(\tfrac{\nabla d_\Gamma}\eps -\nabla ^\Gamma h_1\right) \\
		& \qquad 
		+ \Divx \left(\tfrac{h}{\varepsilon} \prho \Big(
		\nu(\theta_0(\rho)) \prho (\hv_0 + \varepsilon \hv_1)
		\Big) \otimes_s \left(\tfrac{\nabla d_\Gamma}\eps -\nabla ^\Gamma h_1\right)\right) \\
		& \qquad 
		+ \tfrac{h}{\varepsilon} \prho^2 \Big(
		2 \nu(\theta_0(\rho)) D_x (\hv_0 + \varepsilon \hv_1) 
		\Big) \left(\tfrac{\nabla d_\Gamma}\eps -\nabla ^\Gamma h_1\right) 
        + \Divx \left(
		\tfrac{h}{\varepsilon} \prho \Big(
		2 \nu(\theta_0(\rho)) D_x (\hv_0 + \varepsilon \hv_1)
		\Big) \right) \\
		& \qquad 
		+ \tfrac{\nabla^\Gamma h}{\varepsilon} \prho \big(\nu(\theta_0(\rho)) (\prho \hv_0 + \varepsilon \prho \hv_1)\big) \cdot \left(\tfrac{\nabla d_\Gamma}\eps -\nabla ^\Gamma h_1\right) \\
		& \qquad 
		+ \prho \big(\nu(\theta_0(\rho)) (\prho \hv_0 + \varepsilon \prho \hv_1)\big) \tfrac{\nabla^\Gamma h}{\varepsilon} \cdot \left(\tfrac{\nabla d_\Gamma}\eps -\nabla ^\Gamma h_1\right)
		+ \tfrac{\Delta^\Gamma h}{\varepsilon} \nu(\theta_0(\rho)) (\prho \hv_0 + \varepsilon \prho \hv_1) \\
		& \qquad 
		+ \nu(\theta_0(\rho)) (\prho \nablax \hv_0 + \varepsilon \prho \nablax \hv_1) \tfrac{\nabla^\Gamma h}{\varepsilon} 
		+ \tfrac{\nabla^\Gamma \nabla^\Gamma h}{\varepsilon} \nu(\theta_0(\rho)) (\prho \hv_0 + \varepsilon \prho \hv_1)
		+ O(\sqrt\eps \normm{h}_{X_{T,h}})
	\end{align*}
        in $L^2(0,T;H^{-1}(\Gamma_t(2\delta))$.
	Let
        $$
        \hat{\pi} \coloneqq -\eps \nabla u^{\mathrm{in}}_A\cdot \nabla c^{\mathrm{in}}_A-\tfrac1\eps f'(c_A^{\mathrm{in}}) u_A^{\mathrm{in}}.
        $$
        Rearranging each term with $ d_\Gamma = \varepsilon (\rho + h_\varepsilon) $ and adding the expansion of $\nabla q_A^{\mathrm{in}}$ yield
	\begin{align*}
		& \Div(2\nu(c_A) D\bw_A^{\mathrm{in}})
		+ \Div(2\nu'(c_A) u_A^{\mathrm{in}} D\bv_A) - \nabla (q_A^{\mathrm{in}} - \hat{\pi}) \\
		&\quad = \frac{h}{\varepsilon^3} \ptial{\rho} \prho \big(\nu(\theta_0(\rho)) (\ptial{\rho} \hv_0 - \bu_0 d_\Gamma \eta'(\rho))\big)
		+ \frac{1}{\varepsilon^2} h \ptial{\rho}^2 \big(\nu(\theta_0(\rho)) (\bu_0 \rho \eta'(\rho))\big) 
		\\
		& \qquad 
		+ \frac{1}{\varepsilon^2} h \ptial{\rho}^2 \big(\nu(\theta_0(\rho)) (\ptial{\rho} \hv_1 - (\bu_1 d_\Gamma - \bu_0 h_1) \eta'(\rho))
		- \hp_0 \nabla d_\Gamma
		+ 2 \nu(\theta_0(\rho)) D_x \hv_0 \nabla d_\Gamma
		\big) \\
		& \qquad 
		+ \frac{h}{\varepsilon^2} \prho \Div_x \Big(
		\nu(\theta_0(\rho)) \prho \hv_0 \otimes_s \nabla d_\Gamma
		\Big)
        - \frac{h}{\varepsilon^2} \prho \Big(
		\prho \big(\nu(\theta_0(\rho)) \Div_x \hv_0 \big) \nabla d_\Gamma
		\Big)
		\\
		& \qquad + \frac{h}{\varepsilon^2} \bigg[\prho^2 \Big(\nu(\theta_0(\rho)) (\prho \hv_0 - \bu_0 d_\Gamma \eta'(\rho)) \cdot \nabla d_\Gamma \nabla d_\Gamma
		\Big) 
		\bigg] 
		+ O(\sqrt{\eps}\|\ol{\we}\|_{L^2(0,T;H^1)})
	\end{align*}
	in $L^2(0,T;H^{-1}(\Gamma_t(2\delta))^d)$, where the $O(\eps^{-3})$-term vanishes due to \eqref{eq:A.42}, the sixth term is identically zero by \eqref{eq:A.44}, and the other $O(\eps^{-1})$-terms -- together with corresponding terms coming from $\pt \bw_A^{\mathrm{in}}
		+ \bv_A^{\mathrm{in}} \cdot \nabla \bw_A^{\mathrm{in}}
		+ \bw_A^{\mathrm{in}} \cdot \nabla \bv_A^{\mathrm{in}}$ -- contain only $\rho$-derivatives which have exponential decay as $\varepsilon \to 0$ and a primitive, which has exponential decay by the matching condition and \eqref{eq:Limit1}, and hence are of order $O(\sqrt{\eps}\|\ol{\we}\|_{L^2(0,T;H^1)})$ in $L^2(0,T;H^{-1}(\Gamma_t(2\delta))^d)$ by Lemma \ref{lem:f'-phi-x-dependent}. Here we also employed, by \eqref{eq:A.44} and \cite[(A.47)]{AbelsFei} with $k=1$, that
    \begin{gather*}
        \tfrac{h}{\eps^3} \prho^2 \Big( \nu(\theta_0(\rho)) \prho \hv_0 \cdot \nabla d_\Gamma \nabla d_\Gamma \Big)
        = \tfrac{h}{\eps^2} \prho^2 \Big( \nu(\theta_0(\rho)) (\rho + h_\eps) \eta'(\rho) \bu_0 \cdot \nabla d_\Gamma \nabla d_\Gamma \Big), \\
        \tfrac{h}{\eps^2} \prho \Big( \nu(\theta_0(\rho))  \big(\prho \hv_1 \cdot \nabla d_\Gamma - \prho \hv_0 \cdot \nabla^\Gamma h_\eps) \Big)
        = - \tfrac{h}{\eps^2} \prho \big(\nu(\theta_0(\rho)) \Div_x \hv_0) \big).
    \end{gather*}
    Therefore we find
	\begin{align*}
		& \pt \bw_A^{\mathrm{in}}
		+ \bv_A^{\mathrm{in}} \cdot \nabla \bw_A^{\mathrm{in}}
		+ \bw_A^{\mathrm{in}} \cdot \nabla \bv_A^{\mathrm{in}} \\
		& \qquad - \Div(2\nu(c_A) D\bw_A^{\mathrm{in}})
		- \Div(2\nu'(c_A) u_A^{\mathrm{in}} D\bv_A) + \nabla (q_A^{\mathrm{in}} - \hat{\pi}) \\
		& \quad = - \tfrac{h}{\varepsilon^2} \ptial{\rho} 
		\bigg[\ptial{\rho} \big(\nu(\theta_0(\rho)) (\ptial{\rho} \hv_1 - (\bu_1 d_\Gamma - \bu_0 h_1) \eta'(\rho))\big)
		- \prho \hp_0 \nabla d_\Gamma + \ptial{\rho} \big(\nu(\theta_0(\rho)) (\bu_0 \rho \eta'(\rho))\big)
		\\
		& \qquad\qquad\quad 
		-  \prho \hv_0 (\pt d_\Gamma - \hv_0 \cdot \nabla d_\Gamma)
		+ \prho(2 \nu(\theta_0(\rho)) D_x \hv_0 \nabla d_\Gamma) 
		+ \Div_x \Big(
		\nu(\theta_0(\rho)) \prho \hv_0 \otimes_s \nabla d_\Gamma \Big) \\
        & \qquad\qquad\quad
        - \prho \big(\nu(\theta_0(\rho)) \Div_x \hv_0 \big) \nabla d_\Gamma
		\bigg]
		+ O(\sqrt{\eps}\|\ol{\we}\|_{L^2(0,T;H^1)}) \\
		& \quad \eqqcolon L_0 + O(\sqrt{\eps}\|\ol{\we}\|_{L^2(0,T;H^1)})
	\end{align*}
        in $L^2(0,T;H^{-1}(\Gamma_t(2\delta)^d)$.
	The additional term $ -L_0 $ will be used to deal with the leading order from capillary force.

	It remains to estimate the contribution from $\eps\Div\left(\nabla u_A^{\mathrm{in}}\otimes\nabla c_A^{\mathrm{in}}+\nabla c_A^{\mathrm{in}}\otimes\nabla u_A^{\mathrm{in}}\right)+ \nabla \hat{\pi} $. To this end, note that it is enough to just consider the inner expansion terms in $\Gamma(2\delta)$
	because the functions involved (and their derivatives) decay exponentially (up to $\theta_0$ that converges to $\pm 1$) with $\rho_\eps$ outside $\Gamma(\delta)$ and $u_A=0, c_A=\pm 1$ outside $\Gamma(2\delta)$. 
	We will use that
	\begin{align*}
		&\eps\Div\left(\nabla u_A^{\mathrm{in}}\otimes\nabla c_A^{\mathrm{in}}+\nabla c_A^{\mathrm{in}}\otimes\nabla u_A^{\mathrm{in}}\right)+ \nabla \hat{\pi}\\
		&\quad =\left(\eps\Delta u_A^{\mathrm{in}}-\tfrac{1}{\eps}f''(c_A^{\mathrm{in}})u_A^{\mathrm{in}}\right)\nabla c_A^{\mathrm{in}} 
		+\left(\eps\Delta c_A^{\mathrm{in}}-\tfrac{1}{\eps}f'(c_A^{\mathrm{in}})\right)\nabla u_A^{\mathrm{in}}. 
	\end{align*}
	Now we compute the terms of different orders in
		\begin{align}
	\left(\eps\Delta u_A^{\mathrm{in}}-\tfrac{1}{\eps}f''(c_A^{\mathrm{in}})u_A^{\mathrm{in}}\right)\nabla c_A^{\mathrm{in}} 
	+\left(\Delta c_A^{\mathrm{in}}-\tfrac{1}{\eps^2}f'(c_A^{\mathrm{in}})\right)\eps\nabla u_A^{\mathrm{in}}\label{modify-000}
		\end{align}
	and estimate the remainders in $L^2(0,T;H^{-1}(\Gamma_t(\delta))$. To this end we use
	\begin{alignat*}{2}
		\tfrac{1}{\eps^2}f'(c_A^{\mathrm{in}}) &= \tfrac1{\eps^2} f'(\theta_0)  + f''(\theta_0) \hc_2 +  O(\eps^2) &\quad &\text{in } L^\infty(\Gamma(\delta))
	\end{alignat*}
	and
	\begin{alignat*}{2}
		\tfrac{1}{\eps}f''(c_A^{\mathrm{in}})u_A^{\mathrm{in}} &= \tfrac{h}{\eps^2} f''(\theta_0) \theta_0' + h f''(\theta_0) \partial_\rho \hc_2 + h f'''(\theta_0)\theta_0' \hc_2  + O(\eps\|\ol{\we}\|_{L^2(0,T;H^1)}) &\quad & 
	\end{alignat*}
	in $L^2(0,T;L^\infty(\Gamma_t(\delta))$. In particular, we obtain
	\begin{align}\nonumber
		&\Delta c_A^{\mathrm{in}}-\tfrac{1}{\eps^2}f'(c_A^{\mathrm{in}}) = \tfrac{1}{\eps^2}(\theta_0'' -f'(\theta_0)) +\tfrac{1}{\eps}\theta_0' {\Delta d_\Gamma}\\\nonumber
		&\qquad +\prho^2 \hc_2-f''(\theta_0)\hc_2+\theta_0''(\rho)|\nabla^\Gamma h_1|^2-\theta_0'\Delta^\Gamma h_1+ O(\eps)\\\nonumber
		&\quad=\tfrac{1}{\eps}\theta_0' {\Delta d_\Gamma}
        + \theta_0'(\rho)\rho g_0
        - \theta_0'\Delta^\Gamma h_1+ O(\eps)\quad \text{in }L^\infty(\Gamma(\delta)),
	\end{align}
	where $\hc_2$ is as in \eqref{eq:c2},
	because of \eqref{eq:OptProfile} and \eqref{eq:DeltacA}. Moreover, it holds
	\begin{align}\nonumber
		&\eps\Delta u_A^{\mathrm{in}}-\tfrac{1}{\eps}f''(c_A^{\mathrm{in}})u_A^{\mathrm{in}}= \tfrac{1}{\eps^2}h(\theta_0'''-f''(\theta_0)\theta_0') +\tfrac{h}{\eps}\theta_0'' {\Delta d_\Gamma} + h\partial_\rho (\partial_\rho^2 \hc_2-f''(\theta_0)\hc_2 \\\nonumber
		&\qquad   +\theta_0'' |\nabla^\Gamma h_1|^2-\theta_0'\Delta^\Gamma h_1)-2\theta''_0\nabla^\Gamma h\cdot \nabla^\Gamma h_1+\theta_0' \Delta^\Gamma h +  O_{1,\eps}\\\label{eq:lambdaA}
		&\quad =\tfrac{h}{\eps}\theta_0'' {\Delta d_\Gamma} + h\partial_\rho (\rho \theta_0' g_0)
         - h\theta_0''\Delta^\Gamma h_1 -2\theta''_0\nabla^\Gamma h\cdot \nabla^\Gamma h_1+\theta_0' \Delta^\Gamma h +  O_{2,\eps},
	\end{align}
        where $O_{j,\eps}= O(\eps\|h\|_{X_{T,h}})$ in $L^2(0,T;L^\infty(\Gamma_t(2\delta))$ and $O_{j,\eps}= O(\eps^{\frac32}\|h\|_{X_{T,h}})$ in $L^\infty(0,T;L^2(\Gamma_t(2\delta)))$ for $j=1,2$, due to \eqref{eq:DeltauA}, \eqref{eq:c2} and
	\begin{equation*}
		-\theta_0'''(\rho)+f''(\theta_0(\rho))\theta_0'(\rho)=0 \qquad \text{for all }\rho\in\R,
	\end{equation*}
	which follows from \eqref{eq:OptProfile} by differentiation. 
	
	Using these identities together with \eqref{eq:nablacA} and \eqref{eq:nablauA} we obtain that the term of order $O(\eps^{-2})$ in \eqref{modify-000}  is
	\begin{align*}
		\tfrac{2 h}{\eps^2} \theta_0'(\rho) \theta_0''(\rho) \nabla d_\Gamma \Delta d_\Gamma = \tfrac{h}{\eps^2} \prho (\theta_0'(\rho))^2 \nabla d_\Gamma \Delta d_\Gamma.
	\end{align*}
	By using the equation for $ \hv_1 $ given by \cite[(A.43) together with (A.62)]{AbelsFei}, we have
	\begin{align*}
		L_0 = - \tfrac{h}{\eps^2} \prho \big((\theta_0'(\rho))^2 \nabla d_\Gamma \Delta d_\Gamma\big),
	\end{align*}
	which cancels the $ O(\eps^{-2}) $-term.
	The terms of order  $O(\tfrac{1}{\eps})$   in \eqref{modify-000} are:
	\begin{align*}
		&
		- \tfrac{2h}{\eps} \theta_0''(\rho)\theta_0'(\rho) \Delta d_\Gamma  \nabla^\Gamma h_1
		+\tfrac{(\theta_0'(\rho))^2}{\eps} \Delta d_\Gamma \nabla^\Gamma h
		+\tfrac{h}\eps \hat{a} \theta_0''(\rho) \nabla d_\Gamma   \\
		&\qquad + \tfrac{h}{\eps} \partial_\rho \hat{a} \,\theta_0'(\rho) \nabla d_\Gamma  - 2 (\Delta^\Gamma h_1+\nabla^\Gamma h \cdot \nabla^\Gamma h_1) \tfrac{\theta_0''(\rho) \theta_0'(\rho)}{\eps} \nabla d_\Gamma+ \tfrac{(\theta_0'(\rho))^2}{\eps} \nabla d_\Gamma \Delta^\Gamma h \\ 
		& \quad = \tfrac{(\theta_0'(\rho))^2}{\eps} \Delta d_\Gamma \nabla^\Gamma h
		+ \tfrac{h}{\eps}  \prho \left(\theta_0'\hat{a}\right) \nabla d_\Gamma - \tfrac{h}{\eps} \prho (\theta_0'(\rho))^2 \Delta d_\Gamma \nabla^\Gamma h_1\\
		&\qquad - \tfrac{1}{\eps} \prho (\theta_0'(\rho))^2 \nabla d_\Gamma (\Delta^\Gamma h_1+\nabla^\Gamma h \cdot \nabla^\Gamma h_1)  + \tfrac{(\theta_0'(\rho))^2}{\eps} \nabla d_\Gamma \Delta^\Gamma h
		+ O(\sqrt{\eps}\|\ol{\we}\|_{L^2(0,T;H^1)}) \\
		&\quad = \tfrac{(\theta_0'(\rho))^2}{\eps} \Delta d_\Gamma \nabla^\Gamma h
		+ \tfrac{(\theta_0'(\rho))^2}{\eps} \nabla d_\Gamma \Delta^\Gamma h 
		+ O(\sqrt{\eps}\|\ol{\we}\|_{L^2(0,T;H^1)})
	\end{align*}
	in $ L^2(0,T;H^{-1}(\Gamma_t(\delta)))$ because of Lemma~\ref{lem:f'-phi-x-dependent}. 
	Therefore, we find 
	\begin{align*}
		& \eps\Div\left(\nabla u_A\otimes\nabla c_A+\nabla c_A\otimes\nabla u_A\right) - \nabla \hat{\pi} + L_0 \\
		& \qquad = \tfrac{(\theta_0'(\rho))^2}{\eps} \Delta d_\Gamma \nabla^\Gamma h
		+ \tfrac{(\theta_0'(\rho))^2}{\eps} \nabla d_\Gamma \Delta^\Gamma h 
		+ O(\sqrt{\eps}\|\ol{\we}\|_{L^2(0,T;H^1)})
	\end{align*}
	in $ L^2(0,T;H^{-1}(\Gamma_t(\delta)))$ and in $L^2(0,T;L^2(\Gamma_t(\delta); \oeps^2 \dx))$.
	
	Altogether, using \eqref{eq:Limit1}  
	\begin{align*}
		\partial_t\we_A & +\ve_A\cdot\nabla\we_A+\we_A\cdot\nabla\ve_A +\nabla q_A \\
		& 
		-\Div\left(2\nu(c_A)D\we_A\right)
		-\Div\left(2\nu'(c_A)u_AD\ve_A\right) 
		+\eps\Div\left(\nabla u_A\otimes\nabla c_A+\nabla c_A\otimes\nabla u_A\right)
	\end{align*}
	is controlled by $C\sqrt{\eps}\|\ol{\we}\|_{L^2(0,T;H^1)}$ in $L^2(0,T;V_\sigma')$ because of \eqref{eq:Limit1}.
	Note that the remaining two terms of $\tfrac{(\theta_0'(\rho))^2}{\eps} \Delta d_\Gamma
	\nabla^\Gamma h$ and $\tfrac{(\theta_0'(\rho))^2}{\eps} \nabla d_\Gamma
	\Delta^\Gamma h$ can be shown to cancel with the jump term $\llbracket 2\nu^\pm D\we^\pm -q^\pm \tn{I} \rrbracket \no_\Gamma\delta_\Gamma$ in \eqref{eq:jump-outer} from above up to the desired error by \eqref{eq:Limit3}, compare with \cite[proof of Theorem 4.1]{AbelsFeiMoser} by Lemma \ref{lem:2delta-Gamma_t}. This concludes the proof of the first part of Theorem \ref{th_leading_velocity}.
	
	Let us now show the second part in $L^2(\Omega;\oeps^2 \dx)$-spaces.   First of all, the previous estimates already yield
	\begin{align*}
		& \eps\Div\left(\nabla u_A\otimes\nabla c_A+\nabla c_A\otimes\nabla u_A\right)+\nabla (\zg \hat{\pi}) = O(\sqrt{\eps}\|\ol{\we}\|_{L^2(0,T;H^1)})
	\end{align*}
	in $L^2(Q_T; \oeps^2 \, \d(x,t))$ since
          \begin{equation*}
            \|\oeps \tfrac{(\theta_0'(\rho))^2}{\eps}( \Delta d_\Gamma\nabla^\Gamma h+\nabla d_\Gamma
	\Delta^\Gamma h)\|\leq C\sqrt{\eps}\|\ol{\we}\|_{L^2(0,T;H^1)}.
      \end{equation*}
      Moreover, we use that
      \begin{align}\nonumber
        & \Div (2\nu(c_A) D \we_A^{\mathrm{out}}) = 2\nu(c_A) \Div (D\we_A^{\mathrm{out}}) + 2 \nu'(c_A) \nabla c_A\cdot D\we_A^{\mathrm{out}}\\\nonumber
        &\quad = (1-\zg)(\nu^+\Delta \we^+ \chi_{\Omega+} + \nu^-\Delta \we^- \chi_{\Omega-})+ \zg \nu(c_A)(\Delta \we^+ \eta(\rho)+\Delta \we^- (1-\eta(\rho)))
        \\\nonumber
        &\qquad + \zg 2\nu(c_A)\nabla (\eta(\rho))\cdot (D\we^+-D\we^-)+ \zg \Delta (\eta(\rho))(\we^+-\we^-) + O(\sqrt{\eps}\|\ol\we\|_{L^2(0,T;H^1)})\\
        &\quad = \nu^+\Delta \we^+ \chi_{\Omega+} + \nu^-\Delta \we^- \chi_{\Omega-}+ O(\sqrt{\eps}\|\ol\we\|_{L^2(0,T;H^1)}) \label{eq:div-DwAout}
      \end{align}
      in $L^2(Q_T;\oeps^2\sd x)$ by Lemma~\ref{lem:wout}, the matching conditions, $\eta(\rho)=1$ for $\rho\geq 1$, and $\eta(\rho)=0$ for $\rho\leq -1$. Here we used
      $\omega_\eps \nabla c_A=\zg \tfrac{\oeps}{\eps} \theta_0'(\rho)\nabla (d_\Gamma-\eps h_\eps)+O(\eps) $ in $L^\infty(Q_T)$ and that $D\we_A^{\mathrm{out}}$ in $(r,s)$-coordinates is in $L^2((0,T)\times\Sigma;L^\infty(-2\delta,2\delta))$. Furthermore, we used that $\we_A^{\mathrm{out}}$ in $(r,s)$-coordinates is in $L^2((0,T)\times\Sigma;C^1([-2\delta,2\delta]))$ and $\we^+|_{\Gamma}=\we^-|_{\Gamma}$ to show
      \begin{equation*}
        \|\oeps\Delta (\eta(\rho)) (\we^+-\we^-)\|_{L^2(Q_T\cap \Gamma(2\delta))} \leq C\sqrt{\eps}\norm{\ol{\bw}}_{L^2(0,T;H^1)}
      \end{equation*}
      similarly as in the proof of \eqref{eq:woutEstim}. Here, because of \eqref{eq:Limit1},  $\nu^+\Delta \we^+ \chi_{\Omega+}+\nu^-\Delta \we^- \chi_{\Omega-}$ will cancel with the corresponding terms in $\partial_t \we_A^{\mathrm{out}}+\ve\cdot \nabla \we_A^{\mathrm{out}}+\we_A^{\mathrm{out}}\cdot \nabla \ve$ calculated before.
    
	We need to control
	\begin{align*}
		(\ve_A-\ve)\cdot\nabla\we_A
		+ \we_A\cdot\nabla(\ve_A-\ve)
		-\Div(2(\nu(c_A) - \nu^\pm)D\we_A)
		-\Div(2\nu'(c_A)u_AD\ve_A) ,
	\end{align*}
	in $L^2(Q_{T};\oeps^2 \, \d(x,t))$. First, by construction \eqref{eq:expvA} we recall
    \begin{align*} 
        \|\mathbf v_A-\mathbf v\|_{L^\infty(Q_T)} + \|\omega_\varepsilon \nabla(\mathbf v_A-\mathbf v)\|_ {L^\infty(Q_T)} &\leq C\varepsilon.
    \end{align*}
    In addition by \eqref{eq:w_A} and \eqref{eq:h-w-estimate}, we have
    \begin{align*} 
		\|\mathbf w_A^{\mathrm{in}}\|_{L^2(Q_T)} + \|\omega_\varepsilon\nabla\mathbf w_A^{\mathrm{in}}\|_{L^2(Q_T)} 
        &\leq C\sqrt{\varepsilon} \|h\|_{L^2(0,T;H^1(\Sigma))} \leq C\sqrt{\varepsilon} \norm{\ol{\bw}}_{L^2(0,T;H^1)},
	\end{align*} 
    and hence
    \begin{align*}
        \|\mathbf w_A\|_{L^2(Q_T)} + \|\omega_\varepsilon\nabla\mathbf w_A\|_{L^2(Q_T)} 
        & \leq C\sqrt{\varepsilon} \norm{\ol{\bw}}_{L^2(0,T;H^1)}
    \end{align*}
    In this case, we find
	\begin{align*}
		\norm{\oeps (\ve_A-\ve)\cdot\nabla\we_A}_{L^2(Q_T)}
		& \leq \norm{\ve_A-\ve}_{L^\infty(Q_T)}
		\norm{\oeps \nabla\we_A}_{L^2(Q_T)} \leq C \sqrt{\eps} \norm{\ol{\bw}}_{L^2(0,T;H^1)},
	\end{align*}
	and
	\begin{align*}
		\norm{\oeps \we_A\cdot\nabla(\ve_A-\ve)}_{L^2(Q_T)}
		& \leq 
		\norm{\oeps \nabla(\ve_A-\ve)}_{L^\infty(Q_T)}
		\norm{\we_A}_{L^2(Q_T)} 
        \leq C \sqrt{\eps} \norm{\ol{\bw}}_{L^2(0,T;H^1)}.
	\end{align*}
	Moreover, direct computations yield
	\begin{align*}
		\Div(2\nu'(c_A) u_A D\bv_A)
		& = 2\nu'(c_A) u_A \Div D\bv_A
		+ 2\nu'(c_A) D \bv_A \nabla u_A
		+ 2\nu''(c_A) D \bv_A \nabla c_A u_A
	\end{align*}
	in $Q_T$. Hence
	\begin{align*}
		& \norm{\oeps \Div(2\nu'(c_A) u_A D\bv_A)}_{L^2(Q_T)} \\
		& \quad 
		\leq \norm{\oeps u_A}_{L^\infty(0,T;L^2(\Omega))}
		\norm{\Div D\bv_A}_{L^2(0,T;L^\infty(\Omega))} \norm{\nu'(c_A)}_{L^\infty(Q_T)} \\
		& \qquad + 
		\norm{\oeps \nabla u_A}_{L^2(0,T;L^3(\Omega))}
		\norm{2 D\bv_A}_{L^\infty(0,T;L^6)} \norm{\nu'(c_A)}_{L^\infty(Q_T)} \\
		& \qquad + 
		\norm{u_A}_{L^2(0,T;L^3(\Omega))}
		\norm{\oeps \nabla c_A}_{L^\infty(0,T;L^\infty(\Omega))}
		\norm{2 D\bv_A}_{L^\infty(0,T;L^6)} \norm{\nu''(c_A)}_{L^\infty(Q_T)} \\
		& \quad \leq C \sqrt{\eps} \norm{\ol{\bw}}_{L^2(0,T;H^1)}.
	\end{align*} 
	Putting everything together completes the proof of $\mathbf{r}_{1,A}$.

        To estimate $r_{2,A}$ we use 
          \begin{equation*}
            r_{2,A} = - \Div \bw_A= -\Div(\bw_A^{\mathrm{out}})
            +\zg\Div\bw_A^{\mathrm{in}}
            +\nabla\zg\cdot\bw_A^{\mathrm{in}}.        
          \end{equation*}
          Here
          \begin{equation*}
            \Div\bw_A^{\mathrm{in}} = \tfrac{h}{\eps^2} \left(\partial_\rho^2\ve_A^{\mathrm{in}}\cdot \nabla (d_\Gamma-\eps h_\eps)+ \partial_\rho \Div_x \ve_A^{\mathrm{in}}\right)+ \tfrac{\nabla^\Gamma h}{\eps}\cdot \partial_\rho \ve_A^{\mathrm{in}} = \tfrac{h}{\eps^2}\partial_\rho \hat{g}(\rho,x,t) + \tfrac{\nabla^\Gamma h}{\eps}\cdot \partial_\rho \ve_A^{\mathrm{in}}, 
          \end{equation*}
          where $\hat{g}$ is as in Remark~\ref{rem:PropertiesAproxSol}, $\partial_\rho \ve_A^{\mathrm{in}} = d_\Gamma \eta'(\rho)\ue_0+ O(\eps)$, $d_{\Gamma} = \eps (\rho + h_\eps)$, and $\nabla d_\Gamma\cdot \ue_0|_{\Gamma_t}= \partial_{\no_{\Gamma}} (\we^+-\we^-)|_{\Gamma}\cdot \no_{\Gamma_t}=0.$
          From this one derives together with Lemma~\ref{lem:wout} 
          \begin{align*}
            \normm{\Div \bw_A}_{L^\infty(0,T;L^2)}
        \leq C \sqrt{\eps} \norm{\ol{\bw}}_{L^2(0,T;H^1)}.
    \end{align*}
        Then similarly, with weight $\oeps$, one derives 
        \begin{align*}
            \normm{\oeps \nabla \Div \bw_A}_{L^2(Q_T)}
        \leq C \sqrt{\eps} \norm{\ol{\bw}}_{L^2(0,T;H^1)}.
        \end{align*}
It remains to estimate $\partial_t \Div \we_A$. Because of \eqref{eq:DivwoutEstim}, we only need to estimate
        \begin{align*}
          \partial_t \Div (\zg \we_A^{\mathrm{in}})= \Div \left(\zg \tfrac{h}{\eps^2} \partial_\rho^2 \ve_A^{\mathrm{in}}\partial_t (d_\Gamma -\eps h_\eps)\right)+ \Div (\tfrac{d_\Gamma}{\eps}\eta'(\rho)\ue \partial_t(\zg h)),
        \end{align*}
        where
        \begin{equation*}
          \|\Div (\tfrac{d_\Gamma}{\eps}\eta'(\rho)\ue \partial_t(\zg h))\|_{L^2(0,T;H^{-1}_{(0)})}\leq C\|\tfrac{d_\Gamma}{\eps}\eta'(\rho)\ue \partial_t(\zg h))\|_{L^2(Q_T)}\leq C\sqrt\eps \|\ol{\we}\|_{L^2(0,T;H^1)}.
        \end{equation*}
        Moreover,
        \begin{equation*}
          \Div \left(\zg \tfrac{h}{\eps^2} \partial_\rho^2 \ve_A^{\mathrm{in}}\partial_t (d_\Gamma -\eps h_\eps)\right) =  \zg \tfrac{h}{\eps^2} \partial_\rho^2 \hat{g}(\rho,x,t)\partial_t (d_\Gamma -\eps h_\eps)+ \tfrac{d_\Gamma}{\eps^2} \eta''(\rho)\ue\cdot \nabla (\zg h (\partial_t (d_\Gamma -\eps h_\eps)).
        \end{equation*}
        By the decay properties of $\hat g $, cf.\ Remark~\ref{rem:PropertiesAproxSol}, the first term can easily be estimate. For the second term we use Lemma~\ref{lem:f'-phi-x-dependent} to obtain
        \begin{equation*}
          \left|\int_\Omega \tfrac{d_\Gamma}{\eps^2} \eta''(\rho)\ue\cdot \nabla (\zg h (\partial_t (d_\Gamma -\eps h_\eps))\varphi\sd x\right|\leq C\sqrt\eps \|\varphi\|_{H^1(\Omega)}\|h\|_{H^1(\Sigma)}\quad \text{for all }\varphi\in H^1_{(0)}(\Omega).
        \end{equation*}
        Altogether this proves
        \begin{align*}
            \norm{r_{2,A}}_{L^2(Q_{T})}
	       + \norm{r_{2,A}}_{H^{1}(0,T; H^{-1}_{(0)}(\Omega))}
	       + \norm{\oeps \nabla r_{2,A}}_{L^2(Q_{T})} \leq C\sqrt{\eps}\|\ol{\we}\|_{L^2(0,T;H^1)}.
        \end{align*}
    This concludes the proof.
\end{proof}

\begin{theorem}\label{thm:LeadingPart}
	Let $r_{3,A}$ be defined as in \eqref{eq:r3}. Then there are $\eps_1\in (0,1]$ and $C>0$ such that
	\begin{align*}
		\|r_{3,A}\|_{L^2(0,T;(\Veps)')}
		& \leq C(M)\sqrt{\eps} \|\ol{\we}\|_{L^2(0,T; H^1(\Omega))},\\
		\sum_{j+k\leq 1}\norm{\tfrac{\oeps^{1+j+k}}{\eps^j}\nabla^kr_{3,A}}_{L^2(Q_T)} & \leq C\sqrt\eps\|\ol{\we}\|_{L^2(0,T; H^1(\Omega))}.
	\end{align*}
        for all $\eps\in (0,\eps_1]$, $T\in (0,T_0]$.
\end{theorem}
\begin{proof}
	The proof is very similar to the proof of \cite[Theorem~3.3]{AbelsFei}. Therefore we shall be brief and refer to this paper for further details.
	
	First of all we have for all $\psi\in H^1(\Gamma_t(2\delta))$
	\begin{equation*}
		\int_{\Gamma_t(2\delta)} r_{3,A} \psi \dx =
		\int_{\Gamma_t(\delta)} r_{3,A} \psi \dx +
		\int_{\Gamma_t(2\delta)\setminus \Gamma_t(\delta)} r_{3,A} \psi \dx,  
	\end{equation*}
	where it is enough to consider the integral in $\Gamma_t(\delta)$ since in $\Gamma_t(2\delta) \setminus \Gamma_t(\delta)$ all terms in $r_{3,A}$ decay exponentially. In $\Gamma_t(\delta)$ we have $u_A=-u_A^{\mathrm{in}}$ and $c_A=c_A^{\mathrm{in}}$.
	
	By definitions of $\bv_A = \bv_0 + \varepsilon \bv_{1,\eps} $ and $ \bw_A $, we obtain using \eqref{eq:nablacA} and \eqref{eq:nablauA}
	\begin{align*}
		\bv_A \cdot \nabla u_A^{\mathrm{in}}
		& = \tfrac{h}{\varepsilon^2} \theta_0''(\rho) \bv_0 \cdot \nabla d_{\Gamma}
		+ \tfrac{h}{\varepsilon} \theta_0''(\rho) \bv_{1,\eps} \cdot \nabla d_\Gamma \\
		& \quad 
		- \frac{h}{\varepsilon} \theta_0''(\rho) \bv_0 \cdot \nabla^\Gamma h_1
		+ \frac{1}{\varepsilon} \theta_0'(\rho) \bv_0 \cdot \nabla^\Gamma h
		+ O(\sqrt \varepsilon\|\ol{\we}\|_{L^2(0,T;H^1)}),\\
		\bw_A \cdot \nabla c_A
		& =  \bw_A \cdot \nabla d_{\Gamma} \tfrac1{\varepsilon}\theta_0'(\rho) 
		+ O(\sqrt \varepsilon\|\ol{\we}\|_{L^2(0,T;H^1)})
	\end{align*}
	in $L^2(0,T;L^2(\Gamma_t(\delta)))$.
	Now using \eqref{eq:partialtuA} and \eqref{eq:lambdaA}  we conclude 
	\begin{align*}
		- r_{3,A}
		& = - \tfrac{h}{\varepsilon^2} \theta_0''(\rho) \Big(\pt d_\Gamma + \bv_0 \cdot \nabla d_\Gamma - \Delta d_\Gamma\Big) \\
		& \quad + \tfrac{h}{\varepsilon} \Big(
        \partial_\rho (\theta_0'(\rho)\rho g_0)
		+ \theta_0''(\rho) \big( \pt^\Gamma h_1 + \bv_0 \cdot \nabla^\Gamma h_1-\Delta^\Gamma h_1
        - \bv_{1,\eps} \cdot \nabla d_\Gamma\big)
		\Big) \\
		& \quad
		- \tfrac{2\theta_0''(\rho)}{\eps} \nabla^\Gamma h \cdot \nabla^\Gamma h_1
        + \tfrac{\theta_0'(\rho)}{\varepsilon} \Big(
		- \pt^\Gamma h - \bv_0 \cdot \nabla^\Gamma h + \bw_A \cdot \nabla d_{\Gamma}
		+ \zg\ol{\bw}|_{\Gamma_t} \cdot \nabla d_{\Gamma}
		+ \Delta^\Gamma h
		\Big)\\
		&\quad
		+ O(\sqrt \varepsilon\|\ol{\we}\|_{L^2(0,T;H^1)})\\
		&= 
		\theta_0''(\rho)\tfrac{h}\eps \big( \pt^\Gamma h_1 + \bv_0 \cdot \nabla^\Gamma h_1+g_0 h_1-\Delta^\Gamma h_1-\bv_{1,\eps} \cdot \nabla d_\Gamma\big)
		\\
		& \quad - \tfrac{2\theta_0''(\rho)}{\eps} \nabla^\Gamma h \cdot \nabla^\Gamma h_1
        -\tfrac{\theta_0'(\rho)}{\varepsilon} \Big(
		\pt^\Gamma h + \bv_0 \cdot \nabla^\Gamma h - \bw_A \cdot \nabla d_{\Gamma}
		- \zg\ol{\bw}|_{\Gamma_t} \cdot \nabla d_{\Gamma}
		+g_0h - \Delta^\Gamma h
                  \Big)\\
          &\quad+ O(\sqrt \varepsilon\|\ol{\we}\|_{L^2(0,T;H^1)})
	\end{align*}
	in $L^2(0,T;L^2(\Gamma_t(\delta)))$ since $\pt d_\Gamma + \bv_0 \cdot \nabla d_\Gamma - \Delta d_\Gamma=-\eps \rho g_0 - \eps h_1 g_0 + O(\eps^2)$. Here the bracket with $\theta_0'$ can be controlled due to Lemma \ref{lem:2delta-Gamma_t} and \eqref{eq:linTwoPhase5}, where $\bw_A^{\mathrm{in}}$ can be controlled by construction. Altogether we obtain
	\begin{equation*}
		r_{3,A}(x,t) = \tfrac1\eps a(\rho,x,t) + O(\sqrt \varepsilon\|\ol{\we}\|_{L^2(0,T;H^1)}) \qquad \text{in }L^2(0,T;H^{-1}(\Gamma_t(\delta))),
	\end{equation*}
	where
	\[
	\int_{\mathbb{R}} a(\rho, x, t) \theta_0^{\prime}(\rho) \,d \rho=0
	\quad \text { for all } x\in \Gamma_t, t\in [0,T].
	\]
	Therefore we can decompose $\tfrac1\varepsilon a(\rho,x,t)=f_1(x,t)+f_2(x,t)$ with
	\[
	f_1(x, t)=\tfrac1\varepsilon a(\rho, P_{\Gamma_t}x, t),\quad
	f_2(x, t)=\tfrac1\varepsilon(a(\rho, x, t)-a(\rho, P_{\Gamma_t}x, t)).
	\]
	Because of Lemma~\ref{lem:EstimMeanValueFree} and Lemma~\ref{lem:rescale}, we get
	\[
	\left\|f_1\right\|_{L^2\left(0,T; V_{\varepsilon}^{\prime}\right)}
	\leq C \sqrt\varepsilon\|\ol\we\|_{L^2(0,T;H^1)},\qquad
	\left\|f_1\right\|_{L^2\left(0,T; L^2\left(\Gamma_t(\delta)\right)\right)}
	\leq C \varepsilon^{-\frac12}\|\ol\we\|_{L^2(0,T;H^1)},
	\]
	and, using Lemma~\ref{lem:rescale} again,
	\[
	\left\|f_2\right\|_{L^2\left(0,T; \Veps\right)}
	\leq C\left\|f_2\right\|_{L^2\left(0,T;
		L^2\left(\Gamma_t(2 \delta)\right)\right)}
	\leq C' \sqrt\varepsilon \|\ol\we\|_{L^2(0,T;H^1)}.
	\]
	Finally, using $\oeps^j \theta_0^{(k)}, \oeps^j \partial_\rho^k \nabla_x^l   \hc_2\in \mathcal{R}_{j,\alpha}$ for every $k\in \N,j,l \in\N_0$ and the identities above one can show in a straight forward manner
	\begin{align*}
		\norm{\tfrac{\oeps^{1+j+k}}{\eps^j} \nabla^k r_{3,A}}_{L^2(0,T;L^2(\Gamma_t(\delta)))}
		\leq C \sqrt\eps \left(\norm{h}_{L^\infty(0,T;H^{1})} + \norm{h}_{L^2(0,T;H^{\frac{5}{2}})}\right)
		\leq C \sqrt\eps\norm{\ol \we}_{L^2(0,T; H^1)}.
	\end{align*}
	for all $j+k\leq1$ with the aid of Lemma~\ref{lem:rescale}.
	This finishes the proof.
\end{proof}

\noindent
\begin{proof*}{of Theorem~\ref{thm:FullLinearizedSystem}}
	First of all, using that $(\ol{\we}, \ol{u})$ solves the modified linearized system \eqref{eq:linNSAC'} with $\ol{\mathbf{r}}_1= \mathbf{r}_1+\mathbf{r}_{1,A}$, $\ol{r}_2= r_2+r_{2,A}$, and $\ol{r}_3=r_3+r_{3,A}$ the Theorems~\ref{thm:main-0}, \ref{th_leading_velocity}, and \ref{thm:LeadingPart} yield
	\begin{align*}
		&\|\ol{u}\|_{L^\infty(0,T;L^2)}+\|\ol u\|_{L^2(0,T;\Veps)}+ \|\nabla_\tau \ol u\|_{L^2(0,T;L^2(\Gamma_t(2\delta)))}+\|(\oeps \nabla \ol u, \tfrac{\oeps}\eps \ol u)\|_{L^2(Q_T)}\\
		&\quad +\|\ol{\we} \|_{L^\infty(0,T;L^2)} +\|\nabla \ol{\we}\|_{L^2(Q_{T})}\\
		&\leq C\left(\|\ol{u}_0\|_{L^2(\Omega)}+\|\ol\we_0\|_{H^1(\Omega)}
        + \|\ol{\mathbf{r}}_1\|_{L^2(0,T;V_\sigma')}+ \|(\ol r_2,\oeps \nabla \ol r_2)\|_{L^2(Q_T)}\right.\\
		&\qquad \left.  + \| \ol{r}_2\|_{H^1(0,T; H^{-1}_{(0)}(\Omega))}+\|\ol{r}_3\|_{L^2(0,T;(\Veps)')}+\|\oeps^2\ol{r}_3\|_{L^2(Q_T)}\right)\\
		&\leq C\left(\|u_0\|_{L^2(\Omega)}+\|\we_0\|_{H^1(\Omega)}+\|\mathbf{r}_1\|_{L^2(0,T;V_\sigma')}+ \|(\ol r_2,\oeps \nabla \ol r_2)\|_{L^2(Q_T)} \right.\\
		&\qquad \left.  + \| r_2\|_{H^1(0,T; H^{-1}_{(0)}(\Omega))}+\|r_3\|_{L^2(0,T;(\Veps)')}+\|\oeps^2r_3\|_{L^2(Q_T)}+\sqrt{\eps}\|\ol \we\|_{L^2(0,T;H^1)}\right).
	\end{align*}
	Hence choosing $\eps_1\in (0,1]$ sufficiently small, we obtain
	\begin{align*}
		&\|\ol{u}\|_{L^\infty(0,T;L^2)}+\|\ol u\|_{L^2(0,T;\Veps)}+ \|\nabla_\tau \ol u\|_{L^2(0,T;L^2(\Gamma_t(2\delta)))}+\|(\oeps \nabla \ol u, \tfrac{\oeps}\eps \ol u)\|_{L^2(Q_T)}\\
		&\quad +\|\ol{\we} \|_{L^\infty(0,T;L^2)} +\|\nabla \ol{\we}\|_{L^2(Q_{T})}\\
		&\leq C\left(\|u_0\|_{L^2(\Omega)}+\|\we_0\|_{H^1(\Omega)}+\|\mathbf{r}_1\|_{L^2(0,T;V_\sigma')}+ \|(\ol r_2,\oeps \nabla \ol r_2)\|_{L^2(Q_T)}\right.\\
		&\qquad \left. + \| r_2\|_{H^1(0,T; H^{-1}_{(0)}(\Omega))} +\|r_3\|_{L^2(0,T;(\Veps)')}+\|\oeps^2r_3\|_{L^2(Q_T)}\right).
	\end{align*}
	Now combining this estimate with Theorems~\ref{thm:HighOrderEstimate} and \ref{thm:HigherOrderW} (and \eqref{eq:h-w-estimate}), Proposition~\ref{prop:wpm-h}, and Lemma \ref{lem:wout}  yields \eqref{eq:MainUbarEstim}. Finally, \eqref{eq:uAEstim2} is a consequence of Proposition \ref{prop:uAEstim} together with \eqref{eq:h-w-estimate2}.
\end{proof*}

\section{Convergence Result}\label{sec:Convergence}
This section is devoted to the proof of our first main result, i.e., Theorem~\ref{thm:main}. To this end, we use the results for the linearized system and establish suitable estimates for the terms due to linearization of the nonlinear terms and include remainders for the approximation solutions. With all these ingredients with leading error estimates in Theorems \ref{th_leading_velocity} and \ref{thm:LeadingPart}, we use a continuation argument and that the nonlinear terms are dominated by the linearized system if $\eps\in (0,\eps_0]$ and $\eps_0>0$ is sufficiently small to show the main result.

Let $\oeps$ be as in \eqref{eq:weight} and $0 < T \leq  T_0$. We define $X_{T} \coloneqq X_{T}^1 \times X_{T}^2$ with
\begin{align*}
	X_{T}^1 & \coloneqq L^2(0,T;H^2(\Omega \setminus \Gamma_t)) \cap H^1(0,T;L^2(\Omega)) \cap L^\infty(0,T;H^1(\Omega)), \\
	X_{T}^2 & \coloneqq L^2(0,T;H^3(\Omega)) \cap H^1(0,T;H^1(\Omega)) \cap L^\infty(0,T;H^2(\Omega)),
\end{align*}
which are endowed with the $\eps$-dependent norms
\begin{align*}
	\norm{\bw}_{X_{T,\eps}^1}
	& \coloneqq \norm{(\bw,\oeps \nabla \bw)}_{L^\infty(0,T;L^2)}
	+ \norm{\bw}_{L^2(0,T;H^1)} 
	+ \norm{(\oeps \pt \bw,\oeps \nabla^2 \bw)}_{L^2(Q_{T})}
\end{align*}
and
\begin{align*}
	\norm{u}_{X_{T,\eps}^2}
  & \coloneqq \left\|u\right\|_{L^\infty(0,T;\wW^2)}+
		\left\|(u, \nabla_\btau u)\right\|_{L^2 (0,T;\wW^2)}+ \|\oeps \partial_t u\|_{L^2 (0,T;\wW^1)} \\
		& \qquad + \left\|\oeps\mathcal{L}_\eps u\right\|_{L^2 (0,T;\wW^1(\Gamma_t(2\delta))}+ \|(\oeps^2 \nabla^2 u, \oeps^3 \nabla^3 u )\|_{L^2 (Q_{T})}.
\end{align*} 
Moreover, let $Y_{T} \coloneqq Y_{T}^1 \times Y_{T}^2 \times Y_{T}^3$ such that
\begin{align*}
	Y_{T}^1 & \coloneqq L^2(Q_{T}), \quad	Y_{T}^2  \coloneqq L^2(0,T;H_{(0)}^1(\Omega)) \cap H^1(0,T; H^{-1}_{(0)}(\Omega)), \quad
	Y_{T}^3 \coloneqq L^2(0,T;H^1_{0}(\Omega)),
\end{align*}
endowed with norms
\begin{align*}
	& \norm{\mathbf{r}_1}_{Y_{T,\eps}^1} \coloneqq \norm{\mathbf{r}_1}_{L^2(0,T;V_\sigma')}
	+ \norm{\oeps \mathbf{r}_1}_{L^2(Q_{T})}, \\
	& \norm{r_2}_{Y_{T,\eps}^2} \coloneqq \norm{r_2}_{L^2(Q_{T})}
	+ \norm{r_2}_{H^{1}(0,T; H^{-1}_{(0)}(\Omega))}
	+ \norm{\oeps \nabla r_2}_{L^2(Q_{T})}, \\
	& \norm{r_3}_{Y_{T,\eps}^3} \coloneqq \norm{r_3}_{L^2(0,T;(V_t^\eps)')}
	+ \norm{\oeps r_3 }_{L^2(0,T;\wW^1)}.
\end{align*}

Now let Assumption~\ref{assump:main} be satisfied, let $(\ve_A,p_A, c_A)$ be as in Theorem~\ref{thm:ApproximateSolution} and set $\ve_{A,0}:= \ve_A|_{t=0}$, $c_{A,0}:= c_A|_{t=0}$. Moreover, let $(\ve_\eps, p_\eps, c_\eps)$ be the solution of \eqref{eq:NSAC} with initial values $(\ve_{0,\eps}, c_{0,\eps})$ satisfying \eqref{initial assumption} for some $R_0>0$.  
Then $u\coloneqq c_\eps-c_A$, $\we\coloneqq \ve_\eps-\ve_A$, $q\coloneqq p_\eps-p_A$ solve the linearized system \eqref{eq:linNSAC} with
\begin{equation*}
  (\br_1,r_2,r_3)= (\bR_1,0,R_3)- (\bR_\eps,G_\eps,S_\eps),
\end{equation*}
where $(\bR_\eps,G_\eps,S_\eps)$ are the remainder terms from the approximate solution in Theorem~\ref{thm:ApproximateSolution} and
\begin{align}
	\mathbf{R}_1
	& \coloneqq 
	-\bw \cdot \nabla \bw 
	- \eps \Div (\nabla u \otimes \nabla u)- \Div \left(2 \big(\nu(c_A) - \nu(c_\eps)\big) D \bw\right)\nonumber  \\\label{eq:R1}
	& \quad\  - \Div \left(2 \left(\nu(c_A)+\nu'(c_A)u - \nu(c_\eps)\right) D \bv_A\right), \\\label{eq:R3}
	R_3
	 &\coloneqq -\bw \cdot \nabla u
	- \tfrac{1}{\eps^2} \left[f'(c_\eps) - f'(c_A) - f''(c_A) u\right]
\end{align}
are error terms due to linearization. We note that we have the compatibility condition $r_2|_{t=0}= -G_\eps|_{t=0}= \Div \we_0$ by construction and $r_3|_{\partial\Omega}=R_3|_{\partial\Omega}= S_\eps|_{\partial\Omega}=0$ due to $c_\eps|_{\partial\Omega}=c_A|_{\partial\Omega}=-1$, $f'(-1)=0$, $u|_{\partial\Omega}=0$ and Theorem~\ref{thm:ApproximateSolution}.
Furthermore, let $u=u_A+\ol u$ be the decomposition of $u$ as in Theorem~\ref{thm:FullLinearizedSystem}.
Then $u_A$ is uniquely determined by $\we$, cf.\ Remark~\ref{rem:uA}. By \eqref{eq:MainUbarEstim} we have for any $T\in (0,T_0]$ and $\eps\in (0,\eps_1]$
\begin{equation}\label{eq:MainUbarEstim-recall}
  \|(\ol \we, \ol u)\|_{X_{T,\eps}}\leq C(T_0) \left(\|(\br_1,r_2,r_3)\|_{Y_{T,\eps}}+\norm{(\bw_0, u_0)}_{X_\gamma}
  \right),
\end{equation}
where
\begin{align*}
	\norm{(\bw_0,u_0)}_{X_{\gamma,\eps}}
	= \left\| u_0\right\|_{\wW^2} +\|\nabla_\btau u_0\|_{L^2(\Gamma_0(2\delta))} 
    + \|\we_0\|_{H^1(\Omega)}
\end{align*}
and $\eps_1\in (0,1]$ is chosen so small that the estimates of Sections~\ref{sec:higher-order} and \ref{sec:FullSystem} hold true.
We note that using the Gagliardo--Nirenberg inequality $\norm{u}_{L^\infty(\Omega)} \leq C_{\rm GN} \norm{u}_{H^1(\Omega)}^{\onehalf} \norm{u}_{H^2(\Omega)}^{\onehalf}$ and $\norm{u}_{H^1(\Omega)}\leq \eps^{-1}\norm{u}_{H_\eps^1(\Omega)}$ and $\norm{u}_{H^2(\Omega)}\leq \eps^{-2}\norm{u}_{H_\eps^2(\Omega)}$. Thus we obtain
\begin{equation}
  \label{eq:L-infty-Estim}
  \|\ol u\|_{L^\infty(Q_{T})}\leq C_{\rm GN} \eps^{-\frac32} \|\ol u\|_{X_{T_\eps,\eps}^2}.
\end{equation}

Now we choose $R>0$ such that
\begin{equation*}
  R>2C(T_0)R_0,
\end{equation*}
where $C(T_0)$ is as in \eqref{eq:MainUbarEstim-recall}. Then, by the assumptions on the initial values \eqref{initial assumption}, for every $\eps\in (0,\eps_1]$ there is some $T_\eps \in (0,T_0]$ such that
\begin{equation}
  \label{eq:AprioriBound}
  \|(\ol \we, \ol u)\|_{X_{T_\eps,\eps}}\leq R\eps^{N+\frac12}.
\end{equation}
We choose $T_\eps\in (0,T_\eps]$ maximal such that system \eqref{eq:NSAC} has a smooth solution $\ve_\eps\colon \ol{\Omega}\times [0,T_\eps)\to \R^d$, $(p_\eps,c_\eps) \colon \ol{\Omega}\times [0,T_\eps)\to \R^2$ satisfying \eqref{eq:AprioriBound}. In order to prove Theorem~\ref{thm:main} we will show that $T_\eps = T_0$ for all $\eps\in (0,\eps_0]$ provided $\eps_0\in (0,1]$ is sufficiently small. 
To this end, we consider the closed ball
\begin{align*}
	B_{R,T}^{\eps,N}
	\coloneqq {
		\left\{
		(\bw,u) \in X_{T}:
		\norm{(\bw,u)}_{X_{T,\eps}} \leq R \eps^{N+\onehalf}
		\right\}
	}
\end{align*}
for some fixed integer $N\geq 3$ as in Theorem~\ref{thm:main}.

The following estimate of the nonlinear terms $(\mathbf{R}_1,R_3)$ is essential for the proof of our main result:
\begin{thm}\label{thm:perturbation}
	Let $0 < T \leq T_0$, $\eps \in (0,1]$, and let $\mathbf{R}_1, R_3$ be defined as in \eqref{eq:R1}-\eqref{eq:R3}.
	Then there is some $\eps_1\in (0,1]$ such that
	\begin{align}
		\norm{\mathbf{R}_1}_{Y_{T,\eps}^1}+\norm{R_3}_{Y_{T,\eps}^3} \leq C \eps^{2N-2}
	\end{align}
        for all $\eps \in (0,\eps_1]$ provided $(\ol{\bw}, \ol{u}) \in B_{R,T}^{\eps,N}$, 
	where $C > 0$ is a constant independent of $\eps \in (0,\eps_1]$ and $ T \in (0,T_0] $.
\end{thm}
\begin{proof}
	\emph{Estimates of $\mathbf{R}_1$, $L^2(0,T;V_\sigma')$-norm:} As before we choose $\eps_1\in (0,1]$ so small that the estimates of Sections~\ref{sec:higher-order} and \ref{sec:FullSystem} hold true.
	First, by \eqref{eq:h-w-estimate}, we know 
	\begin{align}
		& \|\pt \we^\pm\|_{L^2(0,T; L^2(\Omega))}
		+ \|\we^\pm\|_{L^2(0,T; H^2(\Omega\setminus \Gamma_t))}
		+ \|\we^\pm\|_{L^\infty(0,T; H^1(\Omega\setminus \Gamma_t))} \nonumber \\
		& \quad \leq C(T_0)\|\ol{\we}\|_{L^2(0,T;H^1)}
		\leq C \eps^{N+\onehalf}\nonumber
	\end{align}
	for $T \leq T_0$ which combined with the decomposition of $ \bw = \bw_A + \ol{\bw} $ and Lemma \ref{lem:wout} yields
	\begin{align}
	\norm{(\bw,\oeps \nabla \bw)}_{L^\infty(0,T;L^2)}
        + \norm{\oeps \pt \bw}_{L^2(Q_{T})} \leq C \eps^{N+\frac12}, \quad 
        \norm{\oeps^2 \nabla^2 \bw}_{L^2(Q_{T})} \leq C \eps^{N+1}.
		\label{eq:w-eps}
	\end{align}
	By $ L^{6/5}(\Omega)\hookrightarrow H^{-1}(\Omega)$ and $\omega_\eps \geq \eps$, we directly have
	\begin{align}
		\nonumber
		& \norm{\bw \cdot \nabla \bw}_{L^2(0,T,V_\sigma')}  \leq C \norm{\bw}_{L^\infty(0,T,L^3(\Omega))} \norm{\nabla \bw}_{L^2(0,T,L^2(\Omega))} \\
		\label{eq:wnablaw-V'}
		& \quad \leq C \norm{\bw}_{L^\infty(0,T,L^2(\Omega))}^\onehalf
		\norm{\bw}_{L^\infty(0,T,L^6(\Omega))}^\onehalf
		\norm{\nabla \bw}_{L^2(0,T,L^2(\Omega))}
		\leq C \eps^{2N-\frac12}.
	\end{align}
	We note that by Lemma \ref{lem:Lp} with $ p = 4 $ and for $j=1,2$
	\begin{align}
		\nonumber
		& \int_0^{T} \norm{\oeps^j \nabla \ol{u}}_{L^4(\Gamma_t(2\delta))}^4 \dt \\
		\nonumber
		&\quad \leq C \eps^{4j-5} \int_0^{T} \Big(\norm{\oeps^2 \nabla \ol{u}}_{L^2(\Gamma_t(2\delta))} + \norm{\oeps^2  \pr\nabla \ol{u}}_{L^2(\Gamma_t(2\delta))}\Big) \\
		\nonumber
		& \qquad \qquad \times \Big(\norm{\oeps \nabla \ol{u}}_{L^2(\Gamma_t(2\delta))}
		+ \norm{\oeps \nabla_\btau \nabla \ol{u}}_{L^2(\Gamma_t(2\delta))}
		\Big)^2
		\norm{\oeps \nabla \ol{u}}_{L^2(\Gamma_t(2\delta))} \dt \\
		&\quad \leq C \eps^{4j-5} \norm{\ol{u}}_{X_{T,\eps}^2}^4 \leq C \eps^{4(N+j)-3},
		\label{eq:omega-du-L4}
	\end{align}
        since $\nabla \oeps$ is uniformly bounded. This implies 
	\begin{align*}
		\norm{\eps \Div (\nabla \ol{u} \otimes \nabla \ol{u})}_{L^2(0,T;V_\sigma')} 
		\leq \eps^{-1}\norm{\eps \nabla \ol{u}}_{L^4(Q_{T} \cap \Gamma(2\delta))}^2
		+ \eps^{-1}\norm{\eps \nabla \ol{u}}_{L^4(Q_{T} \setminus \Gamma(2\delta))}^2
		\leq C \eps^{2N - \onehalf}.
	\end{align*}
	Moreover, by \eqref{eq:uAEstim2} we have
    \begin{equation}
        \label{eq:u_A-eps}
	\begin{aligned}
		& \eps^{\frac14} \|u_A\|_{L^\infty(Q_{T})}
        + \eps^{\frac14-\frac1p-j}\|(\oeps^ju_A, \oeps^j\nabla_\btau u_A, \omega_\eps^{1+j} \nabla u_A)\|_{L^\infty(0,T;L^p)} \\
        & \qquad + \|\omega_\eps^2 \nabla^2 u_A\|_{L^2(0,T;L^4)}
        + \eps^{-\frac14} \|\omega_\eps^2 \nabla^2 u_A\|_{L^4(0,T;L^2)}
		\leq C \eps^{N-\frac14}
	\end{aligned}
      \end{equation}
      for all $1\leq p \leq 4$ and $j=0,1$.
	Then it holds
	\begin{align*}
		& \norm{\eps \Div (\nabla u \otimes \nabla u)}_{L^2(0,T;V_\sigma')} \leq C \eps^{2N-\frac{3}{2}}.
	\end{align*}

	Furthermore, we derive
	\begin{align*}
		& \norm{\Div \Big(2 \big(\nu(c_A) + \nu'(c_A)u - \nu(c_\eps)\big) D \bv_A\Big)}_{L^2(0,T;V_\sigma')}
		\leq C \norm{u}_{L^4(Q_{T})}^2 
		\norm{D \bv_A}_{L^\infty(Q_{T})} \leq C \eps^{2N-\onehalf}, \\
		& \norm{\Div \Big(2 \big(\nu(c_A) - \nu(c_\eps)\big) D \bw\Big)}_{L^2(0,T;V_\sigma')}
		\leq 2 \norm{\nu(c_A) - \nu(c_\eps)}_{L^\infty(Q_{T})} 
		\norm{D \bw}_{L^2(Q_{T})} \leq C \eps^{2N-1}.
	\end{align*}
        \emph{Estimates of $\mathbf{R}_1$ in weighted $L^2(Q_T)$-norm:}
	The convective term is controlled as follows:
	\begin{align*}
		\norm{\oeps \bw \cdot \nabla \bw}_{L^2(Q_{T})}
		& \leq  \norm{\oeps\bw}_{L^\infty(0,T;L^6)} \norm{\nabla \bw}_{L^2(0,T;L^3)}
		\leq C \eps^{2N-\frac12}.
	\end{align*}
	By \eqref{eq:u_A-eps}, we find
	\begin{align*}
		\norm{\oeps \eps \Div (\nabla u \otimes \nabla u)}_{L^2(Q_{T})} 
		\leq \frac{C}\eps \norm{\oeps^2 \nabla^2 u}_{L^2(0,T;L^4)} \norm{\oeps \nabla u}_{L^\infty(0,T;L^4)}
		\leq C \eps^{2N-\frac{3}{2}}
	\end{align*}
        because of \eqref{eq:uAEstim2} and
        \begin{align}\nonumber
          \|\oeps \nabla \ol u\|_{L^\infty(0,T;L^4)}&\leq C \|\oeps \nabla \ol u\|_{L^\infty(0,T;L^2)}^{\frac14}\|\oeps \nabla \ol u\|_{L^\infty(0,T;H^1)}^{\frac34}\leq C\eps^{N-\frac14},\\\label{eq:L2L4D2u}
          \|\oeps^2 \nabla^2 \ol u\|_{L^2(0,T;L^4)}&\leq C \|\oeps^2 \nabla^2 \ol u\|_{L^2(0,T;L^2)}^{\frac14}\|\oeps^2 \nabla^2 \ol u\|_{L^2(0,T;H^1)}^{\frac34}\leq C\eps^{N-\frac14}.
        \end{align}
	We also remark that 
	\begin{align*}
		& \norm{\oeps \Div \Big(2 \big(\nu(c_A) - \nu(c_\eps)\big) D \bw\Big)}_{L^2(Q_{T})} \leq \norm{\nu(c_A) - \nu(c_\eps)}_{L^\infty(Q_{T})} 
		\norm{\oeps \Div D \bw}_{L^2(Q_{T})} \\
		& \qquad + \norm{\oeps \nu'(c_A) \nabla c_A - \oeps \nu'(c_\eps) \nabla c_\eps}_{L^\infty(0,T; L^4))} \norm{D \bw}_{L^2(0,T;L^4)} \leq C \eps^{2N-\frac{9}{8}}
	\end{align*}
        because of \eqref{eq:w-eps}, $\norm{D \bw}_{L^2(0,T;L^4)} \leq C(T_0) \norm{D \bw}_{L^\infty(0,T;L^2)}^{\frac{1}{4}} \norm{\nabla D \bw}_{L^2(0,T;L^2)}^{\frac{3}{4}}$ and
        \begin{align*}
          &\norm{\oeps \nu'(c_A) \nabla c_A - \oeps \nu'(c_\eps) \nabla c_\eps}_{L^\infty(0,T; L^4)}\\
          &\quad\leq \norm{\nu'(c_A)  - \nu'(c_\eps)}_{L^\infty(0,T;L^4)} \norm{\oeps \nabla c_A }_{L^\infty(Q_T)} + C \norm{\oeps \nabla (c_A -  c_\eps)}_{L^\infty(0,T; L^4)}\leq C\eps^{N-\frac14}.  
        \end{align*}
    For the remaining viscous term, by Taylor's theorem, one obtains
    \begin{align*}
        & \norm{\oeps \Div \Big(2 \big(\nu(c_A) + \nu'(c_A) u - \nu(c_\eps)\big) D \bv_A\Big)}_{L^2(Q_{T})} \\
		& \quad \leq C \norm{u \nabla u}_{L^2(Q_{T})}
		\norm{D \bv_A}_{L^\infty(Q_{T})}
        + C \norm{u^2}_{L^2(Q_{T})} \norm{\oeps \Div D \bv_A}_{L^\infty(Q_{T})} 
        \leq C \eps^{2N-\frac32}
    \end{align*}
    since $\|\eps D^2 \ve_A\|_{L^\infty(Q_T)}$ is uniformly bounded.
	Altogether we  arrive at
	\begin{align*}
		\norm{\bR_1}_{Y_{T}^1} \leq C \eps^{2N-\frac{3}{2}}.
	\end{align*}
	
	\noindent\emph{Estimates of $R_3$ in  $L^2(0,T;(\Veps)')$-norm:}
	We first recall
	\begin{align*}
		R_3
		& =-\bw \cdot \nabla u
		- \tfrac{1}{\eps^2} \left[f'(c_\eps) - f'(c_A) - f''(c_A) (\ol u + u_A)\right].
	\end{align*}
	Then similar to \eqref{eq:wnablaw-V'} with \eqref{eq:u_A-eps}, we find
	\begin{align*}
		\norm{\bw \cdot \nabla u_A}_{L^2(0,T,(\Veps)')}
		& \leq C \norm{\bw \cdot \nabla u_A}_{L^2(0,T,L^{\frac{6}{5}}(\Omega))}  \leq C \norm{\bw}_{L^2(0,T,L^3(\Omega))} \norm{\nabla u_A}_{L^\infty(0,T,L^2(\Omega))} \\
		& \leq C \norm{\bw}_{L^2(0,T,L^2(\Omega))}^\onehalf
		\norm{\bw}_{L^2(0,T,L^6(\Omega))}^\onehalf
		\norm{\nabla u_A}_{L^\infty(0,T,L^2(\Omega))}
		\leq C \eps^{2N-1},
	\end{align*}
	which further implies
	\begin{align*}
		\norm{\bw \cdot \nabla u}_{L^2(0,T;(\Veps)')}
		\leq C \eps^{2N-1}.
	\end{align*}

    The estimates for the error in $\tfrac1\eps f'(c_\eps)$ is more involved.
    By Taylor's theorem and $c_A = \theta_0 + \eps^2 \hc_2+O(\eps^3)$ in $\Gamma(2\delta)$, it holds
	\begin{align*}
		\cN &\coloneqq 
		f'(c_\eps) - f'(c_A) - f''(c_A) u
		 = \tfrac12 f'''(\theta_0) u^2 + O(\eps^2 u^2)+ O(u^3) \\
		& = \tfrac12 f'''(\theta_0) u_A^2 +O(u_A \ol u) +O(\ol u^2) + O(\eps^2 u^2)+ O(u^3) \qquad \text{in }\Gamma(\delta),
	\end{align*}
    where we have use $f'''(c_A) = f'''(\theta_0) + O(\eps^2 \hc_2)+O(\eps^3)$ and $u = u_A + \ol u$. In $\Omega \setminus \Gamma_t(\delta)$, $c_A = \pm 1+O(\eps^3)$ in $\Omega^\pm$ and $\tfrac{\cN}{\eps^2}$ can be estimated in a straight forward manner. Note that the estimates of $\ol u$ is better than $u$, we only need to control $u$-terms instead of $\ol u$-terms here.
        Since
        \begin{align*}
          &\tfrac1{\eps^2}\left(\|u_A u\|_{L^2(Q_T)}+ \eps^2\|u^2\|_{L^2(Q_T)}+\|u^3\|_{L^2(Q_T)} \right)\\
          &\quad \leq \tfrac{C}{\eps^2}\left(\|u_A\|_{L^\infty(0,T;L^4)}\|u\|_{L^4(Q_T)}+ \|u\|_{L^\infty(Q_T)}\|u\|_{L^4(Q_T)}^2\right)\leq \tfrac{C}{\eps^2}\left(\eps^{2N}+ \eps^{2N+\frac32}\right)\leq C \eps^{2N-2},
        \end{align*}
        it remains to estimate the leading contribution
	$\frac{1}{2\varepsilon^2}f'''(\theta_0)u_A^2$ in
	$L^2(0,T;(V_t^\varepsilon(2\delta))')$. In $\Gamma(\delta)$,  by the definition of $u_A$, we have
	\begin{align*}
		\frac{1}{2\varepsilon^2}f'''(\theta_0)u_A^2
		&=
		\frac{1}{2\varepsilon^4}
		f'''(\theta_0(\rho))(\theta_0'(\rho))^2h^2
		+\frac{1}{2\varepsilon} f'''(\theta_0(\rho))\partial_\rho\hc_2 h\left(\tfrac{h}\eps \theta_0'-u_A \right). 
	\end{align*}
	Set $a_0(\rho):=f'''(\theta_0(\rho))(\theta_0'(\rho))^2$ for all $\rho\in\R$.
	Since $f$ is even, $f'''$ is odd. Hence $a_0$ is odd  since $\theta_0$ is odd and $\theta_0'$ is even. Therefore
$\int_{\mathbb R}a_0(\rho)\theta_0'(\rho)\sd\rho=0$.
	Thus Lemma \ref{lem:EstimMeanValueFree} gives
	\begin{align*}
		\left\|
		\tfrac{1}{2\varepsilon^4}a_0(\rho)h^2
		\right\|_{L^2(0,T;(V_t^\varepsilon(2\delta))')}
		&\leq
		C\varepsilon^{-4}\varepsilon^{3/2}
		\|h\|_{X_{T,h}}^2
        \leq C \eps^{2N-\frac{3}{2}}.
	\end{align*}
	For the remaining terms, since $f'''(\theta_0)\partial_\rho\hc_2$ is uniformly bounded, we obtain
	\begin{align*}
		\left\|
		\frac{1}{2\varepsilon} f'''(\theta_0)\partial_\rho\hc_2\,h(\tfrac{h}\eps \theta_0'-u_A)
		\right\|_{L^2(Q_T)}
		& \leq
		C\varepsilon^{-1}
		\|h\|_{X_{T,h}}
		\|\tfrac{h}\eps \theta_0'-u_A\|_{L^\infty(0,T;L^2)} 
        \leq
		C\varepsilon^{2N-\frac12},
	\end{align*}
        Altogether we obtain
$
      \norm{R_3}_{L^2(0,T;(\Veps)')} \leq C \eps^{2N-2}.
$

\smallskip
		\noindent\emph{Estimates of $\oeps R_3$ in  $L^2(0,T;\wW^1)$-norm:}
First of all we note that it is sufficient to estimate $\|\oeps^2 \nabla R_3\|_{L^2(Q_T)}$ and $\|\tfrac{\oeps^2}{\eps} R_3\|_{L^2(Q_T)}$ since $\|\oeps R_3\|_{L^2(Q_T)}\leq\|\tfrac{\oeps^2}{\eps} R_3\|_{L^2(Q_T)}$ due to $\oeps\geq \eps$.  
	Invoking \eqref{eq:u_A-eps}, one derives
	\begin{align*}
		\norm{\tfrac{\oeps^2}\eps \bw \cdot \nabla u}_{L^2(Q_{T})}
		\leq  C \norm{\bw}_{L^2(0,T:L^4)} \norm{\tfrac{\oeps^2}\eps \nabla u}_{L^\infty(0,T;L^4)}
		\leq \eps^{2N-\frac{1}{2}},
	\end{align*}
        where we note that we have for $\tfrac{\oeps^2}\eps \nabla u$ the same estimates as for $\oeps \nabla u$ and $\normm{\bw}_{L^2(0,T;L^4)} \leq C(T_0) \normm{\bw}_{L^\infty(0,T;L^2)}^{\frac{1}{4}} \normm{\nabla \bw}_{L^2(0,T;L^2)}^{\frac{3}{4}}$.
	Moreover, using 
    \eqref{eq:w-eps},
    \eqref{eq:u_A-eps}, and \eqref{eq:L2L4D2u}, 
	we obtain
	\begin{align*}
		&\norm{\oeps^2 \nabla (\bw \cdot \nabla u)}_{L^2(Q_{T})}\\
		&\quad \leq C \left(\norm{\oeps \nabla \bw}_{L^2(0,T;L^6)} \norm{\oeps \nabla {u}}_{L^\infty(0,T;L^3)}  + \norm{\bw}_{L^\infty(0,T;L^4)} 
		\norm{\oeps^2 \nabla^2 u}_{L^2(0,T;L^4)} \right) \\
		&\quad \leq C \eps^{2N -\frac16} + C \eps^{2N-\frac12}\leq C'\eps^{2N-\frac12}.
	\end{align*}
	Invoking Lemma \ref{lem:Lp} with $ p = 4 $, we find
	\begin{align*}
&		\int_0^{T} \norm{\oeps^j\ol{u}}_{L^4(\Gamma_t(2\delta))}^4 \dt\\ 
		&\quad \leq C \eps^{-1} \int_0^{T} \norm{(\oeps^{1+j} \ol{u},\oeps^{1+j} \ptial{r} \ol{u})}_{L^2(\Gamma_t(2\delta))}\norm{\oeps^j\ol{u},\oeps^j\nabla_\btau \ol{u}}_{L^2(\Gamma_t(2\delta))}^2 \norm{\oeps^j\ol{u}}_{L^2(\Gamma_t(2\delta))} \dt  \\
		& \quad
		\leq \eps^{-1+4j} \norm{\ol{u}}_{X_{T}^2}^4,
	\end{align*}
    which yields
    \begin{align}
        \norm{\oeps^j\ol{u}}_{L^4(Q_{T})}
        \leq C \eps^{N+\frac14+j}, \quad \norm{\oeps^ju}_{L^4(Q_{T})}
        \leq C \eps^{N-\frac14+j} \qquad \text{for }j=0,1.
		\label{eq:u-L4}
    \end{align}
    Since $\|u\|_{L^\infty(Q_T)}\leq C$, we can use Taylor's theorem to estimate $|\mathcal{N}|\leq C |u|^2$ and obtain
    \begin{align*}
      \left\|\tfrac{\oeps^2}{\eps^3} \mathcal{N}\right\|_{L^2(Q_T)}\leq C\eps^{-3} \|\oeps^2 u\|_{L^4(Q_T)}^2\leq C\eps^{2N-\frac32}. 
    \end{align*}
	Direct computation with Taylor's theorem yields
	\begin{align*}
		& \abs{\nabla \left[f'(c_\eps) - f'(c_A) - f''(c_A) u\right]} \\
		&\quad = \abs{f''(c_\eps) \nabla c_\eps 
			- f''(c_A) \nabla c_A
			- f'''(c_A) u \nabla c_A
			- f''(c_A) \nabla u} \\
		& \quad\leq \abs{[f''(c_\eps) - f''(c_A)] \nabla u}
		+ \abs{\left[f''(c_\eps) - f''(c_A) - f'''(c_A) u\right] \nabla c_A} \\
		& \quad \leq C\abs{ u \nabla u}
		+ C \abs{u^2 \nabla c_A}.
	\end{align*}
	Then by \eqref{eq:omega-du-L4}, \eqref{eq:u_A-eps}, and \eqref{eq:u-L4}, we have
	\begin{align*}
		\norm{\tfrac{\oeps^2}{\eps^2}\nabla \cN}_{L^2(Q_{T})}
		& \leq  C \norm{\tfrac{\oeps}\eps u}_{L^4(Q_{T})} \norm{\tfrac{\oeps}{\eps} \nabla u}_{L^4(Q_{T})} + C \norm{\tfrac{\oeps}{\eps} u}_{L^4(Q_{T})}^2 \norm{\nabla c_A}_{L^\infty(Q_{T})} \\
		& \leq C \eps^{N - \frac{1}{4}} \eps^{N - \frac{5}{4}}
		+ C \eps^{2N - \frac{5}{4}} \eps^{N - \frac{3}{4}} 
		\leq C' \eps^{2N-\frac{3}{2}}. 
	\end{align*}
	In summary, we have derived
$		\norm{R_3}_{Y_{T,\eps}^3} \leq C \eps^{2N-2}$.
	This finishes the proof.
\end{proof}

\noindent
\begin{proof*}{of Theorem \ref{thm:main}}
   With all the estimates above at hand, we are ready to derive the convergence result. As before let $T_\eps\in (0,T_0]$, $\eps\in (0,1]$, be maximal such that \eqref{eq:NSAC1} has a smooth solution $\ve_\eps\colon \ol{\Omega}\times [0,T_\eps)\to \R^d$, $(p_\eps,c_\eps) \colon \ol{\Omega}\times [0,T_\eps)\to \R^2$ satisfying \eqref{eq:AprioriBound}, as before and let $\eps_1\in (0,1]$ be such that the statements of Theorem~\ref{thm:ApproximateSolution} and Theorem~\ref{thm:perturbation} hold true. Then 
   we have
  \begin{equation*}
    \norm{(\br_1,r_2,r_3)}_{Y_{T_\eps,,\eps}}\leq C\left( \eps^{2N-2} + C \eps^{N+1}\right),
  \end{equation*}
  where $C$ is independent of $T_\eps\in (0,T_0]$ and $\eps\in (0,\eps_1]$.
  Since $N\geq 3$, we can choose $\eps_0\in (0,\eps_1]$ so small that
  \begin{equation*}
    \norm{(\br_1,r_2,r_3)}_{Y_{T_\eps,\eps}}\leq R_0\eps^{N+\frac12} \qquad \text{for all } \eps \in (0,\eps_0].
  \end{equation*}
Now \eqref{eq:MainUbarEstim-recall} and \eqref{initial assumption} yield
\begin{equation*}
  \norm{(\ol \bw,\ol{u})}_{X_{T_\eps,\eps}} \leq 2C(T_0) R_0 \eps^{N+\frac12} < R \eps^{N+\frac12}\qquad \text{for all }\eps \in (0,\eps_0].
\end{equation*}
Since $T_\eps \in (0,T_0]$ is chosen maximal such that \eqref{eq:AprioriBound} holds, we conclude $T_\eps= T_0$ for all $\eps \in (0,\eps_0]$ since otherwise, by standard results on local existence of smooth solutions, we could extend the smooth solution $\ve_\eps\colon \ol{\Omega}\times [0,T_\eps)\to \R^d$, $(p_\eps,c_\eps) \colon \ol{\Omega}\times [0,T_\eps)\to \R^2$ on a time interval $[0,T'_\eps)$, still satisfying \eqref{eq:AprioriBound}, for some $T'_\eps\in (T_\eps,T_0]$. This would contradict the maximality of $T_\eps$.
In particular, \eqref{eq:convVelocityb} holds true for every $R_1\geq R$. 
To show  \eqref{eq:convc} we use that $u=u_A+\ol u$ and \eqref{eq:uAEstim2} to obtain
\begin{align*}
	& \|(u, \nabla_\btau u, \omega_\eps \nabla u)\|_{L^\infty(0,T_0;L^2)}
    + \|(\omega_\eps\partial_t u,\oeps^2 \nabla^2 u)\|_{L^2(Q_{T_0})}+ \|\omega_\eps \nabla_{\btau}\nabla u\|_{L^2(\Gamma(2\delta))} \\
    &\quad\leq (1 + C(T_0) \eps^{-\frac12})R\eps^{N+\frac12} \leq R_1 \eps^{N}\qquad \text{for all }\eps \in (0,\eps_0]
\end{align*}
if $R_1 \geq (1+C(T_0))R$. 
Similarly, we derive from $\bw = \bw_A + \ol \bw$ and Lemma \ref{lem:wout} that
\begin{align*}
	& \|(\bw, \oeps \nabla \bw)\|_{L^\infty(0,T_0;L^2)}
    + \|\oeps \partial_t \bw\|_{L^2(Q_{T_0})} 
    \leq (1 + C(T_0) \eps^{\frac12})R\eps^{N+\frac12} \leq R_1 \eps^{N+\frac12}, \\
    & \|\oeps^2 \nabla^2 \bw\|_{L^2(Q_{T_0})} 
    \leq (\eps + C(T_0) \eps^{\frac12})R\eps^{N+\frac12} \leq R_1 \eps^{N+1},
\end{align*}
for all $\eps \in (0,\eps_0]$, which proves \eqref{eq:convVelocityb}.
Finally, it follows from \eqref{eq:uAEstim2} and
\eqref{eq:L-infty-Estim} that 
\begin{align*}
    \norm{u}_{L^\infty(Q_{T_0})} \leq  \norm{u_A}_{L^\infty(Q_{T_0})} +\norm{\ol u}_{L^\infty(Q_{T_0})}
    \leq (C(T_0)\sqrt{\eps} +C_{\rm GN} ) R \eps^{N-1}\leq (C(T_0) +C_{\rm GN} ) R \eps^{N-1},
\end{align*}
where $C_{\rm GN}$ is the constant given in \eqref{eq:L-infty-Estim}. This shows \eqref{eq:convc-Linfty} for a sufficiently large $R_1>0$.

The claimed convergence for $\ve_A$ is referred to Remark \ref{rem:vAC1bounded} and the convergence of $c_A$ follows from the construction.
This completes the proof of Theorem \ref{thm:main}. 
\end{proof*}

%%%%%%%%%%%%%%%%%%%%%%%%%%%%%%%%%%%%%%%	
\section*{Statements and Declarations}
\subsection*{Conflict of interest} 
The authors have no competing interests to declare that are relevant to the content of this article.

\subsection*{Data availability statement}
Data sharing is not applicable to this article as no datasets were generated during the current study.

\subsection*{Funding} 
M. Fei is partially supported by the NSF of China under Grant No.~12271004 and 12471222, and the NSF of Anhui Province of China under Grant No.2308085J10. Y. Liu is partially supported by the NSF of Jiangsu Province of China under Grant No.~BK20240572, the NSF of Jiangsu Higher Education Institutions of China under Grant No.~24KJB110020, and the China Postdoctoral Science Foundation under Grant No.~2025M773078. 

\subsection*{Acknowledgments}
This work was partially done when Y.L. was visiting Fakult\"at f\"ur Mathematik, Universit\"at Regensburg, and School of Mathematics and Statistics, Anhui Normal University, whose hospitality is gratefully acknowledged. The first three authors also thank Tianyuan Mathematics Research Center for the research visit when the work was partially conducted.

%\bibliographystyle{abbrv}
%\bibliography{Bibliography}

\def\ocirc#1{\ifmmode\setbox0=\hbox{$#1$}\dimen0=\ht0 \advance\dimen0
  by1pt\rlap{\hbox to\wd0{\hss\raise\dimen0
  \hbox{\hskip.2em$\scriptscriptstyle\circ$}\hss}}#1\else {\accent"17 #1}\fi}

\bigskip

\noindent
{\it
	(H. Abels) Fakult\"at f\"ur Mathematik,
	Universit\"at Regensburg,
	93040 Regensburg,
	Germany}\\
{\it E-mail address: {\sf \href{mailto:helmut.abels@mathematik.uni-regensburg.de}{helmut.abels@mathematik.uni-regensburg.de}} }\\[1ex]
{\it
	(M. Fei) School of  Mathematics and Statistics, Anhui Normal University, Wuhu 241002, China}\\
{\it E-mail address: {\sf \href{mailto:mwfei@ahnu.edu.cn}{mwfei@ahnu.edu.cn}} }\\[1ex]
{\it
	(Y. Liu) School of Mathematical Sciences, and Ministry of Education Key Laboratory of NSLSCS, and Key Laboratory of Jiangsu Provincial Universities of FDMTA, Nanjing Normal University, Nanjing 210023, China}\\
{\it E-mail address: {\sf \href{mailto:ydliu@njnu.edu.cn}{ydliu@njnu.edu.cn}} }\\[1ex]
{\it
	(M. Moser) Fakult\"at f\"ur Mathematik,
	Universit\"at Regensburg,
	93040 Regensburg,
	Germany}\\
{\it E-mail address: {\sf \href{mailto:maximilian1.moser@mathematik.uni-regensburg.de}{maximilian1.moser@mathematik.uni-regensburg.de}} }\\[1ex]

\end{document}